\documentclass[11pt]{amsart}

\usepackage{amsmath, amsthm, amssymb}

\usepackage[dvipsnames]{xcolor}
\definecolor{indigo}{rgb}{0.29, 0.0, 0.51} 
\usepackage[colorlinks, urlcolor=indigo, linkcolor=indigo, citecolor=indigo]{hyperref}
\definecolor{darkred}{RGB}{210, 0, 0}
\definecolor{darkblue}{RGB}{0, 10, 210}
\definecolor{darkgreen}{RGB}{0, 180, 0}
\usepackage[dvipsnames]{xcolor}

\usepackage{lmodern, enumerate}
\usepackage{palatino}
\usepackage[noabbrev,capitalize,nameinlink]{cleveref}
\usepackage[bigdelims,vvarbb]{newpxmath}    
\usepackage[scaled=0.95]{inconsolata}       
\usepackage[T1]{fontenc}                    

\usepackage{comment}
\usepackage{soul}

\usepackage{tikz-cd}

\usepackage[abs]{overpic}		  

\usepackage[hcentering, top = 1.5in, total={5.9in, 8.2in}]{geometry}  

\theoremstyle{plain}
\newtheorem{theorem}{Theorem}
\newtheorem{corollary}[theorem]{Corollary}
\newtheorem{proposition}[theorem]{Proposition}
\newtheorem{lemma}[theorem]{Lemma}
\newtheorem{question}[theorem]{Question}
\newtheorem{conjecture}[theorem]{Conjecture}

\newtheorem{mainthm}{Theorem}

\theoremstyle{definition}
\newtheorem{definition}[theorem]{Definition}
\newtheorem{definitiontheorem}[theorem]{Definition-Theorem}

\theoremstyle{remark}
\newtheorem{remark}[theorem]{Remark}

\numberwithin{theorem}{section}

\newcommand{\R}{\mathbb{R}}           
\newcommand{\Z}{\mathbb{Z}}           
\newcommand{\C}{\mathbb{C}}           

\makeatletter
\newcommand*\bigcdot{\mathpalette\bigcdot@{0.6}}
\newcommand*\bigcdot@[2]{\mathbin{\vcenter{\hbox{\scalebox{#2}{$\m@th#1\bullet$}}}}}
\makeatother

\DeclareMathOperator{\std}{st} 

\DeclareMathOperator{\reg}{\operatorname{reg}}

\DeclareMathOperator{\exot}{\operatorname{exot}}

\DeclareMathOperator{\Crit}{\operatorname{Crit}}

\DeclareMathOperator{\Rea}{\operatorname{Re}}

\DeclareMathOperator{\concave}{\operatorname{concave}}
\DeclareMathOperator{\Weinstein}{\operatorname{Weinstein}}

\newcommand{\FMod}{\operatorname{FMod}}

\newcommand{\Stein}{\operatorname{Stein}}
\newcommand{\ind}{\operatorname{ind}}

\newcommand{\crit}{\operatorname{crit}}
\newcommand{\fib}{\operatorname{fib}}

\newcommand{\PSH}{\operatorname{PSH}}

\newcommand{\bC}{\mathbb{C}}
\newcommand{\bD}{\mathbb{D}}

\newcommand{\bR}{\mathbb{R}}

\newcommand{\cF}{\mathcal{F}}

\newcommand{\cH}{\mathcal{H}}

\newcommand{\cJ}{\mathcal{J}}

\newcommand{\cM}{\mathcal{M}}
\newcommand{\cN}{\mathcal{N}}

\newcommand{\cU}{\mathcal{U}}

\newcommand{\cW}{\mathcal{W}}

\usepackage{thmtools} 
\usepackage{thm-restate}

\begin{document}

\title{Nearly Lefschetz fibrations, Stein structures, and regular Lagrangians} 

\author{Joseph Breen}

\author{Agniva Roy}

\author{Luya Wang}

\address{University of Alabama \\ Tuscaloosa, AL}
\email{jjbreen@ua.edu}

\address{Department of Mathematics \\ Boston College \\ Chestnut Hill, MA}
\email{agniva.roy@bc.edu}

\address{Institute for Advanced Study \\ Princeton, NJ}
\email{luyawang@ias.edu}

\thanks{JB was partially supported by NSF Grant DMS-2038103 and an AMS-Simons Travel Grant. AR was partially supported by an AMS-Simons Travel Grant. LW was supported by NSF Grant DMS-2303437, DMS-2641772, and IAS Giorgio and Elena Petronio Fellow II Fund.}

\begin{abstract}
We study Stein structures on nearly Lefschetz fibrations, with applications to symplectic and Lagrangian submanifolds in Weinstein domains. We prove that, up to Weinstein deformation equivalence, all Lagrangian disks with Legendrian boundary in $4$-dimensional Weinstein domains are regular in the sense of Eliashberg-Ganatra-Lazarev. This settles a Weinstein analogue of the nearby Lagrangian conjecture for two-dimensional Lagrangian disks, and resolves part of a Lagrangian analogue of the Slice-Ribbon conjecture. Along the way, we show that positive allowable nearly Lefschetz fibrations are supported by canonical Stein structures and prove a corresponding quasiflexibility result in the planar case. This yields a generalization of work of Boileau-Orevkov which may be of independent interest. Finally, we give an explicit algorithm producing a Weinstein Kirby diagram from the nearly Lefschetz fibration structure of a multisection complement. 
\end{abstract}

\maketitle
\tableofcontents

\section{Introduction}\label{sec: intro}
Lefschetz fibrations and pencils have played a central role in symplectic topology since the foundational work of Thurston, Donaldson, and Gompf \cite{Thurston,donaldson1999lefschetz,gompf_topological_symplectic}. The twenty-first century has seen the emergence of a rich theory of Stein structures on Lefschetz fibrations, deepening the bond between Stein geometry, contact topology, and mapping class factorizations \cite{LoiPiergallini2001Stein,Akbulut_Ozbagci,Wendl_strongly,p-vhm,GirouxPardon2017Lefschetz}. This article develops the theory further by studying Stein structures on \emph{nearly} Lefschetz fibrations, which were first hinted at in \cite{spinal_II} and defined in \cite{spinal_exotic}. We give several applications to symplectic and Lagrangian submanifolds.

\subsection{Stein structures on surface exteriors}
\label{subsec: intro_stein_exterior}
Given a ribbon surface $S$ embedded in $B^4$, its exterior $B^4 \setminus \nu(S)$ is a compact orientable $4$-manifold admitting a handle decomposition with handles of index at most $2$; here $\nu(S)$ denotes a tubular neighborhood of $S$. One can further ask if the exterior admits a Stein structure. This is not always the case: for example, the standard annulus filling of a $2$-component unlink has an exterior diffeomorphic to $(S^1 \times B^3) \natural (S^2 \times D^2)$, which contains a homologically essential sphere of self-intersection number $0$, violating the adjunction inequality \cite{Lisca1997Tight}. In general, there is no known characterization of which ribbon surfaces have exteriors that admit Stein structures, though we highlight two known cases of interest:
\begin{enumerate}
\item If $\Sigma$ is a complex curve in $B^4 \subset \C^2$, then $B^4 \setminus \nu(\Sigma)$ admits a Stein structure. In fact, the exterior is a Stein domain using the ambient complex structure (as opposed to simply admitting some Stein structure, a priori unrelated to the complex structure on $\C^2$);
see \cite{simha}, or \cite[\S 5, Theorem 5]{grauert1979steinspaces}. More generally, work of Boileau-Orevkov \cite{boileau-orevkov} and Rudolph \cite{rudolph2004algebraic} implies that any symplectic surface in $B^4$ is ambiently diffeomorphic to a complex curve, hence exteriors of symplectic surfaces in $B^4$ also admit Stein structures. (In this case, however, the resulting Stein structure need not be compatible with the restriction of the ambient symplectic form.)\label{part:stein}

\item If $L$ is a regular Lagrangian disk, then $B^4 \setminus \nu(L)$ admits the structure of a Weinstein (hence Stein) domain \cite{EliashbergGanatraLazarev2018Regular}. In fact, it inherits a Weinstein structure directly from an ambient Weinstein structure on $B^4$: regularity implies that the Weinstein structure on $B^4$ can be homotoped to witness $L$ as a cocore of a $2$-handle, and the exterior is then obtained by removing the handle, leaving a Weinstein domain.\label{part:weinstein} 
\end{enumerate}

Generalizing \eqref{part:stein} further, properly embedded complex curves in Stein surfaces have Stein domain exteriors. However, it is unknown whether every symplectic surface in a Stein domain is ambiently diffeomorphic to a holomorphic curve. Hayden \cite{hayden2021quasipositive} has generalized some aspects of Boileau-Orevkov and Rudolph to Stein domains; namely, that holomorphic curves bound quasipositive links, and any quasipositive link bounds a symplectic surface in some Stein filling. A weaker exterior version of this problem is: 
\begin{question}\label{question:symplectic-exterior-stein}
Let $(W, J)$ be a $4$-dimensional Stein domain, and let $\Sigma\subset W$ be a properly embedded symplectic surface with nonempty transverse boundary in $\partial W$. Does $W \setminus \nu(\Sigma)$ admit the structure of a Stein domain?
\end{question}

\noindent A corollary of our work is an affirmative answer in two cases.

\begin{corollary}
The answer to Question \ref{question:symplectic-exterior-stein} is yes if either 
\begin{enumerate}
    \item $\Sigma$ is a symplectic push-off of a Lagrangian disk with Legendrian boundary, or
    \item $(W,J)$ is supported by a planar nearly Lefschetz fibration.
\end{enumerate}
\end{corollary}

\noindent We will expand on the role of nearly Lefschetz fibrations below. 

Let us recall the definition of regularity, first studied by Eliashberg-Ganatra-Lazarev \cite{EliashbergGanatraLazarev2018Regular}, more precisely. A Lagrangian cobordism $L$ with Legendrian boundary in a Weinstein cobordism $(W,\lambda,\phi)$ is \emph{regular} if the Weinstein structure can be homotoped to $(\lambda',\phi')$, keeping $L$ Lagrangian throughout, so that the Liouville vector field $X_{\lambda'}$ is tangent to $L$. While non-regular caps and concordances with nonempty negative end have been constructed \cite{Murphy2013Exact,Lin2016Caps,DimitroglouRizellGolovko2024NonRegular,DimitroglouRizellGolovko2025NonRegular}, it is an open question as to whether there exists a non-regular exact Lagrangian filling in a Weinstein domain. 

\begin{question}\label{question:lagrangian-weinstein}
Let $(W^{2n}, \lambda, \phi)$ be a Weinstein domain of any dimension, and let $L\subset W$ be an exact Lagrangian filling. 
\begin{enumerate}
    \item Is $L$ regular? 
    \item When does $W \setminus \nu(L)$ admit the structure of a Weinstein domain?
\end{enumerate}
\end{question}

\noindent If $L$ is regular, $W\setminus\nu(L)$ naturally inherits the structure of a Weinstein \emph{cobordism}; the case of a regular disk is special in that the disk can be pushed "to the top," so that the exterior becomes a Weinstein domain \cite[Proposition 2.3]{EliashbergGanatraLazarev2018Regular}. Since zero-sections and cotangent fibers in cotangent bundles are examples of regular Lagrangians, Question \ref{question:lagrangian-weinstein}(1) can be understood as the Weinstein analogue of the nearby Lagrangian and nearby fiber conjectures. These were studied by \cite{Eliashberg_Polterovich, Hind2012Stein, Abouzaid2012Nearby,Kragh2013Nearby,DimitroglouRizellGoodmanIvrii2016, AbouzaidKragh2018Nearby,Ekholm_Smith, Cote_DR, asplund2025homotopyrigiditynearbylagrangian}.

Our first theorem is an affirmative answer to  Question \ref{question:lagrangian-weinstein}(1) for disks in dimension $4$, up to Weinstein deformation equivalence. Recall that two Weinstein domains $(W_i, \lambda_i, \phi_i)$, $i\in\{1,2\}$, are \emph{Weinstein deformation equivalent} if there is a diffeomorphism $\psi: W_1 \to W_2$ such that $(\lambda_1, \phi_1)$ and $(\psi^*\lambda_2, \psi^*\phi_2)$ are Weinstein homotopic.

\begin{mainthm}[Lagrangian disks are regular]
\label{thm: regular_lag_disk}
Let $(W^4, \lambda, \phi)$ be a $4$-dimensional Weinstein domain and $L\subset W$ a Lagrangian disk with Legendrian boundary. Then $L$ is a cocore, hence regular, for a Weinstein structure $(\lambda', \phi')$ on $W$ which is Weinstein deformation equivalent to $(\lambda, \phi)$. 
\end{mainthm}
\noindent 

\begin{remark}
    In general, it is open whether the space of Weinstein structures compatible with a given symplectic form (or symplectic deformation equivalence class) is connected; see \cite[Problems 1.4, 1.5]{eliashberg2018revisited}. In the case of $(B^4, \omega_{\std})$, this is related to the question of whether $\mathrm{Diff}_+(B^4)$ is connected, while it is known that Weinstein fillings of $(S^3,\xi_{\std})$ are unique up to Weinstein deformation equivalence \cite[\S 16.2]{Cieliebak_Eliashberg}.
\end{remark}

For Question \ref{question:lagrangian-weinstein}(2), we give an answer for a specific class of regular Lagrangians.

\begin{mainthm}
\label{mainthm: regular_complement}
Let $(W^{2n}, \lambda, \phi)$ be a Weinstein domain and $L\subset W$ a regular Lagrangian filling built from $0$-handles and $1$-handles.
Then $W \setminus \nu(L)$ admits the structure of a Weinstein domain.
\end{mainthm}

\noindent In particular, this is an affirmative answer to Question \ref{question:lagrangian-weinstein}(2) for regular Lagrangian fillings in dimension $2n=4$.

\begin{remark}
There are basic homological obstructions to generalizing the handle index assumptions in Theorem \ref{mainthm: regular_complement} further. For example, if $W = T^*(S^2 \times [-1,1])$ and $L^3\subset W^6$ is the $0$-section, then $H_4(W \setminus \nu(L); \Z) \neq 0$; thanks to the presence of supercritical topology, $W\setminus\nu(L)$ does not admit the structure of a Weinstein domain.
\end{remark}

Before discussing the technology developed to prove these theorems, we describe some consequences. Theorem \ref{thm: regular_lag_disk} resolves one rung of a Lagrangian version of the Slice-Ribbon conjecture (see \cite{breen2026derivativelinkscontacttopology}). Briefly, recall that the classical Slice-Ribbon conjecture asks whether every slice knot in $S^3$ bounds a ribbon disk. One refinement of interest is \emph{handle-ribbonness}, which requires a knot to bound a slice disk whose exterior admits a $0$-, $1$-, $2$-handle decomposition. The following implications are known, and each converse implication is open: 
\begin{equation}\label{eq:SR}
\{\text{ribbon}\} \Rightarrow \{\text{handle-ribbon}\} \Rightarrow \{\text{slice}\}.  
\end{equation}
By the \emph{Lagrangian Slice-Ribbon conjecture} we mean the analogous chain of implications for Legendrian knots in $(S^3, \xi_{\mathrm{st}})$:
\begin{equation}\label{eq:SSR}
\{\text{decomposably slice}\} \Rightarrow \{\text{regularly slice}\} \Rightarrow \{\text{Lagrangian slice}\}.     
\end{equation}
A Legendrian knot is \emph{decomposably} (resp.\ \emph{regularly}) \emph{slice} if it bounds a decomposable (resp.\ regular) Lagrangian disk in $(B^4, \omega_{\mathrm{st}})$. The first implication follows from \cite{Conway_Etnyre_Tosun}, and prior to this work, both converse implications were open; we refer to \cite[\S 2]{breen2024regularlysliceimpliesoncestably} for a more thorough survey and comparison between \eqref{eq:SR} and \eqref{eq:SSR}, as well as the introduction of \cite{breen2026derivativelinkscontacttopology}. Theorem \ref{thm: regular_lag_disk} confirms that every Lagrangian slice knot is regularly slice. 

A further corollary is that the notions of Stein fillability and strong symplectic fillability coincide for contact-$(+1)$ surgeries.

\begin{corollary}
Let $\Lambda \subset (S^3, \xi_{\mathrm{st}})$ be a Legendrian knot, and let $(Y, \xi)$ denote contact-$(+1)$ surgery along $\Lambda$. Then $(Y, \xi)$ is Stein fillable if and only if it is strongly fillable.     
\end{corollary}

\noindent Indeed, by \cite{Conway_Etnyre_Tosun}, a contact-$(+1)$ surgery is strongly fillable if and only if the knot bounds a Lagrangian disk in $(B^4, \omega_{\std})$. By Theorem \ref{thm: regular_lag_disk}, such a disk is regular up to Weinstein deformation equivalence, hence the contact-$(+1)$ surgery, witnessed by removing the disk, is Stein fillable.

Since a contact-$(+1)$ surgery along a Lagrangian slice knot $\Lambda \subset (S^3, \xi_{\mathrm{st}})$ arises by removing a Lagrangian disk, the result sits as a contact-type hypersurface in $(\R^4, \omega_{\mathrm{st}})$. We more generally conjecture: 

\begin{conjecture}\label{conj:hypersurface}
Every contact-type hypersurface in $(\R^4, \omega_{\mathrm{st}})$ is Stein fillable. 
\end{conjecture}

\noindent We highlight some relevant questions that place the conjecture in context. First, motivated by conjectures of Chen \cite{chen2018splitting} and Gompf \cite{gompf2013smoothstein}, Mark-Tosun \cite[Question 1.5]{mark2022contacttype} have asked whether any rational homology sphere besides $S^3$ embeds as a contact-type hypersurface in $(\R^4, \omega_{\mathrm{st}})$, with their work giving evidence to a negative answer. In later work \cite[Question 1.9]{mark2025weinsteindomainssymplecticmanifolds}, they ask if the convex side of any contact-type hypersurface in $\mathbb{C}P^2$ is necessarily a Weinstein domain.

\subsection{Nearly Lefschetz fibrations and quasiflexibility}
The proof of Theorem \ref{thm: regular_lag_disk} involves a new construction of Stein structures on \emph{nearly Lefschetz fibrations}, which is of independent interest. Briefly, a nearly Lefschetz fibration is a generalization of a Lefschetz fibration with an additional class of smooth fibers, called \emph{exotic fibers}, with different topology than the regular fibers. While the monodromy around a typical Lefschetz singularity is a Dehn twist, the monodromy around an exotic fiber is modeled after a boundary interchange map as depicted in Figure \ref{fig:interchange}. Nearly Lefschetz fibrations were first formally defined in \cite{spinal_exotic}, drawing on ideas from \cite{baykur2016multisections,spinal_I,spinal_II}; their boundaries naturally carry spinal open book decompositions. Definitions are reviewed in \S \ref{sec: PANLF}. 

\begin{figure}[ht]
        \centering
        \includegraphics[width=0.9\linewidth]{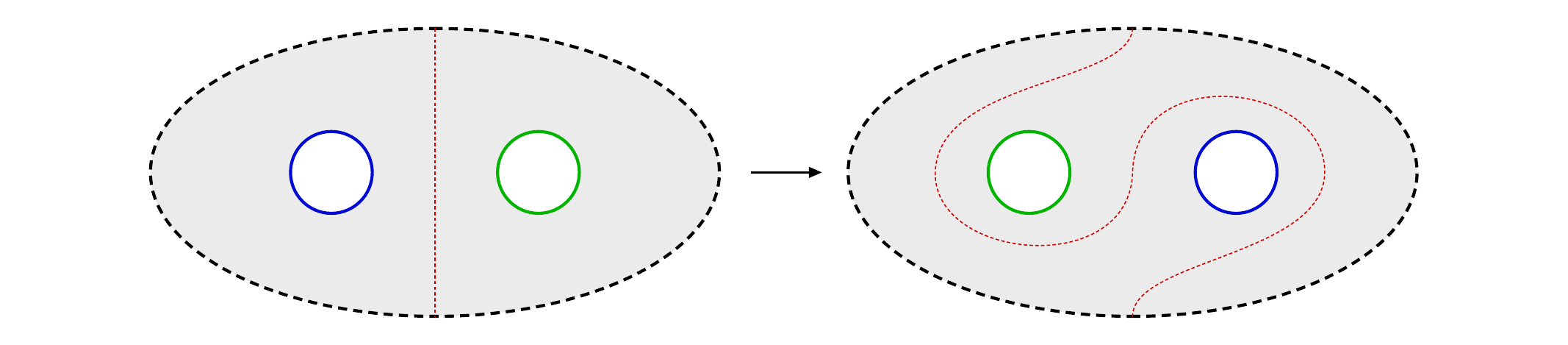}
        \caption{A boundary interchange map. Each boundary component is rotated clockwise by $\pi$, which is not shown in the picture.}
        \label{fig:interchange}
    \end{figure}

The fact that a positive allowable Lefschetz fibration (PALF) over the disk supports a Stein structure is well-known \cite{gompf1999manifolds,LoiPiergallini2001Stein,Akbulut2001Lefschetz}.  
This was generalized to PALFs over any compact orientable surface with boundary by \cite{baykurvhm,spinal_I}. Our interest in generalizing even further stems from the fact that the complement of a Lefschetz fibration \emph{multisection}, i.e.\ a surface along which the fibration map restricts to a simple branched cover, is naturally a nearly Lefschetz fibration. 
We associate a canonical Stein structure to a \emph{positive allowable nearly Lefschetz fibration (PANLF)}, where the monodromy is given by a positive factorization on essential curves and boundary interchanges. This generalizes the Stein existence results above.

\begin{mainthm}[Stein structures on PANLFs]
\label{thm: stein_on_panlf}
    Let $E$ be a positive allowable nearly Lefschetz fibration with projection map $\Pi: E \to B$. Then there exists a nonempty and contractible space of Stein structures supporting $\Pi$.
\end{mainthm}

Two remarks are in order. First, there is an equivalence between \emph{almost} Stein structures and Stein structures via the equivalence between Weinstein and Stein structures by \cite{Cieliebak_Eliashberg}; more details will be given in \S\ref{subsec:almost_stein_review}. To prove Theorem \ref{thm: stein_on_panlf}, we thus analyze the space of supporting almost Stein structures associated to a projection map, a notion defined in Definitions \ref{def: supported_almost_Stein} and \ref{def: J_Stein}. 

Second, one naturally wants to talk about Stein structures and Stein homotopy without having to specify the projection map $\Pi$ on the nearly Lefschetz fibration. In the case when the fibers are of genus zero, this is possible via the following "quasiflexibility" theorem.

\begin{mainthm}[Quasiflexibility of planar PANLFs]
\label{thm: quasiflexibility}
    Let $\Pi: E\to B$ be a nearly Lefschetz fibration with planar fibers. Suppose $J_0$ and $J_1$ are two Stein structures with associated symplectic structures $\omega_0$ and $\omega_1$. Suppose $J_0$ is supported by $\Pi$. Then $J_0$ and $J_1$ are Stein homotopic if and only if $\omega_0$ and $\omega_1$ are homotopic through symplectic structures convex at the boundary. 
\end{mainthm}
\noindent In other words, given a planar PANLF, there is a canonical supporting Stein structure (regardless of the projection map).

As an application of Theorem \ref{thm: quasiflexibility}, we prove the following generalization of the work of Boileau-Orevkov \cite{boileau-orevkov} on characterizing properly embedded symplectic surfaces in $D^2 \times D^2$. We extend their characterization to $4$-dimensional Weinstein domains which are products of surfaces.

\begin{mainthm}
\label{thm: planarNLF_multisection}
Let $B$ and $F$ be compact surfaces with nonempty boundaries, and assume $F$ is planar. In the product $4$-dimensional Weinstein domain $F \times B$, any properly embedded symplectic surface with braided boundary in the vertical boundary $F \times \partial B$ is symplectically isotopic to either a fiber $F$, or a positive multisection of the projection map to $B$. 
\end{mainthm}

The proof of Theorem \ref{thm: planarNLF_multisection} is based on a similar strategy to part of the proof of Theorem \ref{thm: regular_lag_disk}. In the latter, we go from an almost complex structure making the symplectic surface pseudoholomorphic to an almost complex structure making the fiber pseudoholomorphic; in the former, we go the opposite way. A version of Theorem \ref{thm: planarNLF_multisection} is also claimed in the work-in-progress by \cite{hedden_quasipositive}.

\subsection{From PANLFs to handle diagrams} Finally, it is natural to ask how the Stein structure supported by a PANLF can be described in other standard ways, for example by a Weinstein Kirby diagram or a supporting PALF. In \S\ref{sec:algo} we describe an algorithm for converting the PANLF on a multisection complement in $B^4$ to a Weinstein Kirby diagram directly from a (monotonic) quasipositive braid presentation of the multisection, and we compute several examples; Figure \ref{fig:intro820} depicts a sample output. Once a handle diagram is obtained, there are standard arguments for obtaining a supporting PALF, and we give an example of this as well. This gives an explicit global algorithm to convert spinal open books to ordinary open books in these settings.

In general, as a result of Theorem \ref{thm: stein_on_panlf}, we know abstractly that we can always convert a spinal open book on the boundary of a PANLF to an ordinary open book. Finding an explicit (and perhaps more local) algorithm in general would be very interesting. See \S \ref{subsec:questions} for further questions.

\begin{figure}[ht]
	\centering
	\begin{overpic}[width=\linewidth]{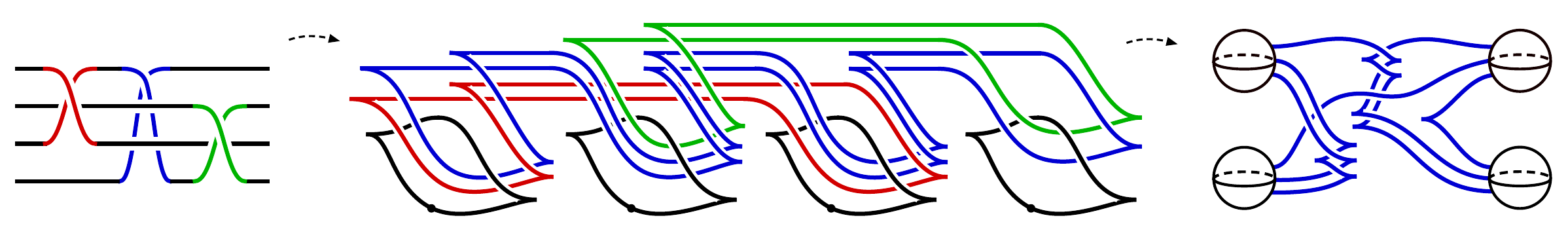}
        \put(180,63){\footnotesize \textcolor{darkred}{$(-1)$}}
        \put(200,63){\footnotesize \textcolor{darkblue}{$(-1)$}}
        \put(220,63){\footnotesize \textcolor{darkgreen}{$(-1)$}}
        \put(365,60){\footnotesize \textcolor{darkblue}{$(-1)$}}
	\end{overpic}
	\caption{A sample output from the algorithm in \S\ref{sec:algo}; c.f.\ Figure \ref{fig:m820}.}
	\label{fig:intro820}
\end{figure}

\subsection*{Organization} In \S\ref{sec: PANLF} we review the definitions of spinal open books, nearly Lefschetz fibrations, and almost Stein structures. In \S\ref{sec:PANLF_stein} and \S\ref{sec: quasiflex}, we establish the Stein existence and quasiflexibility results, Theorems \ref{thm: stein_on_panlf} and \ref{thm: quasiflexibility}. In \S\ref{subsec:panlf_weinstein}, we construct a specific Weinstein structure on a PANLF, and in \S\ref{sec: proof_of_main} we prove Theorems \ref{thm: regular_lag_disk} and \ref{mainthm: regular_complement}. The proof of Theorem \ref{thm: planarNLF_multisection} is contained in \S\ref{sec: planarNLF_multisection}. Finally, the algorithm for converting a multisection complement PANLF into a handle diagram is presented in \S\ref{sec:algo}. We note that \S\ref{subsec:panlf_weinstein} and \S\ref{sec:algo}, which contain the concrete descriptions of Weinstein structures, can be read independently on a first pass before the technical material in \S\ref{sec:PANLF_stein} and \S\ref{sec: quasiflex}.

\subsection*{Acknowledgments}
The authors started thinking about ideas in this paper during the 2023 BIRS workshop on ``Interactions between Symplectic and Holomorphic Convexity in 4 Dimensions'' and thank its organizers. The authors thank helpful conversations from Georgios Dimitroglou Rizell, Yakov Eliashberg, John Etnyre, Oleg Lazarev, Sam Lisi, Mark McLean, B\"ulent Tosun, and Chris Wendl. Special thanks to Kyle Hayden for pointing out missing assumptions in Theorem \ref{thm: planarNLF_multisection} in a preliminary draft. LLMs were only used in checking (and correcting) computations on vector fields in Lemma \ref{lem: contact_lambda} and Lemma \ref{lem: symplectic_to_J_transverse}, as well as writing a program to help us visualize the local model in \S \ref{subsec: exotic_weinstein}. LLMs were not used in generating ideas, texts, and figures.

\section{Nearly Lefschetz fibrations}\label{sec: PANLF}
\subsection{Spinal open books and nearly Lefschetz fibrations}
A nearly Lefschetz fibration is a generalization of Lefschetz fibration, with the addition of a new type of "singularity" near the horizontal boundary. The boundary of a nearly Lefschetz fibration takes the form of a spinal open book, which we describe first.

\begin{definition}\label{def:spinalobd}
    A \textbf{spinal open book decomposition} of a closed oriented $3$-manifold $M$ is a decomposition $M = M_\Sigma \cup M_P$ (into components called the \textbf{spine} and \textbf{paper}, respectively) together with a pair of fibrations $\pi = (\pi_{\Sigma} : M_\Sigma \to \Sigma,\,  \pi_{P}:M_P \to S^1)$ where:
    \begin{enumerate}
        \item The spine $\Sigma$ is a compact oriented surface whose connected components (called \textbf{vertebrae}) have nonempty boundary, and $\pi_\Sigma$ is a trivial fibration with $S^1$ fiber.
        \item The fibers of $\pi_P$ are compact oriented surfaces whose connected components (called \textbf{pages}, denoted by $P$) have nonempty boundary and intersect $\partial M_P$ transversely. Additionally, the intersection of $P$ and $M_\Sigma$ consists of fibers of $\pi_\Sigma$.
        \item For every component of $\partial M_P$, there exist local coordinates $(\phi, t, \theta) \in S^1 \times (-1,0] \times S^1$ on a collar neighborhood such that $\pi_P(\phi,t,\theta) = m\phi$ for some $m \in \mathbb{N}$. On every component of $\partial M_\Sigma$, there exist local coordinates $(s, \phi, \theta) \in (-1,0] \times S^1 \times S^1$ such that $\pi_{\Sigma}(s, \phi, \theta) = (s,\phi)$. The number $m$ is called the {\bf multiplicity} of $\pi_P$ at that boundary component of $M_P$. On the overlap between $M_P$ and $M_\Sigma$ the $2$-torus coordinates $(\phi, \theta)$ agree, while the interval coordinates are related by $s = -t$.
        \item The paper $M_P$ can be identified with a mapping torus 
        \[
            M_P = \R \times P \big/ \sim, \quad ( \tau, p) \sim (\tau-1,\phi(p))
        \]
        and $\pi_P: M_P \to S^1$ is given by $[(\tau, p)] \to [\tau]$. Here, $\phi \in \FMod(P)$ is the monodromy in the framed mapping class group of $P$. 
    \end{enumerate}
\end{definition}

\begin{definition}\label{def: panlf}
  A \textbf{nearly Lefschetz fibration} of a $4$-dimensional compact manifold with corners $E$ over a compact surface $B$ is a smooth map $\Pi: E \to B$ with finitely many interior critical points $E_{\crit} \subset E^{\circ}$ and critical values $B_{\crit} \subset B^{\circ}$, and finitely many \textbf{exotic points} $B_{\exot} \subset B^{\circ}$ such that:
  \begin{enumerate}
    \item For each $q \in B_{\crit}$, there are local holomorphic coordinates on $E$ centered at a corresponding critical point and on $B$ centered at $q$ such that $\Pi(z_1, z_2) = z_1^2 + z_2^2$.
    \item All fibers $E_z := \Pi^{-1}(z)$ for $z\in B$ are connected with nonempty boundary in $\partial_h E$.
    \item For each $p \in B_{\exot}$, there are local holomorphic coordinates $\psi \colon U_p \to \C$, where $U_p \subset B^{\circ}$, sending $p$ to $0$, and $\phi\colon V_p \to (\C \setminus D) \times \C$, where $V_p \subset \nu(\partial_h (E)) \cap \Pi^{-1}(U_p)$ and $D$ is a small neighborhood of $0 \in \C$, such that the following diagram commutes:
    \[
    \begin{tikzcd}
      V_p \arrow{r}{\phi}  \arrow{d}{\Pi} &(\C \setminus D) \times \C\arrow{d}{\Pi_p}\\
      U_p \arrow{r}{\psi} & \C
    \end{tikzcd}
    \]
    where $\Pi_p(z_1,z_2) = z_2^2-z_1$. We will refer to the fiber corresponding to $\Pi^{-1}(0)$ as an \textbf{exotic fiber}, and its boundary component in $V_p$ as an exotic boundary component. We will also denote the boundary component of $(\C \setminus D)$ as $\partial_+(\C\setminus D)$.
    \item The vertical boundary is $\Pi^{-1}(\partial B) = \partial_v E$ and $\Pi|_{\partial_v E} :\partial_v E \to \partial B$ is a smooth fiber bundle.
    \item The horizontal boundary $\partial_h E$, formed by the boundaries of the fibers of $\Pi$, admits a smooth fiber bundle structure $\Pi_h: \partial_h E \to B$, where \[\Pi_h = \Pi|_{\partial_h E} \text{ on } \partial_h E \setminus \bigcup_{c \in B_{\exot}} V_p\] and on the interior of $V_p$ the fiber at $\psi^{-1}(d)$ is given by $\phi^{-1}(\partial_+(\C \setminus D) \times \{\pm\sqrt{d}\})$ where $d \in \psi(U_p)$.
    \end{enumerate}
\end{definition}

The boundary of a nearly Lefschetz fibration $\Pi: E \to B$ naturally gives rise to a spinal open book. In the Lefschetz \emph{amenable} case, i.e. when there are no exotic fibers, the spinal open book has $\partial_v E$ as the paper and $\partial_h E$ as the spine, and the corresponding projection maps are obtained by $(\pi_P:=\Pi|_{\partial_v E}, \pi_{\Sigma}:=\tilde \Pi|_{\partial_h E})$, where $\tilde \Pi|_{\partial_h E}$ is defined by passing the base of the fibration of $\Pi|_{\partial_h E}$ to an oriented cover so that the fibers of $\tilde \Pi|_{\partial_h E}$ are connected. We now describe $\pi_\Sigma$ on the boundary of a nearly Lefschetz fibration in the presence of exotic fibers. In the "ideal" situation, i.e.\ when we complete the Lefschetz fibration by an appropriate double completion,
the exotic neighborhood is given a Seifert fibration. In a large compact cut-off, the exotic neighborhood contains singular "figure-eight" fibers as in Figure \ref{fig:hypersurf}. However, on the spine $M_\Sigma:=\partial_hE$, we nevertheless have another fibration given by the Reeb foliation in the local model; see \S \ref{subsec: exotic_weinstein}, where this explicit fibration essentially mimics the "ideal" situation. 

\begin{figure}
        \centering
	    \begin{overpic}[width=0.9\textwidth]{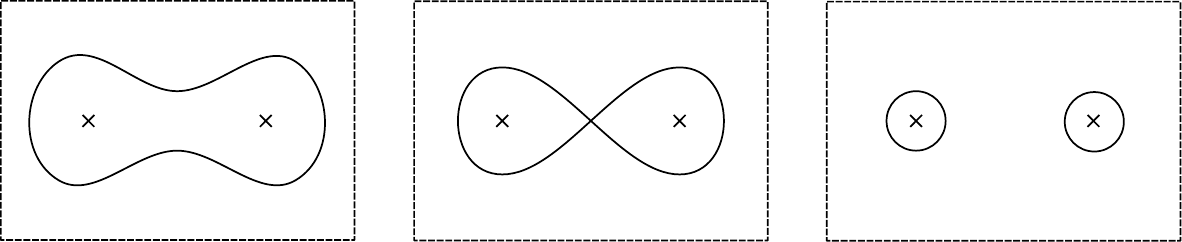}
        \put(5,5){close to an exotic fiber}
        \put(271,5){far from an exotic fiber}
        \end{overpic}
        \caption{Intersection of the fibers close to an exotic fiber with a contact hypersurface cut out at a level set near infinity.}
        \label{fig:hypersurf}
    \end{figure}
There is a Seifert fibration on $\partial_h E$ given by the lift of $\Pi_h$ defined in (5) of Definition \ref{def: panlf}.
 The base of this Seifert fibration $\Sigma$ is a disjoint union $\Sigma_1 \sqcup \dots \sqcup \Sigma_k$, and there are branched cover maps $\pi_i: \Sigma_i \to B$, with the branch points corresponding to exotic fibers. The projection map of the Seifert fibration is then $\pi_\Sigma$ and $ \Pi_h = \pi_i\circ\pi_{\Sigma}$ on each component of $\partial_h E$.

\begin{remark}
    Throughout the paper, we often suppress the distinction between a bordered Lefschetz fibration with corners and its appropriate smoothing, since the two Weinstein structures are canonically identified. 
    The exception is in Lemma \ref{lem: completion}, where the smoothing needs to be carefully chosen in order for the boundary hypersurface to be equipped with an appropriate stable Hamiltonian structure.
\end{remark}

A PANLF can be understood as a the complement of a positive multisection in a bordered Lefschetz fibration. See also \cite{baykur2016multisections, spinal_exotic}. We first recall the definition of branched cover of surfaces with cylindrical ends.
\begin{definition}[Definition 6.2 in \cite{spinal_II}]
    Let $\Sigma_1$ and $\Sigma_2$ be two closed oriented surfaces. Then a continuous map $\pi: \Sigma_1 - \{p_1, \dots, p_m\} \to \Sigma_2 - \{q_1, \dots, q_n\}$ is a \textbf{branched cover of surfaces with cylindrical ends (with degree $d$)} if it is proper and its extension to a map $\Sigma_1 \to \Sigma_2$ is a branched cover (with degree $d$). 
\end{definition}
\noindent In this paper, we often work with the compact-with-boundary version of the above definition.

\begin{definition} \label{def: multisection}
  Let $\Pi\colon E \to B$ be a bordered Lefschetz fibration. A properly embedded compact oriented surface $\Sigma \subset E$ is called a {\bf positive multisection} of $E$ if it satisfies the following conditions: 

  \begin{enumerate}
    \item The restriction $\Pi|_{\Sigma}$ is an $n$–fold simple branched covering map.
    \item For any branch point $p \in  \Sigma \setminus B_{\crit}$, there exist local coordinates $\phi: U \to \C^2$ of $p$ and local coordinates $\psi: V \to \C$ of $q = \Pi(p)$ such that the following diagram commutes: 
    \[
    \begin{tikzcd}
      (U, U \cap \Sigma) \arrow{r}{\phi}  \arrow{d}{\Pi} &(\C^2, \Sigma_0)\arrow{d}{f_0}\\
      V \arrow{r}{\psi} & \C
    \end{tikzcd}
    \]
    where $f_0: \C^2 \to \C$ is projection to the first component, and $\Sigma_0 = \{(z^2,z)\} \subset \C^2$.
    \item If a branch point $p \in \Sigma$ of $\Pi|_{\Sigma}$ is in $B_{\crit}$, then there are complex coordinates $(U,\phi)$ and $(V, \psi)$ as above such that $\phi(U \cap \Sigma)$ is equal to $\{(z,z) \in \C^2\}$. 
  \end{enumerate}
\end{definition}

\noindent Note that we can (and will) always assume the branch points to be disjoint from $B_{\crit}$ via a Weinstein homotopy by adjusting the critical points of the associated Morse function. We also occasionally drop the adjective positive in positive multisections when it is clear from context.

\subsection{Almost Stein structures}\label{subsec:almost_stein_review}
In this subsection, we review almost Stein structures, a structure lying between Weinstein and Stein structures.

\begin{definition}[Definition 1.22 in \cite{spinal_I}]
Let $W$ be a compact manifold, possibly with corners. An \textbf{almost Stein structure} on $W$ is a pair $(J, f)$ consisting of an almost complex structure $J$ and a smooth function $f: W \to \bR$ such that, for $\lambda:= -df \circ J$,
\begin{enumerate}
    \item $d\lambda$ is a symplectic form taming $J$;
    \item $\lambda$ restricts to a contact structure on every smooth face of $\partial W$.
\end{enumerate}
\end{definition}

We sometimes refer to $(J, f, \lambda)$ in the above definition as an \emph{almost Stein triple}. We will also refer to $f$ as a \emph{(strictly) plurisubharmonic}\footnote{The word "strictly" is often omitted in this paper as well as in the symplectic literature. See \cite[p33]{Cieliebak_Eliashberg}.} function. Note that compared to Stein structures, an almost Stein structure does not require $J$ to be integrable nor $f$ to be constant on the boundary. Nevertheless, if the Liouville vector field associated to $\lambda$ points transversely outward at the boundary of $W$ after smoothing, then $(J,f)$ determines a canonical Weinstein structure up to Weinstein homotopy via \cite{Cieliebak_Eliashberg}. See \cite[\S 1.1.3]{spinal_I} for further details relating Stein, almost Stein, and Weinstein structures. In particular, we have the following lemma as a consequence of the deep work of \cite{Cieliebak_Eliashberg}.
\begin{lemma}[{\cite{spinal_I}}]
\label{lem: almost_Stein=Stein}
    There is a canonical one-to-one correspondence between almost Stein homotopy classes on $W$ and the Stein homotopy classes on $W$ after smoothing corners. 
\end{lemma}

\section{Almost Stein structures on nearly Lefschetz fibrations}
\label{sec:PANLF_stein}
This section constructs a canonical Stein structure on a given PANLF. In \cite{spinal_I}, the authors build an almost Stein structure supporting a bordered Lefschetz fibration, which in turn determines a canonical Stein structure on the PANLF. 
Our strategy to build an almost Stein structure on a PANLF is similar to that in \cite{spinal_I} -- we extend their approach by locally constructing $J$-convex functions near the exotic neighborhoods and then matching them up with the fiberwise $J$-convex functions constructed in \cite[\S 2]{spinal_I}. After that, an application of the Thurston trick allows us to construct a global $J$-convex function, and therefore an almost Stein structure, compatible with the PANLF. Finally, we use a similar argument via a Serre fibration as in \cite[\S 2]{spinal_I} to show that the space of compatible almost Stein structures on a given PANLF is contractible.

Recall that as part of the definition of a nearly Lefschetz fibration, we have the induced spinal open book on its boundary and can talk about the fibers of $\pi_h$, i.e. the circle fibers in the spine. The following definition is similar to Definition 1.23 in \cite{spinal_I}.

\begin{definition}[Supported almost Stein structures]
\label{def: supported_almost_Stein}
    Let $\Pi: E \to B$ be a nearly Lefschetz fibration. An almost Stein structure $(J, f)$ on $E$ is \textbf{supported by} $\Pi$ if
\begin{itemize}
    \item There exists a complex structure $j$ such that $\Pi: (E, J) \to (B, j)$ is pseudoholomorphic;
    \item $\lambda:= -df \circ J$ restricts to $\partial E$ as a Giroux form for $\pi:=\partial \Pi$;
    \item For every $z\in B$, $f$ is constant on the fibers of $\pi_h$;
    \item The maximal $J$-complex subbundle of $T(\partial_h E)$ is invariant under the Reeb flow determined by $\lambda|_{T(\partial_h E)}$.
\end{itemize}
\end{definition}

\begin{definition}[Fiberwise almost complex structure]
\label{def: fiberwise_acs}
    Let $\cJ(\Pi)$ denote the space of pairs $(J, \partial \theta)$ where $J$ is an almost complex structure on $E$ compatible with orientation and $\partial \theta$ is a nowhere zero vector field on $\partial_h E$ such that
\begin{itemize}
    \item $\partial \theta$ is oriented in the positive direction of the fibers (of $\pi_\Sigma: \partial_h E \to \Sigma$);
    \item There exists a complex structure $j$ on $B$ such that $\Pi: (E, J) \to (B, j)$ is pseudoholomorphic;
    \item The flow of $\partial \theta$ is $1$-periodic; 
    \item The maximal $J$-complex sub-bundle in $\partial_h E$ is the $2$-plane field preserved by the flow of $\partial \theta$. 
    \end{itemize}
\end{definition}

\begin{remark}
    In this paper, we will use the notation $J\in\cJ(\Pi)$ for the discussion only concerning the almost complex structures that make $\Pi$ pseudoholomorphic. The vector field $\partial \theta$ gives additional information of a strict contact fiber bundle on $\partial_h E$. See also the \cite[Remark 2.16]{spinal_I}.
\end{remark}

\begin{lemma}\label{lem:contractible_acs_space}
    The space of almost complex structures $\cJ(\Pi)$ is contractible. 
\end{lemma}
\begin{proof}
    Fix a $\partial \theta$, which is a contractible choice. Now, choosing $J \in \cJ(\Pi)$ is equivalent to choosing $j$ on $B$, and an oriented almost complex structure on the plane field preserved by the flow of $\partial\theta$. Since the space of positively oriented complex structures on any oriented real vector bundle of rank two is nonempty and contractible, these are both contractible choices.
\end{proof}

\begin{definition}[Fiberwise forms]\label{def: fiberwise_forms}
    Let $VE:= \ker T\Pi \subset TE$, $V(\partial_h E):= VE \cap T(\partial_h E)$, and let $X$ be a non-vanishing vector field on $\partial E$ that is oriented along the fibers of $\pi_h$ in $\partial_h E$. Note that $X$ agrees with $V(\partial_h E)$ away from the exotic neighborhoods. 
    Then, $\Lambda^{\text{fib}}(\partial \Pi)$ denotes the space of germs of $1$-forms $\lambda$ defined on a neighborhood of $\partial E$ in $E$ such that
    \begin{itemize}
        \item $\lambda \in \Lambda^{\text{fib}}(\partial_h \Pi)$, i.e. 
        $$\lambda(X)>0 \text{ and } X \subset \ker \Big(d\lambda|_{T(\partial_h E)}\Big)$$
        \item $d\lambda|_{VE}>0$ at $\partial_v E$.
    \end{itemize}
    Similarly, $\Omega_{\text{strong}}^{\text{fib}}(\Pi)$ (resp. $\Omega_{\text{exact}}^{\text{fib}}(\Pi)$) denotes the space of germs of closed $2$-forms $\omega$ such that
    \begin{itemize}
        \item $\omega>0$ on all fibers in $E \setminus E_{\crit}$;
        \item $\omega$ is nondegenerate and tames $J_{\crit}$ on $E_{\crit}$, where $J_{\crit}$ is fixed such that $\Pi$ is holomorphic near $E_{\crit}$;
        \item $\omega$ is nondegenerate and tames $J_{\exot}$ on the intersection of exotic fibers with $\partial_h E$, where $J_{\exot}$ is fixed such that $\Pi$ is holomorphic near the intersection of the exotic fibers with $\partial_h E$;
        \item $\omega = d\lambda$ near $\partial E$ for some $\lambda \in \Lambda^{\fib}(\partial \Pi)$ (and in the case of $\Omega_{\text{exact}}^{\text{fib}}(\Pi)$, $\lambda$ extends to a global primitive of $\omega$ on $E$).
    \end{itemize}
\end{definition}

Roughly speaking, suppose $J \in \mathcal{J}(\Pi)$ and $\lambda = -df \circ J$, $\omega = d\lambda$. Then, $\omega\in \Omega_{\text{exact}}^{\text{fib}}(\Pi)$ implies that $\omega$ is positive on the fibers of $\Pi$ and its primitive $\lambda$ belongs to $\Lambda^{\text{fib}}(\partial \Pi)$.

\begin{definition}[{Fiberwise $J$-convex, \cite[Definition 2.17]{spinal_I}}]
\label{def: fiberwise_J_convex}
    Let $(J, \partial \theta)\in \cJ(\Pi)$ and $f: E \to \bR$ be a smooth function. Let $\lambda:= -df \circ J$. The function $f$ is \textbf{fiberwise $J$-convex} if
    \begin{itemize}
        \item $f$ is constant on the orbits of $\partial \theta$; 
        \item $d\lambda \in \Omega_{\text{exact}}^{\text{fib}}(\Pi)$;
        \item $\lambda|_{\partial E} \in \Lambda^{\text{fib}}(\partial \Pi)$;
        \item $\lambda(\partial \theta)$ is constant. 
    \end{itemize}
    The space of fiberwise $J$-convex function with respect to $(J, \partial \theta)$ is denoted as $\PSH^{\fib}_{(J, \partial \theta)}(\Pi)$.
\end{definition}

\subsection{Almost Stein structure near an exotic fiber}\label{subsec: exotic_weinstein}
For this subsection, we choose the local model near the exotic fiber to be $z_2^2+z_1$ instead of $z_2^2-z_1$. This is equivalent up to a biholomorphism and makes the notation slightly less cluttered.

Let $\Pi: E\to B$ be a nearly Lefschetz fibration. We first fix an appropriate neighborhood of an exotic fiber. Let $p$ denote the image of an exotic fiber. Given a sufficiently small neighborhood $B_p \subset B$ of $p$, there exist $0 < \delta < 1$ and $A_{\delta} \coloneqq \{\delta < |z| < 1\}$, such that the preimage $\Pi^{-1}(B_p)$ contains a subdomain $V_p \cong A_{\delta} \times D^2$ with complex coordinates $(z_1,z_2)$ so that $\Pi|_{V_p} = z_2^2+z_1$ and $\Pi(0) = p$ (see Figure~\ref{fig:annulus}). In particular, there is an almost complex structure $j$ on $\Pi(V_p)$ such that 
$$\Pi \circ J_{\exot} = j \circ \Pi,$$
where $J_{\exot}$ is the pullback of the standard almost complex structure on $A_\delta \times D^2$ to $V_p$. Let $f_A$ be a function on $A_{\delta}$ such that it is 
$$f_{A}(z_1):= -\log(|z_1|)$$
near $|z_1| = \delta$ and
$f_A(z_1) = \Rea(z_1)$ near $|z_1|= 1$. See Figure~\ref{fig:annulus} for a picture of its gradient vector field. Let $f_{D^2} := 2x_2^2 - y_2^2$. On $A_{\delta} \times D^2$, we consider the function 
$$f_{\exot} \coloneqq 
\frac{1}{2}\bigg(f_A+\epsilon f_{D^2}\bigg) = -\log|z_1| +\epsilon(2x_2^2 - y_2^2)$$ where $\epsilon>0$ is arbitrarily small.

    \begin{figure}
        \centering
        \includegraphics[width=0.5\linewidth]{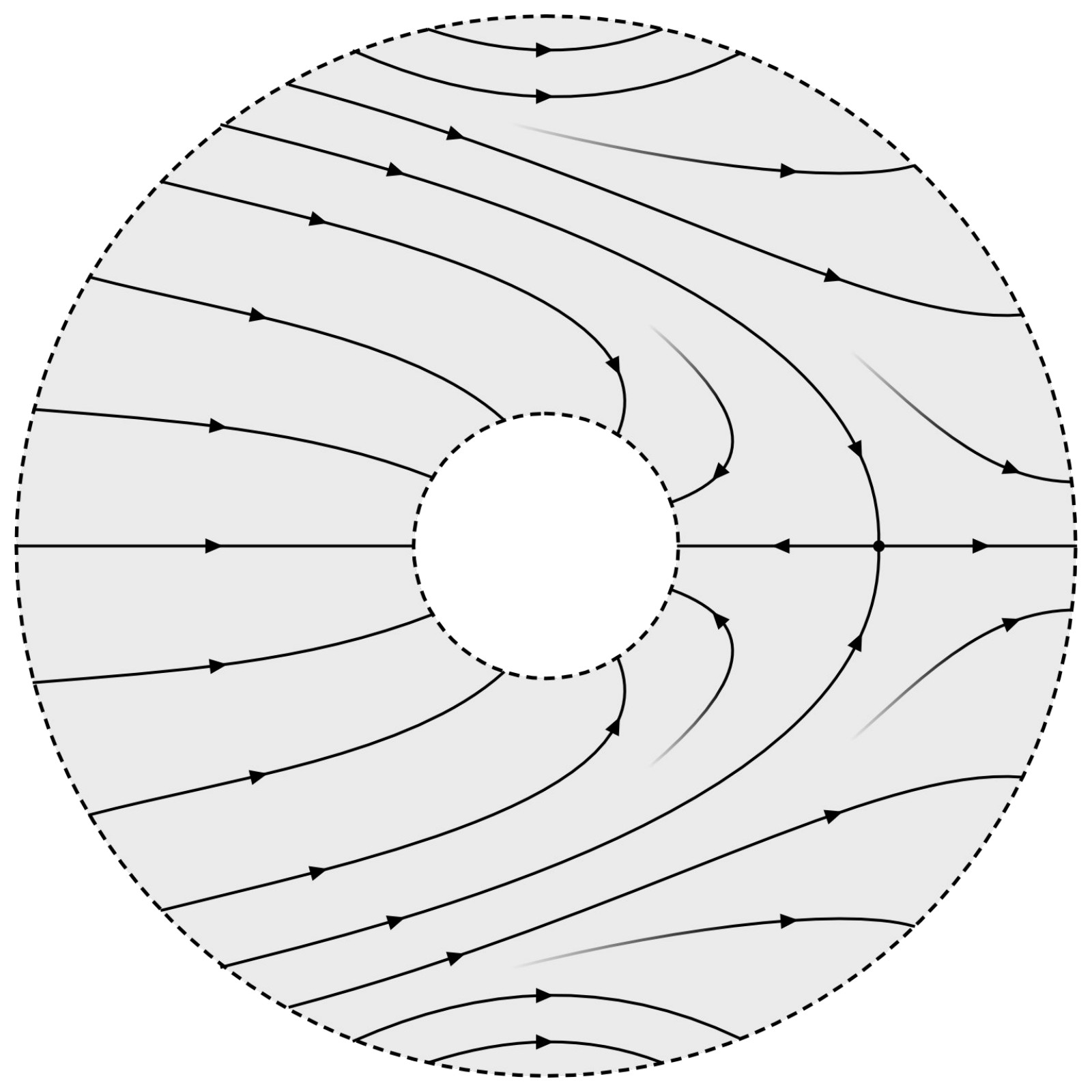}
        \caption{Open annulus and the gradient vector field of $f_A$.}
        \label{fig:annulus}
    \end{figure}

Now, we consider a model for $V_p$ that is ``compatible'', in a sense that will be clarified, with $f_{\exot}$. Let $S^1_\delta$ denote the inner boundary component of $A_\delta$, i.e. the boundary of the $\delta$-disk. Also, consider the second factor of $A_\delta \times D^2$ to be a cornered domain inside $D^2$ so that the gradient flow of $2x_2^2-y_2^2$ is positively transverse along part of the smooth boundary, denoted as $H$, and negatively transverse along the rest. Abusing notation, we will continue referring to this domain as $D^2$. The boundary of $A_\delta \times D^2$ that is positively transverse to the flow corresponding to $f_{\exot}$ is $(S^1_\delta \times D^2) \cup (A_\delta \times H)$.

We may assume that $B_p$ has been chosen small enough so that $V_p$ does not intersect the vertical boundary of $E$. Therefore, the intersection of $V_p$ with the horizontal boundary is given by $S^1_\delta \times D^2$. 

\begin{definition}\label{def:good_exotic_nbd}
For sufficiently small $\delta>0$, let $U_{\exot}^{\delta} := (A_\delta \times D^2) \cap \{f_{\exot}\leq -\log(\delta)\}$. Let $Y_{\exot}^\delta: = (A_\delta \times D^2) \cap \{f_{\exot}= -\log(\delta)\}$.
\end{definition} 

There is a fiber-preserving isotopy identifying $U_{\exot}^{\delta}$ with $A_\delta \times D^2$, by taking $\epsilon$ to 0. Thus, the identification of $V_p$ with $A_\delta \times D^2$ can be further composed with this isotopy, to identify $V_p$ with $U_{\exot}^{\delta}$. We use this to locally define an almost Stein structure on $V_p$ and the set $Y_{\exot}^\delta$ is identified with $V_p \cap \partial_h E$, which furthermore is identified with $V_p \cap \partial E$ by the choice of $B_p$.

As a consequence these identification, we are now ready to state the following lemma, which allows us to build fiberwise almost Stein structures on PANLFs. 

Let $\lambda_{\exot}:= - df_{\exot} \circ J_{\exot}$, where $J_{\exot}$ denotes the standard complex structure on $\C^2$ restricted to $U_{\exot}^{\delta}$. Let $(r_1, \theta_1)$ denote the polar coordinate on $A_\delta \subset \bC$.

\begin{lemma}\label{lem: contact_lambda}
For each $K\gg 0$, the $1$-form $\lambda_{\exot}$ is a contact form on $Y_{\exot}^K$ with Reeb vector field $\partial \theta_1$. 
\end{lemma}

\begin{proof}
     Use the coordinates $z_j:= x_j + i y_j$. By definition, $J_{\exot} \circ \partial x_j = \partial y_j$ and $J_{\exot} \circ \partial y_j = -x_j$. We may compute
    \begin{align*}
        df_{\exot} = -\frac{1}{|z_1|^2}(x_1dx_1 + y_1 dy_1) + \epsilon(4x_2dx_2 - 2y_2dy_2)
    \end{align*}
    and
    \begin{align*}
        \lambda_{\exot} &= - df_{\exot} \circ J_{\exot} = -\frac{1}{|z_1|^2}(x_1dy_1 - y_1 dx_1) + 2\epsilon(2x_2 dy_2 + y_2 dx_2) \\
        &= - d\theta_1 + \epsilon(4x_2 dy_2 + 2y_2 dx_2),
    \end{align*}
    since $d\theta_1 = \frac{x_1dy_1 - y_1dx_1}{|z_1|^2}$.

    Now, $$d\lambda_{\exot} = \epsilon(4dx_2dy_2 +2dy_2dx_2) = 2\epsilon dx_2dy_2.$$
    Then, we have that
    \begin{align*}
        \lambda_{\exot} \wedge d\lambda_{\exot} 
        &= 2\epsilon (-d\theta_1) \wedge dx_2 \wedge dy_2,
    \end{align*}
    which is always contact on any level set $\{f_{\exot} = K\}$. The Reeb vector field $X_{\exot}$ is $\partial \theta_1$.
\end{proof}

This gives us the almost Stein triple $(f_{\exot}, J_{\exot}, \lambda_{\exot})$ in the exotic neighborhood $V_p$. Furthermore, by construction, $f_{\exot}$ is fiberwise $J_{\exot}$-convex when considered as the restriction to the neighborhood $V_p$. In the rest of this section, we show that this almost Stein triple can be extended to an almost Stein structure on the whole positive allowable nearly Lefschetz fibration. In \S \ref{subsec:panlf_weinstein}, we see a Weinstein handle description of the almost Stein structure of the exotic neighborhood.

\subsection{Fiberwise \texorpdfstring{$J$}{J}-convex}

The following lemma is analogous to \cite[Proposition 2.10]{spinal_I}, which allows us to assume some standard properties in an exotic neighborhood for any $(J,\partial \theta) \in \mathcal{J}(\Pi)$.

\begin{lemma}\label{lem:standard_acs_exot}
Let $(J,\partial\theta) \in \mathcal{J}(\Pi)$, $p \in U_p \subset B_{\exot}$, and $V_p \subset \Pi^{-1}(U_p)$ as in Definition~\ref{def: panlf}. Then, there exists a function $f_J$ on $V_p$ such that
\begin{enumerate}
    \item The function $f_J$ is $J$-convex;
    \item $\lambda_J(\partial \theta)$ is constant where $\lambda_J \coloneqq -df_J \circ J$;
    \item The 2-plane field preserved by $\partial \theta$, inside $\partial_h E \cap V_p$, is the contact structure $\ker(\lambda_J)$.
\end{enumerate} 
\end{lemma}

\begin{proof}
    Let us denote the standard complex structure on $\C^* \times \C$ as $\widehat{J}_{\exot}$, to be consistent with the notation in Section~\ref{subsec: exotic_weinstein}. 
    Let $\widehat{V}_p$ denote the completion of the exotic neighborhood and extend $J$ as $\widehat{J}$ on $\widehat{V}_p$. This is given by a usual cylindrical construction where we consider a local Giroux form making $\partial \theta$ into a Reeb vector field and define the cylindrical direction as $J(\partial\theta)$. Note that since the completion takes place along the horizontal boundary, we do not need the full construction of double completion, in particular the contact-type smoothing, of Definition-Theorem \ref{def:double_completion_LF}.
    
    Now, by Definition \ref{def: panlf}, there exists a diffeomorphism $\phi: V_p \to (\bC \setminus D) \times \bC$ such that the fibers are exactly the level sets of $z_2^2-z_1$. Therefore, together with the fact that $J\in \cJ(\Pi)$, the completed almost complex structure $\widehat{J}$ makes the completed fibers in $\widehat{V}_p$ all pseudoholomorphic. Now, by \cite[Proposition 4.10]{spinal_exotic},
    we have a diffeomorphism
    \[
    \Phi: \widehat{V}_p \to \C^* \times \C
    \]
    identifying the pseudoholomorphic leaves on both sides. Now, consider the almost complex structures $\Phi^*(J_{\exot})$ and $\widehat{J}$ on $\widehat{V}_p$. We can ensure that they agree on the $\langle \partial \theta, \bR \rangle$ plane via a potential further diffeomorphism. We still denote the composition as $\Phi$. In addition, for any completed fiber $\widehat{F}$ of $\Pi|_{\widehat{V}_p}$, note that $\Phi^*(J_{\exot})|_{\widehat{F}}$ and $\widehat{J}|_{\widehat{F}}$ are related by a biholomorphism, since $\widehat{F}$ is a twice-punctured complex plane. By composing $\Phi$ with a fiberwise biholomorphism which is again still denoted as $\Phi$, we can ensure that $\Phi^*(J_{\exot}) = \widehat{J}$.

    Observe that the function $f_{\exot}$ can naturally be extended to all of $\C^* \times \C$. Now, define $f_J$ on $\widehat{V}_p$ as $\Phi^*f_{\exot}$; thus $\lambda_J = \Phi^*\lambda_{\exot}$. The intersection of $\partial E$ with $V_p$, which is included inside $\partial_h E$, is then a level set of $f_J$, such that $\lambda_J$ is contact. Now, by definition of $\mathcal{J}(\Pi)$, the maximal $J$-complex subbundle in $\partial_h E$ is the 2-plane field preserved by the flow of $\partial \theta$. This implies that in $\partial_h E\cap V_p$, this 2-plane field coincides with $\ker(\lambda_J)$, and the orbits of $\partial\theta$ agree with the orbits of the Reeb flow of $\lambda_J$. Therefore, by above choices of $\Phi$, $\partial\theta$ is precisely the Reeb vector field of $\lambda_J$.
\end{proof}

Now, analogous to \cite[Proposition 2.25]{spinal_I}, we prove that fiberwise $J$-convex functions exist on PANLFs.

\begin{proposition}[Fiberwise almost Stein]
\label{prop: fiberwise_almost_Stein}
    Let $E$ be a PANLF with the projection map $\Pi: E \to B$. Fix $(J, \partial \theta) \in \cJ(\Pi)$. Then there exists $f\in \PSH^{\fib}_{(J, \partial \theta)}(\Pi)$ and $\PSH^{\fib}_{(J, \partial \theta)}(\Pi)$ is contractible. 
\end{proposition}
\begin{proof}
    The construction of $f\in \PSH^{\fib}_{(J, \partial \theta)}$ depends on an arbitrary choice of $(J, \partial \theta) \in \mathcal{J}(\Pi)$, which we fix throughout. The proof of \cite[Proposition 2.25]{spinal_I} gives fiberwise $J$-convex functions in neighborhoods $\cU_z\subset B$ of regular and singular values $z\in B \setminus B_{\exot}$, denoted as $f_z: E|_{\mathcal{U}_z} \to \bR$. We can check directly that the first and fourth conditions are convex, and the second are third are convex because we are considering fiberwise forms.
    
    In \S \ref{subsec: exotic_weinstein}, we constructed an almost Stein triple $(f_{\exot}, J_{\exot}, \lambda_{\exot})$ near the exotic fibers. Given an exotic fiber at $p \in B_{\exot}$ and a neighborhood $p \in B_p \subset \mathring{B}$, we consider a neighborhood $V_p \subset \Pi^{-1}(B_p)$. In particular, we defined the almost Stein triple on $U_{\exot}^{\delta} \subset A_\delta \times D^2$ and then used an identification of $V_p$ with $U_{\exot}^{\delta}$ to define them on $V_p$ via pullbacks. We need to show that near the exotic fibers, for any $(J,\partial \theta)$ in $\cJ(\Pi)$, there exists a locally defined function, which is fiberwise $J$-convex. This follows from Lemma~\ref{lem:standard_acs_exot}.
    
    Now, we use partition of unity as in the proof of \cite[Proposition 2.25]{spinal_I}. Specifically, choose a partition of unity $\{\rho_z: \cU_z \to [0,1]\}_{z\in I}$ subordinate to a finite cover $\{\cU_z\}_{z\in I}$ of $B$. To match $f_z$ near $\partial_h E$ to a common large positive constant, we may post-compose $f_z$ with a function $h$ with large second derivative while still keeping $h\circ f_z$ $J$-convex by \cite[Lemma 2.21]{spinal_I} which is still denoted as $f_z$; c.f. \cite[Lemma 2.7]{Cieliebak_Eliashberg}. Similarly, we may match $-df_z \circ J \partial \theta$ to a large positive constant. Let 
    $$f:=\sum_{z\in I}(\rho_z\circ \Pi) f_z.$$
    We can check that $\lambda:=-df\circ J$ satisfies $\lambda|_{\partial E} \in \Lambda^{\text{fib}}(\partial \Pi)$ and   
    $d\lambda \in \Omega_{\text{exact}}^{\text{fib}}(\Pi)$ since near $\partial_h E$, $\lambda(\partial \theta) \equiv \nu >0$ and $d\lambda(\partial \theta, \cdot) \equiv 0$ as a consequence of the fact that $\ker \lambda$ is invariant under the flow of $\partial \theta$.
\end{proof}

\subsection{Thurston's trick}
Now, we use Thurston's trick to construct an almost Stein structure on the total space of $\Pi$ from the fiberwise almost Stein structure.

\begin{definition}(Giroux forms)
    The space of \textbf{Giroux forms near $\partial E$} is defined to be
    $$\Lambda(\partial \Pi):=\{\lambda \in \Lambda^{\text{fib}}(\partial \Pi)\mid\lambda|_{T(\partial_h E)} \ \text{and} \ \lambda|_{T(\partial_v E)} \ \text{are both contact}\}.$$
\end{definition}

\begin{definition}
\label{def: PSH}
    Define $\PSH_{(J, \partial \theta)}(\Pi)\subset \PSH^{\fib}_{(J, \partial \theta)}(\Pi)$ to be the subspace for which $d\lambda_J$ is a symplectic form taming $J$ and $\lambda_J|_{\partial E} \in \Lambda(\partial \Pi)$.
\end{definition}

\begin{proposition}
\label{prop:Thurston_fK}
    Let $(J, \partial \theta) \in \cJ(\Pi)$ and $f\in \PSH^{\fib}_{(J, \partial \theta)}(\Pi)$. Let $j$ be the unique complex structure on $B$ such that $\Pi: (E,J)\to (B, j)$ is pseudoholomorphic and let $\varphi$ be a plurisubharmonic function on $(B, j)$. Then there exists $K_0 \geq 0$ depending continuously on $J$ and $f$, such that for every $K\geq K_0$,
    $$f_K:= f + K(\varphi \circ \Pi)$$
    belongs to $\PSH_{(J, \partial \theta)}(\Pi)$.
\end{proposition}
\begin{proof}
    By \cite[Theorem 2.27]{spinal_I}, it suffices to show that $f_K$ is $J$-convex in the exotic neighborhood. 
    Let $\sigma:= -d\varphi \circ j$. Observe that 
    $$\lambda_J^K: = -df_K \circ J = \lambda - Kd(\varphi \circ \Pi) = \lambda - K\Pi^*(-d\varphi \circ j)= \lambda + K \Pi^* \sigma$$
    by holomorphicity of $\Pi$ and 
    \begin{equation}
    \label{eq: dlambda}
        d\lambda_J^K(v, Jv) = d\lambda_J(v, Jv) + K d\sigma(\Pi_* v, j\Pi_* v).
    \end{equation}
    Note that $K d\sigma(\Pi_* v, j\Pi_* v)\geq 0$ for $K \geq 0$. Now, since $f$ is $J$-convex in the exotic neighborhood by Lemma~\ref{lem:standard_acs_exot}, we have that $d\lambda_J(v, Jv)>0$ so $f_K$ is in $\PSH_{(J, \partial \theta)}(\Pi)$.
\end{proof}

\begin{definition}[{\cite[Definition 2.18]{spinal_I}}]
\label{def: J_Stein}
    The \textbf{space of supported almost Stein structures} associated to a nearly Lefschetz fibration with projection $\Pi$ is
    $$\cJ_{\Stein}(\Pi):= \{(J, f) \mid (J, \partial \theta) \in \cJ(\Pi) \text{ for some } \partial \theta \text{, and } f\in \PSH_{(J, \partial \theta)}(\Pi)\}.$$
\end{definition}

While Definition \ref{def: J_Stein} consists of almost Stein structures that are compatible with some $\partial \theta$, this $\partial \theta$ is canonically determined if $\lambda \in \Lambda^{\text{fib}}(\partial_h \Pi)$. 

Notice that when $(J,f)\in \cJ_{\Stein}(\Pi)$, by definition the associated $\lambda_J|_{\partial E} \in \Lambda(\partial \Pi)$ and in particular, $\lambda \in \Lambda^{\text{fib}}(\partial \Pi)$ which implies $\lambda \in \Lambda^{\text{fib}}(\partial_h \Pi)$. This implies that we have a well-defined projection
\begin{equation}
\label{eq:almost_stein_fibration}
    \cJ_{\Stein}(\Pi) \to \cJ(\Pi):(J, f) \mapsto (J, \partial \theta),
\end{equation}
with fiber $\PSH_{(J,\partial \theta)}(\Pi)$. In fact, this is a Serre fibration, since the $J$-convex function $f_K$ constructed in Proposition \ref{prop:Thurston_fK} can be done parametrically. Now we show contractibility of $\cJ_{\Stein}(\Pi)$ by considering the homotopy sequence of a Serre fibration.
\begin{proposition}[Contractibility of supported almost Stein structures]
\label{prop: contractible_almost_Stein}
    Let $E$ be a PANLF with the projection map $\Pi: E \to B$. Then the projection (\ref{eq:almost_stein_fibration}) is a homotopy equivalence. In particular, $\cJ_{\Stein}(\Pi)$ is contractible.
\end{proposition}
\begin{proof}
    With Proposition \ref{prop: fiberwise_almost_Stein} and Proposition \ref{prop:Thurston_fK}, the proof essentially follows from that of \cite[Proposition 2.19]{spinal_I}. For completeness, we give a proof here. First we note that, by Definition \ref{def: PSH}, any $f \in \PSH_{(J,\partial \theta)}(\Pi)$ is also in $\PSH^{\fib}_{(J,\partial \theta)}(\Pi)$.
    In addition, $\PSH^{\fib}_{(J,\partial \theta)}(\Pi)$ is convex by the proof of Proposition \ref{prop: fiberwise_almost_Stein}. 

    Now, since $\PSH^{\fib}_{(J,\partial \theta)}(\Pi)$ is convex, any $S^{n-1}$-family of functions in $\PSH_{(J,\partial \theta)}(\Pi)$ extends to a $\bD^n$-family of fiberwise $J$-convex functions. Now, we consider a cutoff function $\psi: \bD^n \to [0,1]$ which is $0$ near $\partial \bD^n$ and $1$ outside the neighborhood. Since $\bD^n$ is compact, we may find a large $K_0$ such that for $K \geq K_0$, the modified Thurston construction $f_K:= f + K\psi \circ (\phi \circ \Pi)$ from Proposition \ref{prop:Thurston_fK} remains in $\PSH_{(J,\partial \theta)}(\Pi)$. Now, we have extended the $S^{n-1}$-family of functions in $\PSH_{(J,\partial \theta)}(\Pi)$ to an $\bD^{n}$-family in the same space. Therefore, by Whitehead's theorem, $\PSH_{(J,\partial \theta)}(\Pi)$ is contractible. Finally, applying Whitehead's theorem again, the projection \eqref{eq:almost_stein_fibration} is a homotopy equivalence. Since $\cJ(\Pi)$ is contractible by Lemma \ref{lem:contractible_acs_space}, $\cJ_{\Stein}(\Pi)$ is also contractible.
\end{proof}

\begin{proof}[Proof of Theorem \ref{thm: stein_on_panlf}]
    This follows from Lemma \ref{lem: almost_Stein=Stein} and Proposition \ref{prop: contractible_almost_Stein}.
\end{proof}

\begin{remark}
    Similar to \cite[Theorem 1.24]{spinal_I}, given a nearly Lefschetz fibration $\Pi: E\to B$, we can show via a similar argument as in the proof of Theorem \ref{thm: stein_on_panlf} that the space of supported symplectic structures that are strongly convex are both nonempty and contractible. As a summary, one needs to (1) check the space of fiberwise symplectic structures $\Omega_{\text{strong}}^{\text{fib}}(\Pi)$ is nonempty and contractible and (2) use Thurston trick to go from fiberwise space to the total space of symplectic structures supporting $\Pi$ and show that the latter is homotopy equivalent to the former. In comparison to \cite[\S 2.4]{spinal_I}, we only need to examine the situation near the exotic fibers. Observe that in the local model in Section~\ref{subsec: exotic_weinstein}, we construct a symplectic structure on the exotic neighborhood such that it restricts to a contact form on the boundary, whose Reeb orbits are exactly the orbits of $\pi_h$. Thus, by using partitions of unity, we can glue these local constructions around exotic fibers with constructions around regular and singular fibers, to conclude that $\Omega_{\text{strong}}^{\text{fib}}(\Pi)$ is nonempty. The rest of the argument proceeds similarly as in \cite[\S 2.4]{spinal_I}.
\end{remark}

Note that this proposition shows that given a $\Pi$, the space of almost complex structures supported by $\Pi$ is contractible. However, one usually prefers to talk about Stein homotopy independent of the projection maps. This ``quasiflexibility'' that only depends on the boundary symplectic data (under suitable assumptions) is the content of the next section. 

\section{Quasiflexibility of Weinstein structures on PANLFs}
\label{sec: quasiflex}
In this section, we discuss a ``quasiflexibility'' phenomenon of Stein structures on nearly Lefschetz fibrations with planar fibers and prove Theorem \ref{thm: quasiflexibility}. Roughly speaking, quasiflexibility means that a symplectic homotopy convex at infinity induces a Stein homotopy. In the setting of Lefschetz fibrations, this was first done by \cite{Wendl_strongly} for planar Lefschetz fibrations over disks and the by \cite{spinal_I} for amenable planar Lefschetz fibrations over more general compact surfaces. Theorem \ref{thm: quasiflexibility} is a further generalization of \cite{spinal_I}, allowing exotic fibers in the setting of planar nearly Lefschetz fibrations. The assumption of planar is classical, since we rely on the classification of pseudoholomorphic foliations in family. We give an application of these ideas in \S \ref{sec: planarNLF_multisection} in the proof of Theorem \ref{thm: planarNLF_multisection}.

\subsection{Peripherally supported almost Stein structures} As observed in \cite[\S 3]{spinal_I}, to study almost Stein structures up to almost Stein homotopy, we need to relax the first condition in Definition \ref{def: supported_almost_Stein}, i.e., the projection $\Pi$ needs not be pseudoholomorphic on all of $B$, but just an open neighborhood $\cU \subset B$ of $\partial B$. This will give more room to construct the homotopy of almost Stein structures, when, in the homotopy of symplectic structures convex at infinity, we do not necessarily have $\Pi$-supported almost Stein structures along the path.
The goal of this subsection is to show that a Stein structure which is compatible with a nearly Lefschetz fibration $\Pi$ in a neighborhood of $\partial \Pi$ (but not necessarily on the whole space) is nevertheless almost Stein homotopic to a compatible one in the sense of Definition \ref{def: supported_almost_Stein}. 

\begin{definition}[Peripherally supported almost Stein structure]
\label{def: periphally_supported_almost_Stein}
    Let $\Pi: E \to B$ be a nearly Lefschetz fibration. Let $j$ be an almost complex structure on $B$. An almost Stein structure $(J, f)$ on $E$ is \textbf{peripherally supported by} $(\Pi, B, j)$ if
\begin{itemize}
    \item $J$ restricts to a positively oriented complex structure on the smooth part of each fiber;
    \item The projection $\Pi: (E, J) \to (B, j)$ is pseudoholomorphic on an open neighborhood $\cU \subset B$ of $\partial B$;
    \item $\lambda:= -df \circ J$ restricts to $\partial E$ as a Giroux form for $\pi:=\partial \Pi$;
    \item For every $z\in B$, $f$ is constant on each connected component of $\partial E_z$;
    \item The maximal $J$-complex subbundle $\xi_J$ of $T(\partial_h E)$ is invariant under the Reeb flow determined by $\lambda|_{T(\partial_h E)}$ and $J|_{\xi_J} = \Pi^* j$. 
\end{itemize}
\end{definition}

In the following, we state the quasiflexibility result for nearly Lefschetz fibrations. This is a generalization of \cite[Theorem C]{spinal_I}. 

\begin{theorem}
\label{thm: peripheral_asc}
    Let $\Pi: E \to B$ be a nearly Lefschetz fibration and let $(J_i, f_i, \lambda_i:= -df_i \circ J_i)$ be associated almost Stein triples peripherally supported by $\Pi$ and $J_0$ is supported by $\Pi$. Then $(J_0, f_0, \lambda_0)$ and $(J_1, f_1, \lambda_1)$ are almost Stein homotopic.
\end{theorem}

\begin{remark}
    In fact, the proof of Theorem \ref{thm: peripheral_asc} shows that the space of peripherally supported almost Stein triples is contractible.
\end{remark}

Corresponding to the generalization in Definition \ref{def: periphally_supported_almost_Stein}, we need to modify the space of supported almost complex structure in Definition \ref{def: fiberwise_acs}.

\begin{definition}[Peripheral almost complex structure]
\label{def: peripheral_acs}
    Given an open subset $\cU\subset B$, let $\cJ(\Pi; \cU)$ denote the space of pairs $(J, \partial \theta)$ where $J$ is an almost complex structure on $E$ compatible with the orientation and $\partial \theta$ is a nowhere zero vector field on $\partial_h E$ such that
\begin{enumerate}
    \item $\partial \theta$ is oriented in the positive direction of the fibers (of $\partial_h E \to B$);
    \item $J$ restricts to a positively oriented complex structure on the smooth part of every fiber; 
    \item There exists a complex structure $j$ on $B$ such that $\Pi: (E, J) \to (B, j)$ is pseudoholomorphic in $E|_{\cU}$ and along $\partial_h E$;
    \item The flow of $\partial \theta$ is $1$-periodic and preserves $\xi_J: = T(\partial_h E) \cap JT(\partial_h E)$. 
\end{enumerate}
\end{definition}

As in \cite{spinal_I}, the proof of Theorem~\ref{thm: peripheral_asc} follows from the technical proposition stated below, which allows us to associate to any $(J,\partial \theta) \in \mathcal{J}(\Pi;\mathcal{U})$ where $\mathcal{U}$ contains $\partial B$, a perturbation $(J',\partial \theta') \in \mathcal{J}(\Pi;\mathcal{U})$ which admits a suitable $J$-convex function. Further this perturbation can be arranged to depend continuously on parameters.

\begin{proposition}\label{prop:peripheral_to_total}
    Suppose $\Pi: E \to B$ is a PANLF, $\mathcal{U} \subset B$ is an open neighborhood of $\partial B$, $X$ is a compact cell complex, $A \subset X$ is a subcomplex, and
    \begin{align*}
        &X \to \mathcal{J}(\Pi;\mathcal{U}): \tau \mapsto (J_\tau, \partial_\theta^\tau),\\
        &A \to C^\infty(E): \tau \mapsto f_\tau
    \end{align*}
    are continuous maps such that for every $\tau \in A$, 
    \begin{itemize}
    \item $(J_\tau, f_\tau)$  is an almost Stein structure
    \item $f_\tau$  is constant on all boundary components of fibers, and
    \item $\lambda_\tau \coloneqq -df_\tau \circ J_\tau$  restricts to $\partial E$ as a Giroux form for  $\partial\Pi$ with $\lambda_\tau(\partial_\theta^\tau) \equiv$ constant.
    \end{itemize}

    Then there exists a continuous family of almost Stein structures $\{(J_\tau',f_\tau')\}_{\tau \in X}$ matching $(J_\tau, f_\tau)$ for all $\tau \in A$ such that $J_\tau'$ is $C^\infty$ close to $J_\tau$ for all $\tau \in X$.
\end{proposition}

To help the reader and our exposition, we summarize the steps from \cite[\S 3]{spinal_I} to prove the version for bordered Lefschetz fibration in the following:
    \begin{enumerate}[\bf Step 1:]
        \item Find the complex structures $\{j_\tau\}_{\tau\in X}$ which are uniquely determined by requiring
        $$\Pi:(E, J_\tau) \to (B, j_\tau)$$
        to be holomorphic in $E|_{\cU}$ and $\partial_h E$ for all $\tau$. This is to satisfy (3) in Definition \ref{def: peripheral_acs}.
        \item Choose a nice $j_\tau$-convex function $\varphi$ on $B$ for all $\tau$ at once. This is done in \cite[Lemma 3.5]{spinal_I}.
        \item Define $\sigma_\tau := -d\varphi \circ j_\tau$, for $\tau \in X$.
        \item Extend the $(J_\tau)_{\tau \in A}$-convex functions $(f_\tau)_{\tau \in A}$ to fiberwise $(J_\tau)_{\tau \in X}$-convex functions $(f_\tau)_{\tau \in X}$. This is done in \cite[Lemma 3.12]{spinal_I}.
        \item\label{item:crit} For $\tau \in X$, define $J^{\epsilon}_\tau$ on $N_{\crit}$ such that $(\phi \circ \Pi)|_{N_{\crit}}$ is $J^{\epsilon}_\tau$-convex. This is done in \cite[Lemma 3.8]{spinal_I}.
        \item For $\tau \in X$, define a family of smooth one-forms $\eta^{\epsilon}_\tau$ on $N_{\crit}$ so that \[\xi^{\epsilon}_\tau \coloneqq \ker(dF) \cap \ker(-dF \circ J^{\epsilon}_\tau) = \ker(dF) \cap \ker(\Pi^*\sigma_\tau + \epsilon\eta^{\epsilon}_\tau)\] on $N_{\crit}$. This is done in \cite[Lemma 3.11]{spinal_I}.
        \item Extend $\eta^{\epsilon}_\tau$ to be defined over $E$, so that near $\partial_hE$ and on $E|_{\cU}$, \[\eta^{\epsilon}_\tau = \lambda_\tau \coloneqq -df_\tau \circ J_\tau\] Also, $d\eta^{\epsilon}_\tau|_{VE} >0$. In particular, for any point away $N_{\crit}$ and $E|_{\cU}$, $\eta^{\epsilon}_\tau$ can be chosen (independent of $\epsilon$) to be a 1-form that matches $\lambda_\tau$ near $\partial_h E$ and whose exterior derivative is positive on $VE$. Locally, this choice can be the same as the one made in Proposition~\ref{prop: fiberwise_almost_Stein}, which in turn used the choice made in \cite[Proposition 2.25]{spinal_I} for regular and singular neighborhoods. This is done in \cite[Lemma 3.13]{spinal_I}.
        \item Now define
        \[
        F^{\epsilon}_\tau := \varphi \circ \Pi + \epsilon(\beta \circ \Pi)f_\tau
        \]
        where $\beta$ is a bump function on $B$, supported on $U$ and 1 near $\partial B$ and $\Crit(\varphi)$.  
        \item Define $\Theta^{\epsilon}_\tau := \Pi^*\sigma_\tau + \epsilon\eta^\epsilon_\tau$, which is globally defined on $E$. This also allows us to extend the definition of $\xi_\tau^\epsilon$ to now $\xi^{\epsilon}_\tau := \ker dF^{\epsilon}_\tau \cap \ker\Theta^{\epsilon}_\tau$ on $E$.
        \item\label{item:acs} Extend $J^{\epsilon}_\tau$, defined in Step~\ref{item:crit} on $N_{\crit}$, over $E$ so that $J^{\epsilon}_\tau = J_\tau$ on $E|_{\cU}$. In particular, on an open set $E^{\reg}$ with closure disjoint from $N_{\crit} \cup E|_{\cU}$, $J^{\epsilon}_\tau$ is defined as follows. Choose $g_\tau$ a $J_\tau$-invariant metric on $E$ and let $H_\tau E^{\reg}$ be the subbundle of $TE^{\reg}$ which is the $g_\tau$-orthogonal complement of $VE|_{E^{\reg}}$. Then the projection $TE^{\reg} \to VE$ along $H_\tau E^{\reg}$ restricts to a bundle isomorphism $\Psi_\tau: \xi^{\epsilon}_\tau \to VE$. On $E^{\reg}$, $J^{\epsilon}_\tau$ is defined as:
        \[
        J^{\epsilon}_\tau|_{\xi^{\epsilon}_\tau}=\Psi^*J_\tau|_{VE}, J^{\epsilon}_\tau|_{H_\tau E^{\reg}} = J_\tau|_{H_\tau E^{\reg}}
        \]
        \item After replacing $\phi$ by a function $h \circ \phi$ for a sufficiently convex function $h$, the pairs $(J^{\epsilon}_\tau, F^{\epsilon}_\tau)$ become almost Stein structures for all $\tau \in X$ and all $\epsilon>0$ sufficiently small. This is done in \cite[Lemma 3.14, 3.15 and 3.16]{spinal_I}. 
        \item Let $A' \supset A$ be the subset of $X$ such that for $\tau \in A'$, $(J_\tau,f_\tau)$ is an almost Stein structure. Choose a bump function $\rho: X \to [0,1]$ which is 1 on $A$, with support on $A' \supset A$. Then define 
        \[
        f^{\epsilon}_\tau = \rho(\tau)f_\tau + [1-\rho(\tau)]F^{\epsilon}_\tau
        \]
        The pairs $(J^{\epsilon}_\tau, f^{\epsilon}_\tau)$ for all $\epsilon > 0$ are almost Stein structures as required. This is done in \cite[Lemma 3.17 and 3.18]{spinal_I}. 
\end{enumerate}

\begin{proof}[Proof of Proposition \ref{prop:peripheral_to_total}]
    The proof of Proposition~\ref{prop:peripheral_to_total} will follow from the proof of \cite[Proposition 3.4]{spinal_I} outlined above, with an extra ingredient near the exotic fibers. In other words, Steps 1 - 3 can be done as above, and Step 4 can be also completed using necessary modifications near the exotic fibers as done in the proof of Proposition~\ref{prop: fiberwise_almost_Stein}. This gives us a family of fiberwise almost Stein structures $(f_\tau, J_\tau)$ on $E$, for all $\tau \in X$.
        
    Given an exotic fiber at $p \in B_{\exot}$ and a neighborhood $c \in B_p \subset \mathring{B}$, we consider a neighborhood $V_p \subset \Pi^{-1}(B_p)$. We can then repeat the steps (5 - 11) above to produce $J_\tau^{\epsilon}$-convex functions $F_\tau^{\epsilon}$ away from these exotic neighborhoods, such that $(J_\tau^{\epsilon},F_\tau^{\epsilon})$ satisfy the required properties on the boundary of $E \setminus (\cup_{p \in B_{\exot}} V_p)$. Suppose $(J_\tau^{\epsilon},F_\tau^{\epsilon})$ are defined on an open set $E^R$ that covers $E \setminus (\cup _{c \in B_{\exot}} V_p)$.
        
    We then show how to extend the definition of $F^{\epsilon}_\tau$ and $J^{\epsilon}_\tau$ over $V_p$ so that the pair $(J^{\epsilon}_\tau,F^{\epsilon}_\tau)$ define almost Stein structures on $E$. Firstly, we can assume that on the intersection $E^R \cap V_p$, $J^{\epsilon}_\tau$ is defined as on $E^{\reg}$ in Step~\ref{item:acs}, and $F^{\epsilon}_\tau = \varphi \circ \Pi$. Now recall the definition of $U_{\exot}^{\delta} \subset A_{\delta} \times D^2$, and describe $V'_p \subset V_p$ to be identified with the subset of $U_{\exot}^{\delta}$ where $f_{\exot} \neq \varphi \circ \Pi$. 
        
    We extend $J^{\epsilon}_\tau$ to be equal to $J_\tau$ on the interior of $V'_p$, and extend $F^{\epsilon}_\tau$ to equal $f_{\exot}$ on $V_p$, appropriately using bump functions as necessary. Similar to the argument in Proposition~\ref{prop: fiberwise_almost_Stein}, since the $J_\tau$ agree with $J_{\exot}$ on the intersection of the exotic fiber with $\partial_h E$, an appropriate choice of $K$ and $\delta$, and $\epsilon$ sufficiently small, will ensure $F^{\epsilon}_\tau$ is $J^{\epsilon}_\tau$-convex everywhere on $E$, and that it satisfies the appropriate conditions on $\partial E$.

    Now, we can replicate Step 12, by interpolating between the functions $f_\tau$ and $F^{\epsilon}_\tau$ to obtain the almost Stein structures $(J'_\tau, f'_\tau)$ for $\tau \in X$, matching the almost Stein structures $(J_\tau, f_\tau)$ for $\tau \in A$. The extra step in this situation, compared with the computations in \cite[\S 3]{spinal_I} is that in $\partial_hE$, we have to check that the interpolated pairs $(J'_\tau, f'_\tau)$ as defined above restrict to contact forms in a neighborhood of the exotic fibers. The computation will be similar to that in \cite[Lemma 3.17]{spinal_I}.

    This completes the proof of Proposition~\ref{prop:peripheral_to_total}.
\end{proof}

\subsection{Quasiflexibility}
In this subsection we prove Theorem \ref{thm: quasiflexibility}. This is essentially a consequence of Proposition \ref{prop:peripheral_to_total} and \cite[Proposition 6.4]{spinal_II} which analyzes the moduli space under deformation of $1$-parameter data.

In order to work with the moduli space of finite-energy pseudoholomorphic curves, we need to ``complete'' the compact bordered Lefschetz fibrations by attaching cylindrical ends. Given a bordered Lefschetz fibration $E$ whose boundary has a decomposition $\partial E = \partial_h E \cup \partial_v E$, we can associate to it a symplectic ``double completion'' by attaching cylindrical ends to both the fibers and base in a compatible way. The details of the construction are given in \cite[\S 3]{spinal_II} in the setting of stable Hamiltonian structures and perturbed to the contact case in \cite[\S 4]{spinal_II}. See also \cite[Lemma 6.34]{spinal_II} and the discussion afterwards. The planar assumption is used in the step when we perturb to the contact setting, but still want the pseudoholomorphic foliation to persist.
Here, we summarize the definition and properties into the following definition-theorem. 

\begin{definitiontheorem}[Double completion for planar nearly Lefschetz fibrations]
\label{def:double_completion_LF}
    Fix a planar nearly Lefschetz fibration $(E, \Pi)$ together with an almost complex structure $J\in \cJ(\Pi)$. Then, there exist a contact-type\footnote{In \cite{spinal_II}, the notion of \emph{weakly contact-type} is used, which in particular forgoes the local exactness condition near the hypersurface. Since we are only concerned with strong symplectic (and mostly Stein) fillings, the weakly contact hypersurface can be assumed to be contact.} smoothing $M$ of $\partial E$ and a \textbf{double completion} $(\widehat{E}, \widehat{J})$, where topologically $\widehat{E} \cong E \cup [0,\infty)\times M$ and $\widehat{J}$ is an almost complex structure which is 
    \begin{enumerate}
        \item Part of an almost Stein triple, i.e. there exists a $\widehat{J}$-convex function $\widehat{f}$ such that $\widehat{\lambda}:=-d\widehat{f}\circ \widehat{J}$ is a Liouville form on $\widehat{E}$;
        \item $\widehat{J}|_{\widehat\cN_-(\partial E)} \equiv J_+$, where $\widehat\cN_-(\partial E)$ is a neighborhood of infinity in $\widehat\cN_-(\partial E) \subset \widehat{E}$ which can be identified with half-symplectization $([0,\infty)\times M, J_+)$ for an almost complex structure $J_+$ which is $\lambda$-adapted, i.e. $\bR$-invariant, $J_+(\partial r) = R$ and $J_+(\xi) = \xi$, where $R$ is the Reeb vector field of $\lambda$ and $\xi = \ker \lambda$ ;
        \item $[0,\infty)\times M$ admits an embedded $J_+$-foliation whose leaves project to pages of $\Pi|_{\partial_v E}$.
        \item The restriction $\widehat{J}|_{E} \equiv J$.
        \item $\widehat{J}$ is $\widehat{\lambda}$-adapted.
    \end{enumerate}
\end{definitiontheorem}

\begin{definition}
\label{def: compatible homotopy}
    A $1$-parameter family of almost complex structure $\{\widehat{J}_\tau\}_{\tau\in [0,1]}$ on $(\widehat{E}, \widehat{J})$ is a \textbf{compatible homotopy} if the following holds:
    \begin{itemize}
        \item $\widehat{J}_{\tau}|_{\widehat\cN_-(\partial E)} \equiv J_+$ for all $\tau$;
        \item $\widehat{J}_{\tau}$ is $d\widehat\lambda_\tau$-compatible for all $\tau$;
        \item $\widehat{J}_0 = \widehat{J}$;
        \item $\widehat{J}_1$ and the homotopy $\{\widehat{J}_{\tau}\}$ are both generic in the interior of $E$.
    \end{itemize}
\end{definition}

Under these conditions, the following theorem produces a $1$-parameter family of $\widehat{J}_\tau$-holomorphic curves which deforms smoothly under symplectic deformations. In particular, the finite numbers of singular and exotic fibers deform smoothly and we have a $1$-parameter family of nearly Lefschetz fibrations.

\begin{theorem}[{\cite[Proposition 6.4]{spinal_II}}]
\label{thm: spinal_ii_6_4}
    For a compatible homotopy of $\{\widehat{J}_\tau\}_{\tau\in [0,1]}$, there exists a homeomorphism
    \begin{equation*}
        \widehat{\cM}^{\cF}(\widehat{J}) \times [0,1] \to \widehat{\cM}^{\cF}(\{\widehat{J}_\tau\})
    \end{equation*}
    inducing a homotopy of nearly Lefschetz fibrations.
\end{theorem}

\begin{proof}[Proof of Theorem \ref{thm: quasiflexibility}]
    Let $\Pi: E \to B$ be a nearly Lefschetz fibration with genus-zero fibers. Let $(J_0, f_0)\in \cJ_{\Stein}(\Pi)$ be an almost Stein structure on $E$ supported by $\Pi$. Now, let $(J_1, f_1)$ be another almost Stein structure, such that $\omega_1:=-df_1 \circ J_1$ is homotopic to $\omega_0:=-df_0 \circ J_0$ as symplectic structures convex at infinity after appropriate smoothing at corners.

    Now, we may associate to both almost Stein structures their double completions $(\widehat{E}, \widehat{J}_0, \widehat{f}_0)$ and $(\widehat{E}, \widehat{J}_1, \widehat{f}_1)$ such that there is a compatible homotopy $\{\widehat{J}_\tau\}_{\tau\in [0,1]}$ between them induced by the homotopy of symplectic forms $\{\widehat\omega_\tau\}_{\tau\in[0,1]}$ matching $d\widehat\lambda_\tau$ at $\tau = 0,1$ after further deformation as established in the proof of \cite[Theorem 1.13]{spinal_II}. As usual, here $\widehat\lambda_\tau:= -d\widehat f_\tau \circ \widehat J_\tau$. Now, by Theorem \ref{thm: spinal_ii_6_4}, we have a homotopy of nearly Lefschetz fibrations with projection maps $\Pi_\tau:\widehat{E}_\tau \to B_\tau$. Now, since $(\widehat{J}_1, \widehat{f}_1)$ is supported hence peripherally supported by $\Pi_1$, it is also peripherally supported by $\Pi_0$ as they share the same information near the contact boundary. Finally, by Theorem \ref{thm: peripheral_asc}, we know that $(\widehat{J}_1, \widehat{f}_1)$ is supported by $\Pi_0$ up to almost Stein homotopy. Now, for the fixed projection $\Pi_0$, we know that the space of supported almost Stein structure $\cJ_{\Stein}(\Pi_0)$ is compatible by Proposition \ref{prop: contractible_almost_Stein}. Therefore, $(\widehat{J}_0, \widehat{f}_0)$ and $(\widehat{J}_1, \widehat{f}_1)$ are almost Stein homotopic.    
\end{proof}

\section{A Weinstein structure on a PANLF} \label{subsec:panlf_weinstein}
In this section, we construct a specific Weinstein structure supporting a given PANLF in dimension four, in the spirit of \cite[\S 8.2]{gompf1999manifolds}. The results in \S \ref{sec:PANLF_stein} imply that any compatible Weinstein structure is homotopic to this one. This provides a concrete model for applications, for example in the proof of Theorem \ref{thm: regular_lag_disk} given in \S \ref{sec: proof_of_main}.

Recall that for a bordered Lefschetz fibration, a Weinstein structure can be described by starting with a trivial fibration with a subcritical Weinstein structure, and then extending over neighborhoods of the singular fibers via Weinstein $2$-handle attachments. 
Now, let $W$ be supported by a nearly Lefschetz fibration $\Pi$ with fiber $S$ and base $B$, both oriented surfaces with boundary.
By assumption, the critical values of $\Pi$ are isolated, the exotic fibers are isolated, and the images of exotic fibers are distinct from the critical values. We use $B_{\exot}, B_{\crit} \subset B$ to denote the set of exotic and critical values, respectively. 
A particular Weinstein structure on $W$ can be described as follows. (Note that the description around the regular and singular fibers is standard, while the handle attachment around exotic fibers is the novel aspect.) As in \S \ref{subsec: exotic_weinstein}, we use the model $z_2^2+z_1$ for neighborhoods the exotic fibers.

\begin{enumerate}
    \item Choose a Weinstein structure on $S$ given by a standard ``height'' function $f_S$ with a single index $0$ critical point and $1-\chi(S)$ index $1$ critical points, and likewise a Weinstein structure on $B$ given by a function $f_B$ with a single index $0$ critical point and $1-\chi(B)$ index $1$ critical points. 
    
    \item Pick a finite open cover with open disks $\mathcal{U} = \cup_x U_x$ of $B$ such that for every critical value $q \in B_{\crit}$, there is a unique open set $U_q$ containing $q$. Likewise, we assume that for any image $p \in B_{\exot}$ of an exotic fiber, there is a unique open set $U_p$ containing $p$. We ensure that $\cup_{x \notin B_{\crit}\cup B_{\exot}} U_x =: \cU_1 \cup \cU_2$ is connected, and $\mathcal{U}_1\subset B$ is homeomorphic to $B$. See Figure \ref{fig:panlf_stein}. 
    
\begin{figure}[htbp]
    \vspace{0.4cm}
    \begin{overpic}[width=0.9\textwidth]{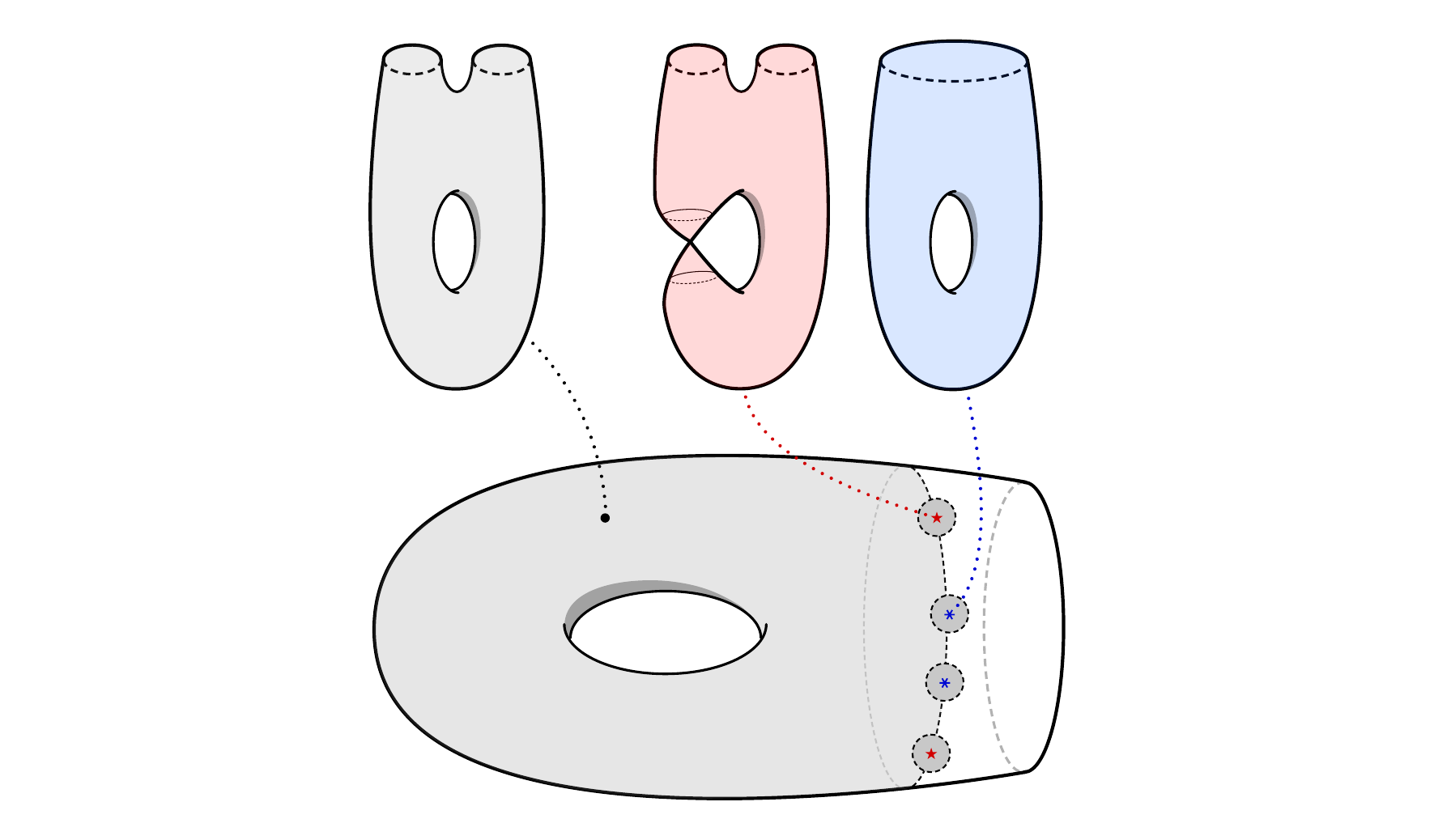}
        \put(85,50){$\Sigma$}
        \put(175,25){\small $\mathcal{U}_1$}
        \put(236.5,58){\tiny $\mathcal{U}_p$}
        \put(233.5,78){\tiny $\mathcal{U}_q$}
        \put(109,219){\scriptsize \textcolor{black}{regular}}
        \put(184,219){\scriptsize \textcolor{darkred}{singular}}
        \put(245,219){\scriptsize \textcolor{darkblue}{exotic}}
    \end{overpic}
    \caption{Constructing a Weinstein structure on a PANLF using local models around regular fibers, singular fibers, and exotic fibers.} 
    \label{fig:panlf_stein}
\end{figure}
    
    \item \emph{Regular fibers.} Note that $\Pi^{-1}(\mathcal{U}_1)$ is diffeomorphic to $B \times S$. Endow this piece with the product Weinstein structure using the function $f_0 = f_S + f_B$. Since $\cU_1$ contains no critical values or exotic values, we can ensure that $f_0$ is arbitrarily close to the real part of $\Pi$ in some choice of coordinates over the rest of $\Pi^{-1}(\mathcal{U}_1)$ and has a Liouville vector field positively transverse to the boundary of $W$. We can ensure these choices of coordinates agree with the choices made for the local models near critical and exotic fibers. 

    \item \emph{Singular fibers.} Recall that near any critical point $c$ of $\Pi$, there are complex local coordinates $(z_1,z_2)$ on a neighborhood $V_c \subset W$ such that $\Pi|_{V_c}: V_c \to B$ takes the form $z_1^2+z_2^2$. Consider the function $f_q = x_1^2+x_2^2-y_1^2-y_2^2$ in these coordinates, where $z_j = x_j+iy_j$.
    Given any critical value $q$ of $\Pi$, a critical point $c$ such that $\Pi(c) = q$, and an open set $U_q$ containing $q$, write $\Pi^{-1}(U_q) = V_c \cup V'_c$ where $V'_c$ is also open and avoids $c$. We can choose $U_q$ so that $U_q \subset \Pi(V_c)$, which in particular ensures that $V'_c$ is a trivial $(S - A)$-bundle over the disk $U_q$, where $A$ is an annulus neighborhood of the vanishing cycle corresponding to the critical value $q$.
    Since $f_q = \Rea(\Pi)$ and $q$ is in the interior, we can arrange for $f_q$ to match up with $f_0$ directly on $V'_c \cap \Pi^{-1}(\mathcal{U}_1)$. 
    We thus extend $f_0$ to a Weinstein structure on $\Pi^{-1}(\mathcal{U}_1)$ to $\Pi^{-1}(\mathcal{U}_1 \cup \bigcup_{q \in B_{\crit}} U_q)$. 
    
    \item \emph{Exotic fibers.} Now, near any $p \in B_{\exot}$, we can again find an open set $U_p$ and observe that $\Pi^{-1}(U_p) = V_d \cup V_d'$, where $d \subset V_d$ is a neighborhood meeting the boundary where $\Pi$ takes the form $\Pi|_{V_d}: \C - B \to \C$, $\Pi(z_1,z_2) = z_2^2+z_1$, with $V_d \cong (\C - B) \times \C$, such that $B$ is a ball around the origin.
    Similarly as above, $V_d'$ is a trivial $S - P$ bundle over the disk, where $P$ is a pair of pants.  On $V_d$, define the function $f_p = f_A+x_2^2-y_2^2$ and observe that on $V_d \cap \Pi^{-1}(\mathcal{U}_1)$, the function $f_p$ agrees with $\Rea(\Pi)$. Note that $f_p$ has a single index two critical point inside $V_d$. We can thus extend $f_0$ over $\Pi^{-1}(U_p)$, so that the Weinstein structure is defined on $\Pi^{-1}(\mathcal{U}_1\cup \bigcup_{z \in B_{\crit}\cup B_{\exot}}U_z)$. 
    
    \item Once the above steps are done for every critical and exotic point, the Weinstein structure can be extended with no critical points over the trivial cobordism described by $\Pi^{-1}(\mathcal{U}_2)$.

    \item By Lemma \ref{lem: contact_lambda}, the Weinstein structure above induces a contact form on the boundary which is a Giroux form for the spinal open book induced on the boundary from the PANLF. 
\end{enumerate}

\begin{figure}
    \centering
    \includegraphics[width=\linewidth]{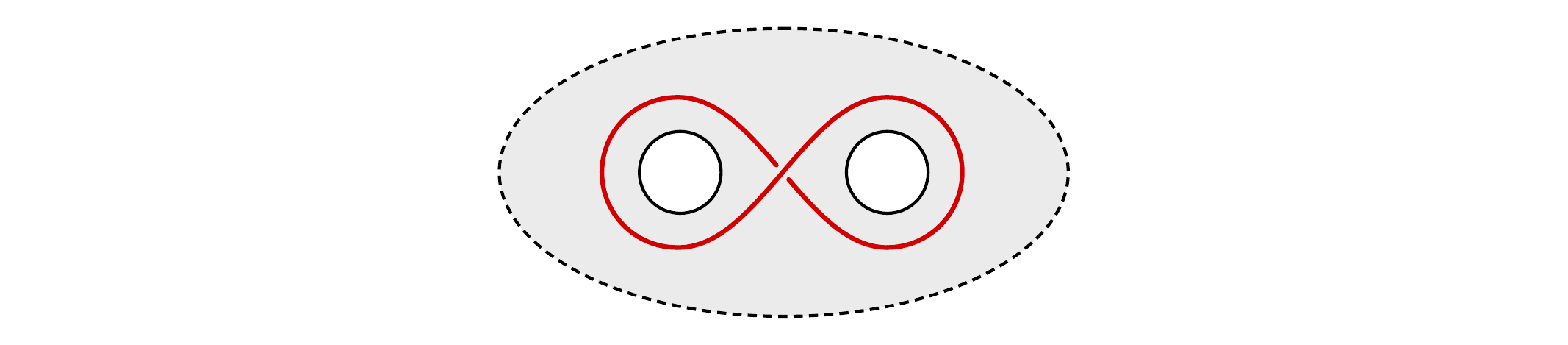}
    \caption{The attaching curve for the Weinstein $2$-handle cobordism corresponding to an exotic fiber, living in a neighborhood of a page of the spinal open book.}
    \label{fig: exotic-cobordism}
\end{figure}

Next, we examine the Weinstein structure near an exotic fiber as described in (5) more carefully. Recall that a right-handed Dehn twist around a vanishing cycle (the monodromy of a singular fiber) is accomplished by attaching a Weinstein handle along the vanishing cycle. Indeed, since the contact framing agrees with the page framing, a Weinstein handle induces a $-1$ surgery framing relative to the pages. In contrast, the monodromy associated to an exotic fiber is a \emph{boundary interchange} map, i.e.\ a right-handed Dehn twist along an arc connecting two boundary components of the page.

Before stating the lemma, we introduce some language that will be used throughout the rest of the article. If $(P, \lambda)$ is an exact symplectic manifold and $\Lambda \subset (P\times \R, dz + \lambda)$ is a Legendrian link in the contactization, we call the projection of $\Lambda$ to $P$ the \emph{Lagrangian} projection, generalizing the usual Lagrangian projection when $P = \R^2$ and $\lambda = -y\, dx$.

\begin{lemma}\label{lemma:NLF_2handle}
    Let $P$ be a pair of pants neighborhood of the arc along which a boundary interchange is performed. The boundary interchange mapping class corresponds to a Weinstein cobordism between the relevant contact manifolds, obtained by attaching a Weinstein $2$-handle along a Legendrian knot in $P \times \R$ whose Lagrangian projection to $P$ is given by Figure~\ref{fig: exotic-cobordism}. 
\end{lemma}

\begin{remark}
To clarify the framing of the handle in the statement of the lemma: If the inner boundary components are capped off so that the page is (locally) a disk, the attaching sphere in Figure \ref{fig: exotic-cobordism} is an unknot with writhe $+1$. The writhe computes the blackboard framing in the figure, which is the the page (and thus contact) framing. A Weinstein handle attached along the knot would then have framing $+1 + (-1) = 0$. For a different perspective on Lemma \ref{lemma:NLF_2handle}, we refer the reader to the explicit piecewise-smooth models of the attaching sphere described later in Lemma \ref{lemma:algo1}. 
\end{remark}

\begin{proof} 
    We identify the attaching sphere of the $2$-handle corresponding to an exotic fiber with the knot in Figure~\ref{fig: exotic-cobordism}. 
    In the local model of the exotic fiber, the attaching sphere is the boundary of the stable manifold of the Liouville vector field associated to the plurisubharmonic function $f_{\exot} \coloneqq f_A+x_2^2-y_2^2$ and the standard complex structure $z_j = x_j + iy_j$. Recall that the gradient vector field of $f_A$ on the annulus $A = \{\delta < |z_1| < 1\}$ is depicted in Figure \ref{fig:annulus}. The stable manifold of $f_{\exot}$ can thus be parametrized by a smoothing of 
    \begin{equation}\label{eq:param-stable}
        \{z_1 = e^{i\pi t/2}, z_2 = is, s \in [-2,2], t \in [-1,1]\}.
    \end{equation}
    Here, $z_1 = e^{i\pi t/2}$ corresponds to the approximate semi-circular stable manifold in Figure \ref{fig:annulus}, and $z_2 = is$ is the stable manifold of $x_2^2-y_2^2$. A piecewise-smooth description of the attaching sphere is therefore the union of the four arcs in \eqref{eq:param-stable} corresponding to the boundary of the parameter domain $[-2,2]_s \times [-1,1]_t$.
    
    The fibers of the PANLF, near the exotic fiber, are locally parametrized as $u(c,z) = (-c(z^2-1),\sqrt{c}z)$, where $z \in (\C \setminus \{\pm 1\})$ is a coordinate on the page and $c$ is a coordinate on the base of the PANLF. We can therefore pull back the relevant curves to the fibers, using $c = z_1+z_2^2, z = z_2 / \sqrt{c}$. 
    In particular, the projection to $P$ is determined by
    \[
    z = \frac{z_2}{\sqrt{c}} = \frac{is}{\sqrt{e^{\tfrac{i\pi t}{2}}- s^2}}.
    \]
    Using a branch of the square-root function so that $\mathrm{Im}\sqrt{c}>0$ and plotting the four arcs corresponding to $s=\pm 2, t=\pm 1$ yields Figure \ref{fig:PLplot}.

\begin{figure}[ht]
    \centering
    \begin{overpic}[width=0.9\linewidth]{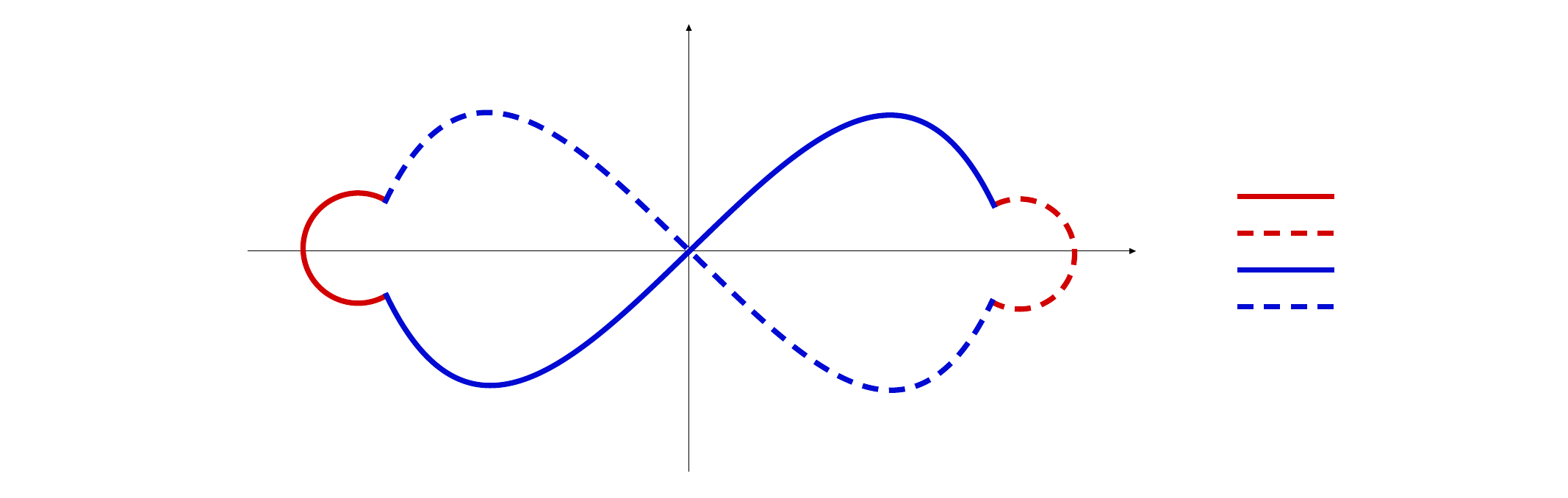}
        \put(280.5,58.5){\tiny $x$}
        \put(166.5,120){\tiny $y$}
        \put(332,72){\tiny \textcolor{darkred}{$s=-2$}}
        \put(332,62.5){\tiny \textcolor{darkred}{$s=+2$}}
        \put(332,53.5){\tiny \textcolor{darkblue}{$t=+1$}}
        \put(332,44){\tiny \textcolor{darkblue}{$t=-1$}}
	\end{overpic}
    \caption{The piecewise-smooth plot of the four components of the attaching sphere, projected to $P$ with coordinate $z = x+ iy$.}
    \label{fig:PLplot}
\end{figure}

    Finally, to verify the under crossing data at the double point of the intersection in the projection of Figure \ref{fig:PLplot}, note that the value of $c$, the coordinate in the base of the fibration, is $c = e^{\pm i \pi/4}$ when $s=0$ and $t=\pm 1$. Thus, the $t=+1$ strand is more counterclockwise in the base and thus above the $t=-1$ strand in the contactization of the page.
\end{proof}

\section{Lagrangians and their complements}
\label{sec: proof_of_main}

In this section, we prove Theorem \ref{mainthm: regular_complement} in \S \ref{subsec: reg_lag_complement}. This is in all dimensions. 
We then set up a preliminary lemma in \S \ref{subsec: symp_complement_to_lag_complement_lemma} and prove Theorem \ref{thm: regular_lag_disk} in \S \ref{subsec: proof_of_main_thm} regarding the regularity of Lagrangian disks in $4$-dimensional Weinstein domains.

\subsection{Proof of Theorem~\ref{mainthm: regular_complement} on complements of regular Lagrangians}
\label{subsec: reg_lag_complement}
In this subsection, given a regular Lagrangian $L$ built from $0$- and $1$-handles with a coupled Weinstein structure on $W$ compatible with $L$, we show that $W \setminus \nu(L)$ is a Weinstein domain.
The main content here is the generalization to Weinstein domains of a result in smooth Morse theory, which states that for a properly embedded submanifold $M^n \subset N^{2n}$, the complement $N \setminus\nu(M)$ gains a $(n-1+m)$-handle for every $m$-handle of $M$ (this is a particular case of \cite[Proposition 6.2.1]{gompf1999manifolds}). We will prove the Weinstein analogue for $L \subset W$, assuming $m = 0,1$, beginning with the following special case. Let $X_L$ denote the complement of the $0$-section in $T^*L$.

\begin{proposition}\label{prop:complement_cotangent }
 Let $L$ be a connected $n$-dimensional smooth manifold and consider its cotangent bundle $T^*L$ with the standard Weinstein structure. Then, $X_L$ admits a Weinstein structure. Furthermore, $T^*L$ is symplectomorphic to a Weinstein sub-domain of $X_L$.
\end{proposition}

\begin{proof}
      Consider a handle structure on $L$ with a single $0$-handle and $k$-many $1$-handles. We claim that $X_L$ is symplectomorphic to a Weinstein handlebody $W_1$ given by 
    \begin{enumerate}
        \item starting with $T^*L$,
        \item attaching an $(n-1)$-handle $h^{n-1}$,
        \item then attaching $k$ many $n$-handles $h^n_1,\dots,h^n_k$.
    \end{enumerate}
    \noindent Let $\mathcal{W}_1$ denote a Weinstein structure on $T^*L$ given by a single $0$-handle and $k$-many $1$-handles.
    We can modify it to a coupled Weinstein structure $\mathcal{W}_2$, homotopic to $\mathcal{W}_1$, on $T^*L$ such that $L$ can be described in terms of the Weinstein handles of $\mathcal{W}_2$ \cite{EliashbergGanatraLazarev2018Regular}. Specifically, given $T^*D^n$ which is Weinstein homotopic to a $0$-handle, following the proof of \cite[Proposition 2.3]{EliashbergGanatraLazarev2018Regular} we can attach a canceling pair of $(n-1)$- and $n$-handles, denoted as $h^{n-1}$ and $h^n$, and homotope the Weinstein structure and obtain the zero section $D^n$ as the cocore of the $n$-handle. Now, attach coupled Weinstein $1$-handles to build $\mathcal{W}_2$. Call these $1-$handles $h_1^1, \dots, h_k^1$. Note that in $\mathcal{W}_2$, the Lagrangian $1$-handles inside the coupled $1$-handles $h^1_j$ are attached along the boundary of the cocore of the $n$-handle $h^{n}$. The homotopy from $\mathcal{W}_1$ to $\mathcal{W}_2$ is thus given by the described pair of canceling $(n-1)$- and $n$-handles and handle-sliding the $1$-handles to lie on the belt sphere of the $n$-handle.
    
     \begin{figure}
        \centering
        \includegraphics[width=0.5\linewidth]{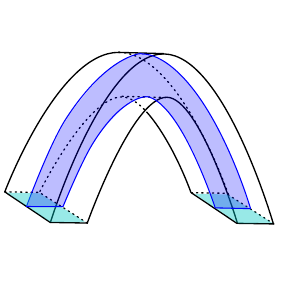}
        \caption{A schematic of a coupled $1$-handle. The shaded blue region is the Lagrangian 1-handle $\{y_2=\dots=y_n=x_1=0\}$, while the two shaded quadrilaterals represent the attaching region of the coupled handle.}
        \label{fig:coupled_1_handle}
    \end{figure}

  First, we establish that that the complement of the Lagrangian core of a Weinstein $1$-handle is symplectomorphic to a coupled ($1$, $n$)-handle pair. The Lagrangian core, or the Lagrangian $1$-handle, is given by $\{y_2=y_3=\dots=y_n=x_1=0\}$, inside the coupled Weinstein $1$-handle $D^1 \times D^{2n-1}$ with coordinate $y_1$ on $D^1$ and $(y_2,\dots,y_n,x_1,\dots,x_n)$ on $D^{2n-1}$ (see \cite[Definition 2.1]{brw1}). We denote the complement of the thickened Lagrangian core by $R_L$, as follows.
    \[
    R_L \coloneqq \{(y_1,y_2,\dots,y_n,x_1,x_2,\dots,x_n) \in D^1 \times D^{2n-1} \mid \sum_{j=2}^n y_j^2+x_1^2 \geq \epsilon\}
    \]
    The symplectic form on $R_L$ in these coordinates is $\sum dx_jdy_j$, restricted from the standard form on the Weinstein $1$-handle.

    Now, consider the following pair of Weinstein $1$- and $n$-handles
    \[
    R_{\std} \coloneqq \{(p_1,p_2,\dots,p_n,q_1,\dots,q_n) \in D^1 \times S^{n-1} \times D^n\}
    \]
    where the coordinates $(p_2,\dots,p_n,q_1)$ belong to $S^{n-1}\times D^1$.  We can see that
    \[R_{\std} = H_1\cup H_n, \]
    where
    \[ H_1 = \{(p,q) \in R_{\std} \mid p_2 \leq 0\} \text{ and } H_n = \{(p,q) \in R_{\std} \mid p_2 \geq 0\}.
    \]
     The $1$-handle $H_1 = D^1 \times D^{n-1} \times D^n$ is attached along $\partial D^1 \times D^{n-1} \times D^n$, while the $n$-handle $H_n = D^1 \times D^{n-1} \times D^{n}$ is attached along $\partial(D^1 \times D^{n-1}) \times D^n$. The symplectic form on $H_1$ in these coordinates is given by $dp_1(-dq_1)+(-dp_2)dq_2+ \sum dp_idq_i$, while on $H_n$ it is $\sum dp_idq_i$. These are standard Weinstein handles, so that $R_{\std}$ is actually a union of Weinstein handles. Now observe that $R_L$ is a disk neighborhood of the Lagrangian annulus $D^1 \times S^{n-1}$ spanned by the directions $(x_1,y_2,\dots,y_n)$, and similarly $R_{\std}$ is a disk neighborhood of the Lagrangian annulus spanned by the directions $(q_1,p_2,\dots,p_n)$. It follows by the standard Lagrangian neighborhood theorem that $R_{\std}$ and $R_L$ are symplectomorphic, since they are both symplectomorphic to the unit cotangent bundle $\bD^*(D^1 \times S^{n-1})$. Further, the Liouville vector field on the attaching region of $R_{\std}$ is the same as the attaching region of $R_L$. It follows that there is a Weinstein structure on $R_L$ that extends the Weinstein structure on $T^*L$, by attaching an $n$-handle to the corresponding $1$-handle in $\cW_1$. For the coupled Weinstein $1$-handle $h^1_j$, denote this $n$-handle as $h^n_j$.

     Now, removing the zero-section from $T^*L$ thus amounts to modifying the symplectic structure associated to $\mathcal{W}_2$ by removing the cocore of $h^n$, which is equivalent to removing $h^n$, and drilling out the insides of $h^1_j$. This symplectic manifold can be given a Weinstein structure by starting with $\mathcal{W}_1$, attaching a Weinstein $(n-1)$-handle $h^{n-1}$, and $k$-many $n$-handles $h^n_1,\dots,h^n_k$. 
\end{proof}
     
\begin{proof}[Proof of Theorem~\ref{mainthm: regular_complement}]
    Since $L$ is regular in $W$, we can write
    \[
    W \simeq T^*L \cup W_L
    \]
    where $W_L$ is a Weinstein cobordism. In particular, on the contact manifold $\partial(T^*L)$, there is a configuration of isotropic submanifolds that describe the attaching region of the handles in $W_L$. The contact boundary of $\partial X_L$ is obtained by modifying the contact structure on $\partial T^*L$ in a neighborhood of the Legendrian $\partial L$. Since the attaching isotropics of the handles in $W_L$ can be made disjoint from $\partial L$, and consequently these neighborhoods, we can construct $W$ as a Weinstein domain. 
\end{proof}

\subsection{Exact and regular Lagrangians} 
\label{subsec: symp_complement_to_lag_complement_lemma}

The work of \cite{eliashberg_2-knots,Cao_Gallup_Hayden_Sabloff} shows that given any exact Lagrangian $L$ with Legendrian boundary in a Weinstein domain, there is a symplectic surface $S$ with transverse boundary smoothly isotopic to $L$, which lives in an arbitrarily small $\bD^*L$ neighborhood. We call this the {\em symplectic pushoff} of $L$. In the case that $L$ is a disk which is the setting of Theorem \ref{thm: regular_lag_disk}, the symplectic pushoff of the Lagrangian $L = \{y_1=y_2=0\} \subset \bD^*L_{(x_1,y_1,x_2,y_2)}$ is taken to be the disk $S=\{x_1=y_2, x_2=-y_1\}$. 

The following lemma concerns an exact Lagrangian filling $L\subset W$ of a Legendrian link $\Lambda \subset \partial W$ in ambient dimension four. First, recall that given a Lagrangian disk $D$ in any Liouville domain $X$, we can decompose $X$ as $X = X' \cup h$ where $X'$ is a Liouville sub-domain, $\partial X'$ is given by contact-$(+1)$ surgery along $\partial D$, and $h$ is a Weinstein handle with D as its cocore. We refer to this process as \emph{adding $L$ back}.

\begin{lemma}\label{lem:lag_complement_stein}
Let $L$ be a properly embedded exact Lagrangian in a $4$-dimensional Weinstein domain $W$. If $S$ is a symplectic pushoff of $L$ such that $W \setminus \nu(S)$ admits a PANLF, then $W \setminus \nu(L)$ admits a Stein structure. Further, if $L$ is a Lagrangian disk, and $\Pi:W \to \bD^2$ is a Lefschetz fibration such that $S$ is a positive multisection of $\Pi$, then the Weinstein structure on $W$ obtained by adding $L$ back to $W \setminus \nu(L)$ is homotopic to an almost Stein structure supported by $\Pi$.
\end{lemma}

\begin{proof}
    Let $\nu(S)$ be a standard neighborhood of $S$. 
    By the contractibility of supporting Stein structures given in Theorem \ref{thm: stein_on_panlf}, we may consider the Weinstein structure given in \S \ref{subsec:panlf_weinstein} on $W\setminus \nu(S)$. Let $\nu(L)$ be a standard neighborhood of $L$ in $W$. In light of the Weinstein structure in \S \ref{subsec:panlf_weinstein}, in particular, the fact that it is obtained by a sequence of (compactly supported) Weinstein $2$-handle attachments, we have that the associated Morse function $\phi$ on $W \setminus \nu(S)$ has no critical points near the boundary, and can thus be homotoped so that the critical points are away from $L$. 
    By choosing $\phi$ and $\nu(S)$ appropriately, we can ensure that, up to a smooth isotopy, $-\partial \nu(L)$ is transverse to the Liouville flow and hence a contact type hypersurface in $W \setminus \nu(S)$. It follows that $W \setminus \nu(L)$ admits a Stein structure (which a priori is not related to the starting symplectic form). 
    Further, flowing along the Liouville vector field produces an isotopy $\Psi_t$ between $\partial (W \setminus \nu(S))$ and $-\partial \nu(L)$. 
    
    \begin{figure}
        \centering
	    \begin{overpic}{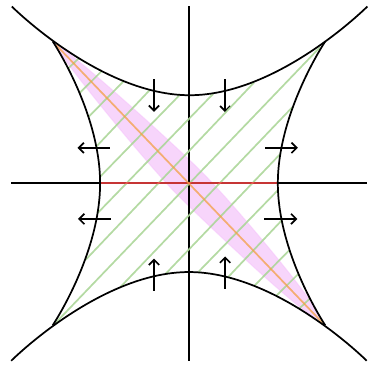}
        \put(140,92){\textcolor{red}{$L$}}
        \put(145,122){\textcolor{YellowGreen}{$N$}}
        \put(73,110){\textcolor{orange}{S}}
        \end{overpic}
        \caption{Foliation over a Weinstein 2-handle.}
        \label{fig:2-handle_LF}
    \end{figure}

    \begin{figure}
        \centering
        \begin{overpic}{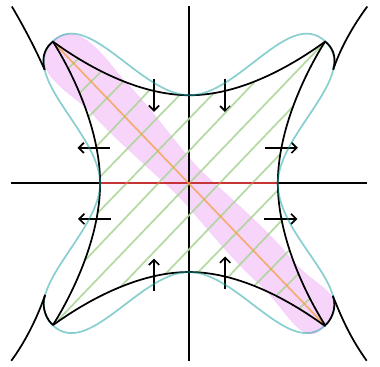}
        \put(140,92){\textcolor{red}{$L$}}
        \put(143,124){\textcolor{YellowGreen}{$N$}}
        \put(73,110){\textcolor{orange}{S}}
        \end{overpic}
        \caption{The outer blue region is identified with $\nu(L) \subset W$, the purple region is $\nu(S)$, the inner black region with ``corners'' is the Weinstein handle model of $\nu(L)$.}
        \label{fig:2-handle_smoothings}
    \end{figure}

        For the second part, we proceed with the following steps:
        \begin{enumerate}
        \item Define a smooth foliation $\Pi_{\bD^*L}$ on (smoothed) $\bD^*L$ such that $S$ is a section. See Figure \ref{fig:2-handle_smoothings}.
        Consider $\bD^*L$ with coordinates $(x_1,y_1,x_2,y_2)$ where $S = \{x_2=-y_1,x_1=y_2\}$. See Figure~\ref{fig:2-handle_LF}. Further, identify $\nu(S) = S \times D^2$ with the $D^2$ factors on planes parallel to $N=\{x_2=y_1, x_1=-y_2\}$, and (unions of) these define the fibers of $\Pi_{\bD^*L}$. To define the map $\Pi_{\bD^*L}$, observe that each leaf in the foliation $\Pi_{\bD^*L}$ on $\bD^*L$ intersects a unique point on $S$, as in Figure \ref{fig:2-handle_LF}. This defines a map $\Pi_{\bD^*L}: \bD^*L \to S$.
        
        \item Identify $\nu(L)\cong \bD^*L$ via $\varphi$, such that $\Pi|_{\nu(L)}$ is the composition $\pi_S \circ \Pi_{\bD^*L}\circ\varphi$, where $\pi_S: S \to \bD^2$ is the branched cover map given by restricting $\Pi$ to $S$ as a positive multisection. 
        
        \item Consider the PANLF on $W\setminus \nu(L)$ given by $\Pi|_{W \setminus\nu(L)}$. Let $(f_1, J_1)$ be an almost Stein structure on $W \setminus\nu(L)$ supported by $\Pi|_{W \setminus\nu(L)}$. Let $j$ be an almost complex structure on $\bD^2$ such that $\Pi|_{W \setminus\nu(L)}$ is $(J_1, j)$-holomorphic. Let $\cN_{\Weinstein}\subset \nu(L)$ be a collar neighborhood of $L$ modeled as a Weinstein $2$-handle. Define a neighborhood $\cN_{\Weinstein}^{\concave}$ of the concave boundary of $\cN_{\Weinstein}$. We extend the PANLF and $(f_1, J_1)$ on $W\setminus\nu(L)$ slightly into $W\setminus(\nu(L)-(\cN_{\Weinstein} \setminus \cN_{\Weinstein}^{\concave}))$.
        \item Choose $f = \frac{1}{2}(x_1^2+x_2^2) - \frac{1}{4}(y_1^2+y_2^2)$, $\omega_{\std}=dxdy$, and $X = 2x\partial_x-y\partial_y$. This defines the standard Weinstein cobordism structure on $\bD^*L$. We further modify $\varphi$ if necessary so that $\varphi^*J = J_1$ on $\cN_{\Weinstein}^{\concave}$. We need to check that $\pi_S \circ \Pi_{\bD^*L}$ supports the almost Stein structure $(f,J = J_{\std})$ on $\bD^*L$ given by the standard Weinstein $2$-handle cobordism up to Weinstein homotopy. For this, it suffices to find $(f',J')$ almost Stein homotopic to $(f,J)$ through a homotopy fixing the concave boundary, such that $\pi_S \circ \Pi_{\bD^*L}$ is $(J',j)$-holomorphic. 
           
        To see this, first, note that $J$ makes all the fibers $N$ pseudoholomorphic. By our choice of $\varphi$,  $$\Pi|_{\nu(L)} = \pi_S \circ \Pi_{\bD^*L}\circ\varphi.$$
        Therefore, we have that $\varphi^*J$  
        makes the $D^2$ fibers of $\Pi|_{\nu(L)}$ pseudoholomorphic in $\cN_{\Weinstein}$.  
        In addition, since $\pi_S \circ \Pi_{\bD^*L}\circ\varphi$ is $(J_1,j)$-holomorphic in a neighborhood $\cN_{\Weinstein}^{\concave}$ of the concave boundary, we have that $\pi_S \circ \Pi_{\bD^*L}$ is $(J,j)$-holomorphic also in $\varphi(\cN_{\Weinstein}^{\concave})$. Now, we claim that we can find an almost Stein structure $(f',J')$ on $\bD^*L$, homotopic to $(f,J)$, such that $\pi_S \circ \Pi_{\bD^*L}$ is $(J',j)$-holomorphic. To do this, we modify $J$ along $S$ to become $\pi_S^*j$. This can be done by the uniformization theorem since $S$ is a disk. Furthermore, we can define a $1$-parameter family $J_t$ of almost complex structures on $\bD^*L$ such that $J_0 = J$ and $J_t$ is given by $J|_N$ on $N$ and $\phi_t^*(J|_S)$ on $S$, where $\phi_t: S \to S$ is an isotopy such that $\phi_0=id$, $\phi_1^*(J|_S) = \pi_S^*j$, and $\phi_t|_{\varphi(\cN_{\Weinstein }^{\concave})} = id$. 
        Let $\tilde{\phi_t}$ denote the extended smooth isotopy on $\bD^*L$ by identity along fibers $N$. Then, $(f', J'):=(\tilde{\phi_1}^*f, \tilde{\phi_1}^*J)$ is the desired almost Stein structure on $\bD^*L$ compatible with $\phi_*(\Pi|_{\nu(L)})$.
        Note that $\tilde{\phi_t}$ preserves the almost Stein structure along the concave boundary since $\phi_t$ preserves $J$ along the concave boundary. Hence the homotopy of almost Stein cobordisms preserves the contact structure on the concave boundary of $\bD^*L$, and thus extends to a homotopy of almost Stein structures on $W$.
        
        \item Now, when we attach the Weinstein/almost Stein $2$-handle $\bD^*L$ to $W\setminus \nu(L)$, we obtain an almost Stein structure on $W$ supported by $\Pi$, since both $W\setminus \nu(L)$ and $\nu(L)$ are supported by $\Pi|_{W\setminus \nu(L)}$ and $\Pi|_{\nu(L)}$ respectively. To be completely precise, we also need to check that the induced one-form on $\partial \Pi$ by the almost Stein structure is a Giroux form. However, this is automatic since the $D^2$ factors that restrict to pages on $\partial W \cap \nu(L)$ are $J$-holomorphic by construction, hence symplectic. In addition, note that $\nu(L)$ does not meet $\partial_h W$ so we indeed have a Giroux form and no other conditions from \cite[Definition 1.23]{spinal_I} need to be checked.
        \end{enumerate}
\end{proof}   

\begin{remark}
    Note that the above lemma says that when a disk multisection in a Lefschetz fibration arises as the symplectic pushoff of a Lagrangian disk with Legendrian boundary, the Stein structure on the complement is a Stein sub-domain of the Lefschetz fibration. One might wonder if this is true if the multisection cannot be the symplectic pushoff of a Lagrangian slice disk, e.g. in the case of the quasipositive disk bounded by $-8_{20}$.
\end{remark}

\subsection{Proof of Theorem \ref{thm: regular_lag_disk}}
\label{subsec: proof_of_main_thm}
In this subsection, we complete the proof of Theorem \ref{thm: regular_lag_disk} using pseudoholomorphic curves. The main remaining ingredient is Proposition \ref{prop:sympl_disk_multisection} showing that a symplectic surface obtained from a symplectic pushoff of a Lagrangian disk with Legendrian boundary can be made a positive multisection in the sense of \cite{boileau-orevkov}.

The overall strategy to prove Theorem \ref{thm: regular_lag_disk} is to leverage the pseudoholomorphic foliation associated to a Lefschetz fibration. Here, the symplectic completion is a refinement of the Theorem-Definition \ref{def:double_completion_LF} in the Lefschetz fibration case for two reasons: (1) we need the contact form on the smoothed hypersurface to have the property such that the boundary of the symplectic surface of interest is of the smallest action; (2) the genus of the fibers of the Lefschetz fibration can now be of higher genus, so we do not have the freedom to do a further perturbation going from stable Hamiltonian structures to contact structures while requiring that the foliation persists. 

In light of the more serious requirement (2), we stay with the stable Hamiltonian structure giving a (potentially nongeneric) pseudoholomorphic foliation, and observe that the stable Hamiltonian structure is of confoliation type, so the appropriate Stokes' theorem still holds. In addition, the primitives of the stable Hamiltonian $2$-form $\omega$  of interest share the same orbits in the binding region, so (1) is also satisfied for the stable Hamiltonian $1$-form $\lambda_0$ defined as the limit of the contact form $\alpha_\epsilon$ as $\epsilon\to 0$, as in the following lemma.

\begin{lemma}
\label{lem: completion}
    Let $\Pi: E \to \bD^2$ be a bordered Lefschetz fibration. 
    There exists a symplectic completion $\widehat{E}$ such that
    \begin{itemize}
        \item $(\widehat{E},\widehat\omega) = (W, \omega) \cup ([0,\infty)_\tau \times M, \omega+d(\varphi \lambda_0))$, where $\varphi$ is a positive function in $\tau$ and $\lambda_0$ is a stable Hamiltonian $1$-form for $\omega$, where $(W,\omega, \lambda_0)$ is specified in the following items;
        
        \item $W$ is a topological smoothing of $E$ and $\omega$ is a symplectic form on $W$ with a primitive Giroux form $\alpha_\epsilon$ near $M$, such that any given binding orbit $\gamma$ of $\alpha_\epsilon$ can be made of arbitrarily small action; 
        
        \item $\Pi$ extends to $\widehat{E}$ and $\Pi|_{M}$ is isotopic to the original open book $\pi$ associated to $\partial \Pi$;
        
        \item $\lambda_0:=\lim_{\epsilon\to 0} \alpha_\epsilon$ is a stable Hamiltonian $1$-form for $\omega$ of confoliation type.
        
        \item There exists an almost complex structure $\widehat{J}$ on $\widehat{E}$ such that 
        \begin{itemize}
            \item $\widehat{J}$ is adapted to the stable Hamiltonian structure $\cH:=(\omega, \lambda_0)$,
           
            \item $\widehat{J}$ matches an almost complex structure $J_+$ in $(0,\infty)\times M$, which admits a $J_+$-holomorphic foliation lifting the pages of $\pi$,
            
            \item $\widehat{J}|_W \in \cJ(\Pi)$.
        \end{itemize}
    \end{itemize}
\end{lemma}
\begin{proof}
    The proof is a variation on a construction given in \cite{Wendl_SHS, wendl_hierarchy, spinal_I}. 
    
    First, we find smoothing $(W, \lambda_W, \omega)$ of the bordered Lefschetz fibration $E$ similar to \cite[\S 2.5]{spinal_I} and \cite[\S 3.1]{spinal_II}. In particular, $\partial W$ can be decomposed as 
    \begin{equation}
    \label{eq: partial_W_decomp}
        \partial W = \partial_h W \cup \partial_v W,
    \end{equation}
    where $\Pi$ can be extended to $\widehat{\Pi}$ so that
    $$\widehat{\Pi}|_{\partial_v W}: \partial_v W \to S^1$$
    and
    $$\widehat{\Pi}|_{\partial_h W}: \partial_h W \to \bD^2$$
    are fiber bundles and the Liouville form $\lambda_W$ induces a contact form, denoted as $\alpha$, on $\partial W$ compatible with the open book associated to~\eqref{eq: partial_W_decomp}. See Figure \ref{fig: smoothed_LF}. 
    
    For the smoothing, the coordinates $(s,\phi, t, \theta)$ on $\nu(\partial_h E)$ are replaced by
    \begin{equation}
        \big\{ (g(\rho), \phi, f(\rho), \theta)\big|\phi, \theta \in S^1 \text{ and } 0\leq \rho<1 \big\},
    \end{equation}
    where $g,f: [0,1)\to (0,1]$ are smooth functions specified below. In particular, $$(g(\rho), f(\rho)) = (1, -\rho)$$ 
    for $\rho \geq \delta$, $g'(\rho)>0$ and $f'(\rho) <0$ for $\rho\in [0,\delta]$ where $\delta$ is specified below.

    \begin{figure}
        \centering
	    \begin{overpic}[width=\linewidth]{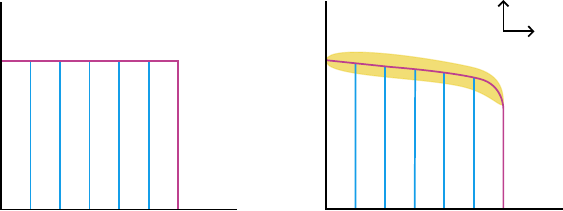}
        \put(400,125){\small $s$}
        \put(387,150){\small $t$}
        \end{overpic}
        \caption{Smoothed Lefschetz fibration and corresponding boundary decomposition. The highlighted region on the right is the smoothed spine.}
        \label{fig: smoothed_LF}
    \end{figure}

    In the local neighborhood $\cN:=\mathbb{D}^2_{\rho, \phi}\times S^{1}_{\theta}$ of the binding orbit $\gamma$, we consider the contact form 
    $$\alpha_B:=f(\rho)d\theta + g(\rho)d\phi,$$
    where $f,g: [0,1]\to \bR$ are some smooth functions specified in \cite[Lemma 3.7]{wendl_hierarchy} with the following properties:
    \begin{itemize}
        \item The function $fg'-f'g>0$, ensuring that $\alpha_B$ is contact.
        \item There is a $\delta>0$ such that $g(\rho)=1$ for $1-\delta<\rho \leq 1$.
        \item The only closed Reeb orbit in $\cN$ is the orbit $\{\rho=0\}=\gamma$ which is elliptic. 
        \item $CZ_{\tau}(\gamma) = 1$ for the asymptotic trivialization $\tau$ given by $\langle \rho, \phi\rangle$ and $\gamma$ has period less than all other orbits.
    \end{itemize}
    The existence of such $f,g$ is given in \cite[\S 3.2.1]{wendl_hierarchy} (see also Figure 8 op. cit.). Let us denote $M = M_P \cup M_\Sigma$ associated to $\Pi$, where $M_P$ denotes the paper, and $M_\Sigma$ the unions over all the bindings. On $M
    _P$, choose the $1$-form $\lambda$ such that $d\lambda>0$ on the pages and near $\partial M_P$, $\lambda$ takes the form
    $$\lambda = (1-\rho) d\theta$$
    in the binding coordinates $(\rho, \phi, \theta)$ and we assume $\rho<1$.
    Now, for $\epsilon>0$ sufficiently small, consider 
    $$\lambda_\epsilon:=
    \epsilon \lambda + d\phi$$ 
    which is contact on $M_P$. 
    Now we extend $\lambda_\epsilon$ to the rest of the boundary $\partial W$ as follows. Recall that $(s,\phi, t, \theta) \in (-1,0]\times S^1 \times (-1,0]\times S^1$ are the coordinates on the components of $\nu(\partial_vE\cap \partial_hE)$ such that
    $$\Pi(s,\phi,t, \theta) = (s,\phi).$$

    Now, for $r<1$, the extension is
    $$\alpha_\epsilon:=
    \begin{cases}
        f_\epsilon(\rho)d\theta + g_\epsilon(\rho)d\phi &\text{ on } \partial_h W\Leftrightarrow \rho \leq r-\delta\\
        \lambda_\epsilon &\text{ on } \partial_v W \Leftrightarrow r-\delta<\rho\leq 1
    \end{cases}$$
    where 
    $f_\epsilon, g_\epsilon: [0,r]\to \bR$ satisfy
    \begin{enumerate}
        \item $(f_\epsilon, g_\epsilon) = (f, g)$ for $\rho \leq r-\delta'$;
        
        \item $g_\epsilon = 1$ and $f_\epsilon' <0$ for $\rho \in [r-\delta', r-\delta]$;

        \item $(f_\epsilon, g_\epsilon) = (\epsilon(1-\rho), 1)$ for $\rho\in [r-\delta, r]$,
        
        \item $f_\epsilon\to f$ and $g_\epsilon\to g$ smoothly as $\epsilon\to 0$,
    \end{enumerate}
    for some $0\leq r-\delta'\leq r-\delta \leq r$. 
    One can check that $\alpha_\epsilon$ is indeed a contact form supporting the same embedded spinal open book as in the proof of \cite[Proposition 2.28]{spinal_I}. Therefore, we may define $\lambda_0:=\lim_{\epsilon\to 0} \alpha_\epsilon$ which is a stable Hamiltonian $1$-form for $\omega$ of confoliation type.

    Finally, we discuss the almost complex structure in the completion amenable for the pseudoholomorphic pages. Start by considering the almost complex structure $J_+$ in $[0,\infty) \times M$ with the following properties: 
    \begin{enumerate}
        \item $J_+\in \cJ(\cH)$ where $\cH:=(\omega, \lambda_0)$.
        \item $J_+v_1 = \beta(\rho)v_2$, where $v_1= \partial \rho$ and $v_2 = -g(\rho) \partial \theta + f(\rho) \partial \phi$, and $\beta$ is a function specified later so that $J_+$ extends smoothly to $\rho=0$. 
    \end{enumerate}

    Now let $C:=\{c\}\times M \subset [0,\infty)_\tau \times M$ be a contact-type hypersurface. We claim that for any $c>0$, the Weinstein domain
    $$W_C: = W \cup \{0\leq \tau\leq c\}\times M$$
    admits a Lefschetz fibration $\Pi_C: W_C \to \bD^2$ with $\Pi_C|_{W} = \Pi$, such that 
    \begin{equation}\label{cond:star}
    \text{for any } x \in \bD^2, \Pi^{-1}(x) \cap (W_C \setminus W) \text{ is $J_+$-holomorphic}.
    \tag{*}
    \end{equation}
    
    To construct the curves in the foliation, we start with the curves $u_{a,x}:=(a,\Pi^{-1}(x))$ in $[0,\infty) \times M_P$. These are $J_+$-holomorphic by the assumption $J_+\in \cJ(\cH)$, since $T\Pi^{-1}(x)=T_xP=\ker \lambda_0\mid_{x}$. Now, recall that $[0,\infty) \times M_\Sigma$ is a union of $[0,\infty)_a \times S^1 \times \bD^2$ components. In each of these components, the curves $u_a$ are extended by solutions $(a(\tau),t,\rho(\tau),\phi_0)$,  where 
    \begin{equation}
    \label{eq:a'_rho'}
    \dfrac{da}{d\tau}=f(\rho), \dfrac{d\rho}{d\tau}=
    \begin{cases}
        \frac{f'(\rho)}{\beta(\rho)D(\rho)} &\text{ if } \rho < 1-\delta\\
        -1 &\text{ if } \rho \geq 1-\delta 
    \end{cases}
    \end{equation}
    where $D(\rho) = f(\rho)g'(\rho) - f'(\rho)g(\rho)$, and $\beta$ is a positive function that is 1 outside a neighborhood of 0. 
    Further, $a(0)$, $\rho(0)=1$ and $\phi_0$ specify the initial conditions of the solutions which match them up smoothly with the curves in $[0,\infty) \times M_P$. This yields a $J_+$-holomorphic half-cylinder asymptotic to the orbit at $\{\rho=0\}$ as in \cite[\S 3]{Wendl_SHS}.

    Now we extend this $J_+$ to $W_C$ such that $\Pi_C|_W = \Pi$ yields the desired foliation satisfying \eqref{cond:star}, by refining the above paragraph. Again, parametrize the unit disk $\bD^2_{\rho,\phi}$ by $\{\rho \leq 1\}$ and recall that our choice of coordinates on $[0,\infty) \times S^1 \times \bD^2$ satisfy that $\{0\} \times S^1 \times \bD^2 \subset \partial W$ and that $$\Pi|_{\{0\}\times S^1 \times \bD^2}: \{0\}\times S^1 \times \bD^2 \to \bD^2$$ is given by projection to the last coordinate. 

    Let $(\tau, \eta)\in [0, \infty) \times S^1$ be the coordinates on the domain of the pseudoholomorphic curve $u_{a,x}$ near a puncture. It suffices to consider each component of the bindings individually.
    Given $(\rho_0,\phi_0) \in \bD^2$,
    define the extension of the curve $\Pi^{-1}(\rho_0,\phi_0)$ into $[0,\infty) \times S^1\times \bD^2$ by $(a(\tau),\eta,\rho(\tau),\phi_0)$, where $a(0) = 0$, $\rho(0)= \rho_0$, and $a'$ and $\rho'$ satisfy the conditions as \eqref{eq:a'_rho'}. This ensures we have a smooth foliation by surfaces on $W \cup [0,\infty)\times M$, which are $J_+$-holomorphic where $J_+$ is defined.
    On $[0,\infty) \times M_P$, this foliation is the $a$-translation of the foliation $\Pi|_{M_P}$. Given any $p \in [0,\infty) \times M_\Sigma$, it lies on a unique leaf of the foliation (this is by a standard observation of how integral curves of a nonzero vector field foliates the domain). For the leaf passing through $p$, denote its minimum $a$ value inside $[0,\infty) \times S^1 \times \bD^2$ by $a_p$.  Also, the intersection of the leaf with $\{a_p\} \times S^1 \times \bD^2$ is a circle at $(\rho_p, \phi_p) \in \bD^2$, where $\rho_p = 1$ if $a_p>0$, since any $\bR$-invariant leaf projects to $\partial_v E$ which is precisely when $\rho = 1$.
    
    Now, consider the disk $\bD^2_c$ given by $\{\rho \leq 1+c\}$. Define \[\Pi_C: W_C \to \bD_c^2\] as follows:
    \[\Pi_C(p) =
    \begin{cases}
        & \Pi(p) \text{, if } p \in W\\
        & (a_p+\rho_p,\phi_p) \text{, otherwise.}
    \end{cases}
    \]
    \noindent It is clear that $\Pi_C$ is a Lefschetz fibration that restricts to $\Pi$ on $W$. Now, pick $J_1 \in \mathcal{J}(\Pi_C)$ and therefore $J_1 \in \mathcal{J}(\Pi)$. Now, since $c>0$, we can use the region $W_C \setminus W$ to interpolate between $J_1$ and $J_+$, and define $\widehat{J}$ so that it restricts to $J_+$ on the cylindrical end and an element of $\mathcal{J}(\Pi)$ on $W$.
\end{proof}

\begin{lemma}
\label{lem: symplectic_to_J_transverse}
    Let $E$ be the total space of a symplectic Lefschetz fibration $\Pi$. Let $S \subset E$ be a properly embedded symplectic disk obtained as the symplectic pushoff of a properly embedded Lagrangian disk $L \subset E$ with Legendrian boundary.
    Then, there exists a cylindrical almost complex structure $\widehat{J}_0$ identifying with $J_+$ on the cylindrical end, such that the completion of $S$ is a transversely cut-out $\widehat{J}_0$-holomorphic plane asymptotic to the Reeb orbit of smallest action.
\end{lemma}
\begin{proof}
    Let $\gamma:= \partial S$ denote the transverse knot. We stabilize the induced boundary open book decomposition from $\Pi$, so that $\gamma$ becomes a binding component, hence a Reeb orbit following an algorithm in \cite{Baker_Etnyre_VHM_rational}. This is still denoted as $\Pi$. We apply Lemma \ref{lem: completion} to obtain the desired Giroux form on $(M,\xi)$ making $\gamma$ of the smallest action. 
    
    Now, we extend $S$ cylindrically to $u$ so that $u$ is $J_+$-holomorphic on $(0,\infty) \times M$. Similar to the proof of Lemma \ref{lem: completion}, we define $S_c:= S \cup [0,c]\times \gamma$ for any $c>0$. Find a $J_0$ that makes $S_c$ $J_0$-holomorphic. Now, since there is an overlap in the domain of definition for $J_0$ and $J_+$, we do an interpolation to ensure that $u$ is $J_0$-holomorphic in $W$ and $J_+$-holomorphic in $(0,\infty)\times M$, so in particular it's asymptotically cylindrical.

    Now we check for automatic transversality. Recall that $\ind(u) = -\chi(u) + 2c_{\tau}(u) + CZ_\tau(\gamma)$. Automatic transversality \cite[Proposition A.1]{Wendl_strongly} states that, if $$\ind(u)>c_N(u):=\frac{\ind(u)-2+2g+\#\Gamma_0}{2},$$ 
    where $\#\Gamma_0$ denotes the number of even Conley-Zehnder orbits, then $u$ is cut out transversely.  Therefore, it suffices to achieve automatic transversality if
    \begin{equation}
    \label{eq:auto_transversality}
        \ind(u)> -2+\#\Gamma_0 = -2
    \end{equation}
    for our $u$ of interest, which has the topology of a plane. 
    
    To check \eqref{eq:auto_transversality}, we compute $c_\tau(u):=c_1^\tau(u^*T\widehat{W}) = \chi(u) + c_1^\tau(N_{u})$. The normal bundle $N_{u}$ matches $\xi$ near infinity. Note that a constant section of the trivialization $\tau$ is given by $\partial_\rho$, which is the oriented normal of the core of the solid torus parameterized by $(\theta, \rho, \phi)$. The value of $c_1^\tau(N_u)$ is given by the number of signed singularities of this vector field $\partial_\rho$ when extended over $u$. We compute this value over the Lagrangian $L$, since $c_1^\tau(N_u)$ is independent of homology class representative.

    By the choice of the Weinstein neighborhood in \S \ref{subsec: symp_complement_to_lag_complement_lemma}, the Lagrangian $L$ can be parameterized as $\{y_1=y_2=0\}$ inside $(D^2_{x_1,x_2}\times D^2_{y_1,y_2}, dx_1dy_1+dx_2dy_2)$, such that $\{x_1^2+x_2^2 = 1\} \subset \partial W$.  
    Further, $\partial L \subset \{x_1^2+x_2^2 = 1\}$ is a Legendrian for the contact structure which is locally the kernel of the 1-form $(x_1dy_1-y_1dx_1+x_2dy_2-y_2dx_2)$. A constant section of the trivialization $\tau$ at the boundary of $L$, given by the oriented normal inside the contact plane, is the vector field $X=-\sin(\theta)\partial_{x_1}+\cos(\theta)\partial_{x_2}$ at the point $(0,0,\cos(\theta),\sin(\theta))$ on $\partial L$. By looking at the winding of $X$ with respect to $\partial_{x_1}$ (which extends without zeros over $L$), one can see that $X$, extended over $L$, has a single singularity at the origin inside $L$, and the sign of that singularity is -1. It follows that $c_1^\tau(N_S) = c_1^\tau(N_L) = -1$.
    Therefore, $c_\tau(u) = 0$.

    Now, we have that
    $$\ind(u) = -\chi(u)+2c_\tau(u)+CZ_\tau(\gamma)=-1+0+1=0.$$
    Therefore, (\ref{eq:auto_transversality}) holds and $u$ is transversely cut out.
\end{proof}

\begin{proposition}
\label{prop:sympl_disk_multisection}
    Let $(E, \lambda, \varphi)$ be a Weinstein domain supported by a bordered Lefschetz fibration $\Pi: E \to \mathbb{D}^2$. Let $S \subset E$ be a properly embedded symplectic disk 
    obtained as the symplectic pushoff of a properly embedded Lagrangian disk $L \subset E$ with Legendrian boundary.
    Then $S$ is symplectically isotopic to a positive multisection of $\Pi$.
\end{proposition}
\begin{proof}
    Let $\gamma:= \partial S$ denote the transverse knot. By Lemma \ref{lem: symplectic_to_J_transverse}, up to ambient Weinstein homotopy, we can make $S$ a transversely cut out $\widehat{J}_0$-holomorphic plane asymptotic to the smallest action Reeb orbit in the contact boundary, still denoted as $(\partial E, \lambda)$. 

    Now we consider an almost complex structure $J_1 \in \cJ(\Pi)$. By definition, any generic fiber of the Lefschetz fibration is $J_1$-holomorphic. We consider the double completion $(\widehat{E}, \widehat{J}_1)$ given in Lemma \ref{lem: completion} where in particular $\widehat{J}_1$ is $\widehat\lambda_1$-adapted and the generic fibers are also completed into pseudoholomorphic curves with asymptotic ends at Reeb orbits.
    Furthermore, we have that $\widehat{J}_1$ agrees with $J_1$ in $W$ and agrees with $J_+$ in $(0,\infty)\times M$. Let $\cH:=(\omega, \lambda_0)$ be the stable Hamiltonian structure given in Lemma \ref{lem: completion}. 
    
    Now since both $\widehat{J}_0$ and $\widehat{J}_1$ are $\cH$-adapted almost complex structures, we may consider a generic path of $\cH$-adapted almost complex structures $\widetilde{J}_t$ between them where $\widetilde{J}_0=\widehat{J}_0$ and $\widetilde{J}_1=\widehat{J}_1$.   
    Now we consider a family of $\widetilde J_t$-holomorphic curves $S_t$ where $S_0=S$. Since $\cH$ is of confoliation type, we have Stokes' theorem by \cite[Proposition 2.4]{spinal_II}. Now, since $\omega$ is exact, there is no bubbling, and because the asymptotic end of $S$ has the smallest action among all Reeb orbits, there is no breaking. Therefore, SFT compactness ensures that $S_t$ forms a smooth family of embedded symplectic planes, all asymptotic to the same orbit $\gamma$. 
    The conclusion is that we have a common almost complex structure $\widehat{J}_1$ which makes both $S_1$ and the smooth part of the completed fibers pseudoholomorphic. Let $\widehat{F}$ denote any completed fiber (including possibly nodal fibers) and further assume that $S_1$ avoids critical points of $\Pi$ up to a small symplectic isotopy.  
    
    Now, consider $S_R$ and $F_R$, the compact cutoff of $\widehat{S}$ and $\widehat{F}$ at a level $R \gg 0$. By Siefring \cite[Theorem 2.3]{Siefring_asymptotics}, $\partial S_R$ is a torus braid around $\gamma$ of braid index $d$. In particular, since the base of $\Pi$ is $\bD^2$, we may assume that $\partial S_R \subset \partial_v E_R$, i.e. the vertical boundary of the compact cutoff of $\widehat{E}$. Now, by an application of the Carleman Similarity Principle (see e.g. \cite[Lemma 2.36]{wendl2016lectures}), we may assume that $S_R$ and $F_R$ have isolated intersections. This is because otherwise $S_R = F_R$, so $S$ is symplectically isotopic to a smooth fiber. However, since $E$ is a PALF, in this case $E = \bD^2 \times \bD^2$ and any fiber is symplectically isotopic to a section by \cite{Eliashberg90b, EliashbergFraser98}. 
    
    Now, by \cite[Lemma 6.31]{spinal_II}, 
    locally in each small neighborhood in the interior of $S_R$, $\Pi$ takes the form of a possibly branched cover. By connectedness of $S_R$, the total degree of the possibly branched cover is a fixed number $d'$. By continuity, $d = d'$ and we conclude that $\Pi|_S$ is a branched cover of surfaces, and thus $S$ is a positive multisection.
    \end{proof}

\begin{proof}[Proof of Theorem \ref{thm: regular_lag_disk}]
    First, we use \cite{GirouxPardon2017Lefschetz} and \cite[Corollary 11.21]{Cieliebak_Eliashberg} to find a Weinstein homotopy from the configuration $(W,L)$ to $(W',L')$, where $W'$ is equipped with a Lefschetz fibration $\Pi: W' \to \bD^2$ and $L'$ is a Lagrangian at time one of the homotopy. Let $\cW$ be an almost Stein structure on $W'$ supported by $\Pi$. Then, in a neighborhood $\nu(L')$, we find $S$, a symplectic pushoff of $L'$ with transverse boundary as in \S \ref{subsec: symp_complement_to_lag_complement_lemma}. 

    By Proposition \ref{prop:sympl_disk_multisection}, $S$ is symplectically isotopic to a positive multisection for $\Pi$, denoted as $S_1$. Let $S_t$ denote the symplectic isotopy, where $S_0=S$.
    Now, by Definition \ref{def: panlf}, $W' \setminus \nu(S_1)$ has the structure of a PANLF. Therefore, by Theorem \ref{thm: stein_on_panlf}, $W' \setminus \nu(S_1)$ is supported by a Stein structure, so by the first part of Lemma \ref{lem:lag_complement_stein}, we know that $W\setminus \nu(L)$ is also Stein. Now, by the proof of \cite[Corollary 3.7]{Conway_Etnyre_Tosun}, we know that $L'$ is a cocore, hence regular, for some Weinstein structure on $W'$. In fact, we know more about how this Weinstein structure is related to the original almost Stein structure $\cW$ as follows.
    
    By the symplectic isotopy extension theorem in \cite[Theorem 3.3.2]{mcduff_salamon_intro} (see also \cite{Banyaga}), we have that the symplectic isotopy $S_t$ induces an ambient symplectic isotopy on the neighborhood $\nu(L')$ described in Lemma \ref{lem:lag_complement_stein}, denoted as $\Phi_t$. 
    Therefore, $\Phi_1(L')$ is also an exact Lagrangian which has a pushoff $\Phi_1(S)$.
    Now, extend $\Phi_t$ to a smooth isotopy $\tilde \Phi_t$ on $W'$. Consider the Lefschetz fibration $\tilde\Phi_1^*\Pi$ on $W'$, and an associated almost Stein structure $\tilde\Phi_1^*\cW$. The above implies that $L'$ is a Lagrangian in $\tilde\Phi_1^*\cW$, whose symplectic pushoff $S$ is a positive multisection for $\tilde\Phi_1^*\Pi$. Then, Lemma~\ref{lem:lag_complement_stein} implies that up to Weinstein homotopy, $L'$ is regular in an almost Stein structure supported by $\tilde\Phi_1^*\Pi$.
    
    Now, since the space of supporting almost Stein structures given an allowable Lefschetz fibration is contractible by \cite[Theorem 1.24]{spinal_I},
    we have that up to Weinstein homotopy, $L'$ is regular in $\tilde\Phi_1^*\cW$. 
\end{proof}
    
\section{Proof of Theorem \ref{thm: planarNLF_multisection}}
\label{sec: planarNLF_multisection}
In this section, we prove Theorem \ref{thm: planarNLF_multisection}, which generalizes \cite{boileau-orevkov}.

\begin{proof}[Proof of Theorem \ref{thm: planarNLF_multisection}]
    Let $W:= B \times F$ where $B$ and $F$ are compact surfaces with nonempty boundaries. $W$ is a PANLF without singular or exotic fibers and the projection map $\Pi: W \to B$.
    Let $S$ be a symplectic surface with braided boundary in the vertical boundary of $W$. The overall strategy is similar to the proof of Theorem \ref{thm: regular_lag_disk}. Specifically, we seek an almost complex structure on the double completion $\widehat{E}$ of $E$ as in Definition \ref{def:double_completion_LF} that makes both the completed fibers and the completed symplectic surface $\widehat{S}$ pseudoholomorphic.

    To find such an almost complex structure, we start with $J\in \cJ(\Pi)$ and consider the completed almost complex structure $\widehat{J}$ as given in Definition \ref{def:double_completion_LF}.
    Now, we find an asymptotically cylindrical almost complex structure $\widehat{J}_S$ that agrees with $J_+$ on $(0,\infty)\times M$ and makes $\widehat{S}$ pseudoholomorphic, as constructed in the proof of Lemma \ref{lem: symplectic_to_J_transverse}. Next, consider a family of almost complex structures $\{\widehat{J}_{\tau}\}_{\tau\in[0,1]}$ where $\widehat{J}_0=\widehat{J}$ and $\widehat{J}_1 = \widehat{J}_S$. A priori, we don't necessarily have that $\{\widehat{J}_\tau\}$ is a compatible homotopy. However, by the assumption of no singular and exotic fibers, the proof of Theorem \ref{thm: spinal_ii_6_4} still gives us a one-parameter family of nearly Lefschetz fibrations without singular or exotic fibers. This is because $W$ is Weinstein so there are no closed pseudoholomorphic spheres, together with the fact that the fundamental group of a generic fiber injects into the fundamental group of the product.
    
    Now, by Theorem \ref{thm: spinal_ii_6_4}, we have a one-parameter family of nearly Lefschetz fibrations associated to $\{\widehat{J}_\tau\}$ and at time $\tau = 1$, we have an almost complex structure $\widehat{J}_S$ making both the completed fibers and $\widehat{S}$ pseudoholomorphic. 
    Similar to the proof of Proposition \ref{prop:sympl_disk_multisection}, we take cutoffs $S_R$ of $\widehat{S}$ and $F_R$ of $\widehat{F}$ and may assume that $S_R$ has only isolated intersections with $F_R$, unless $S$ is symplectically isotopic to a smooth fiber. Therefore, we can similarly apply \cite[Lemma 6.31]{spinal_II} and conclude that $S$ is either symplectically isotopic to a smooth fiber or a positive multisection. 
\end{proof}

\section{Handle diagrams for multisection complements}\label{sec:algo}
In this section we present an algorithm for drawing Weinstein handlebody diagrams for multisection complements in $B^4$, exemplifying the Weinstein structure constructed earlier on nearly Lefschetz fibrations. 

The section is lengthy, so we outline its organization here. Background on monotonic quasipositive braids is provided in \S\ref{subsec:monotonic}, followed by some lemmas which justify the algorithm in \S\ref{subsec:lemmas-for-algo}. The algorithm itself is described and justified in \S\ref{subsec:algo-NLF}, and several examples are computed in \S\ref{subsec:algo-exs}. In \S\ref{subsec:PANLFtoPALF}, we show by example how our algorithm can be used to systematically convert a PANLF to a PALF. Finally, we conclude by identifying some open questions in \S\ref{subsec:questions}. 

\subsection{Monotonic quasipositive braided surfaces}\label{subsec:monotonic}

For technical convenience, we work with \emph{monotonic quasipositive braids}, a framework named by Apostolakis-Piergallini-Zuddas \cite{apostolakis2013lefschetz}. We recall their definition and relevant result.

Let $\sigma_1, \dots, \sigma_{n-1}$ denote the standard Artin generators of the $n$-strand braid group. A \emph{positive band} is any conjugate of $\sigma_i$, and a braid is \emph{quasipositive} if it admits a factorization into positive bands. Following \cite{apostolakis2013lefschetz}, a \emph{monotonic positive band} is a word of the form 
\[
(\sigma_{i+1}^{\epsilon_{i+1}}\cdots \sigma_{j-2}^{\epsilon_{j-2}}\sigma_{j-1}^{\epsilon_{j-1}})^{-1}\, \sigma_i\, (\sigma_{i+1}^{\epsilon_{i+1}}\cdots \sigma_{j-2}^{\epsilon_{j-2}}\sigma_{j-1}^{\epsilon_{j-1}})
\]
for some $i< j$ and $\epsilon_{\ell}\in \{-1,1\}$, and a \emph{monotonic quasipositive braid} is one admitting a factorization into monotonic positive bands. Finally, we call a link $L\subset S^3$ \emph{monotonic quasipositive} if it is the closure of a monotonic quasipositive braid. 

Informally, a monotonic positive band can be described as follows. Consider the $n$ strands of a trivial braid as positioned at heights $z=1,2,\dots, n$ in $\R^3$. Given a positive band, we can identify a core arc $c$ with endpoints on the original braid strands onto which the band retracts. A band is monotonic if the height function $(x,y,z) \mapsto z$ restricts to $c$ without critical points, i.e.\ if $c$ appears approximately vertical with respect to the $z$-height. See the left of Figure \ref{fig:monotonic} for a non-monotonic band (in blue).

\begin{proposition}[{\cite[Proposition 3]{apostolakis2013lefschetz}}]
Let $\beta$ be a quasipositive braid. Then $\beta$ is positive Markov stabilization equivalent to a monotonic quasipositive braid. In particular, if $L\subset S^3$ is a quasipositive link, then $L$ is monotonic quasipositive.    
\end{proposition}

The mechanism of the proof is depicted in Figure \ref{fig:monotonic}, where we convert a standard (non-monotonic) quasipositive factorization of the knot $m(8_{20})$ into a monotonic quasipositive factorization by first Markov stabilizing in an appropriate location, and then sliding the non-monotonic band across the stabilization band.

\begin{figure}[ht]
	\centering
	\begin{overpic}[width=\linewidth]{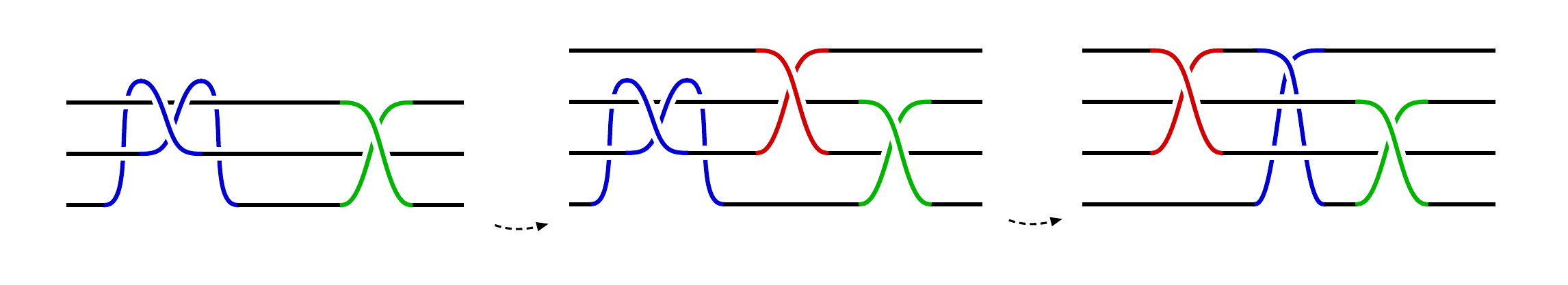}
      
	\end{overpic}
	\caption{On the left is a minimal index quasipositive factorization of the knot $m(8_{20})$. In the middle, we perform a quasipositive Markov stabilization by increasing the braid index with an additional red positive band. On the right, we have slid the blue non-monotonic band across the red positive band. The result is a monotonic quasipositive braid.}
	\label{fig:monotonic}
\end{figure}

\subsection{Some lemmas}\label{subsec:lemmas-for-algo}

Given an abstract open book decomposition $(\Sigma, \phi)$ for a contact $3$-manifold $(Y, \xi)$ with $\phi$ presented as a factorization into Dehn twists, one may produce a contact-$(\pm 1)$ surgery diagram for $(Y, \xi)$ by embedding $\Sigma$ into $(S^3, \xi_{\mathrm{st}})$ as a symplectic ribbon neighborhood of a Legendrian graph with appropriately chosen $1$-handles, and then applying the Legendrian realization principle to Reeb shifts of the Dehn twist generators to draw their front projections. Right-handed twists each induces a contact-$(-1)$ surgery, and left-handed twists each induces a contact-$(+1)$ surgery. A practical implementation was first detailed by Avdek \cite{avdek2013surgery} in the case that the page has connected boundary and the Dehn twists involve standard Lickorish generators; a fully general algorithm for any page and any twist curves was recently described by Stenhede \cite{stenhede2026algorithmlegendrianrealizecurve}.  

Our algorithm can be seen as a further generalization, as the Weinstein handle attaching sphere described by Lemma \ref{lemma:NLF_2handle} is not embedded in a page of the open book decomposition, but rather in a neighborhood of the page. To that end, we begin with the following lemma. 

\begin{lemma}\label{lemma:algo1}
Consider the contact manifold $(P\times \R_z, \ker(dz + \lambda))$, where $P$ is a pair of pants and $(P, \lambda)$ is a Liouville domain. Let $C$ be a smoothly embedded curve in $P\times \R$ as on the left of Figure \ref{fig:2handle-graph}. Then there is a Legendrian realization of $C$ in $P\times \R_z$ and a contact embedding $\iota: P\times \R_z \hookrightarrow (S^3, \xi_{\mathrm{st}})$ such that 
\begin{enumerate}
    \item $\iota(P\times \{0\})$ is a symplectic ribbon neighborhood of the black Legendrian graph $G$ depicted on the right side of Figure \ref{fig:2handle-graph}, and 
    \item $\iota(C)$ is given by the red front projection.
\end{enumerate}
Moreover, the Lagrangian projection of $C$ in $P$ agrees with the left of Figure \ref{fig:2handle-graph}. 
\end{lemma}

\begin{figure}[ht]
	\centering
	\begin{overpic}[width=\linewidth]{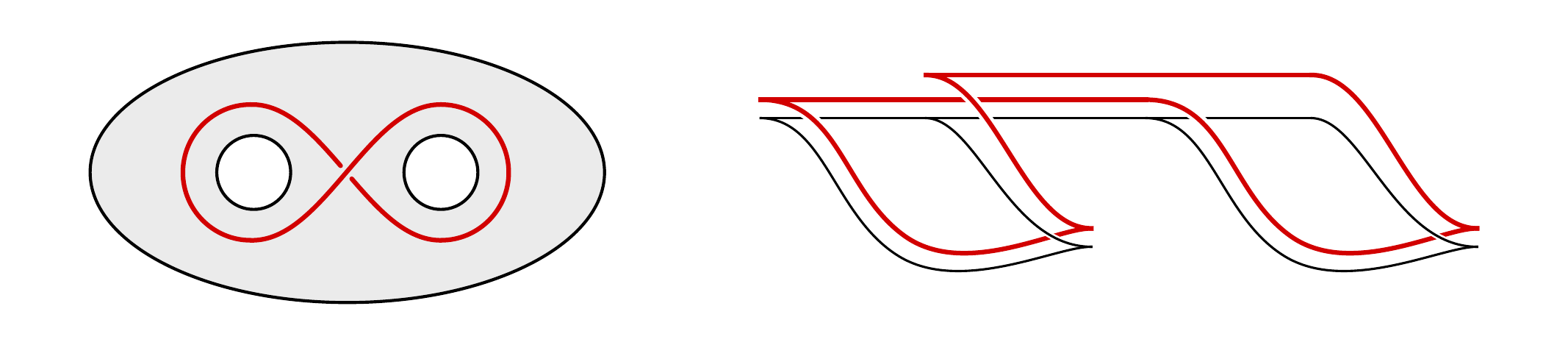}
        \put(300,18){\small $G$}
        \put(296,38){\small \textcolor{darkred}{$\iota(C)$}}
        \put(160,18){\small $P$}
        \put(143,43){\small \textcolor{darkred}{$C$}}
	\end{overpic}
	\caption{Legendrian realization of $C$ in the front projection.}
	\label{fig:2handle-graph}
\end{figure}

\begin{proof}
By Giroux flexibility, we may assume that after an isotopy of $P = P\times \{0\}$ in $P\times \R_z$ (equivalently, by a Liouville homotopy of $\lambda$) the oriented characteristic foliation of $\xi$ on $P$ is given as in the left side of Figure \ref{fig:columns}. In particular, the Liouville skeleton of $P$ is the bold gray trivalent Legendrian barbell graph in Figure \ref{fig:columns},
consisting of singular points of the characteristic foliation. 

\begin{figure}[ht]
	\centering
	\begin{overpic}[width=\linewidth]{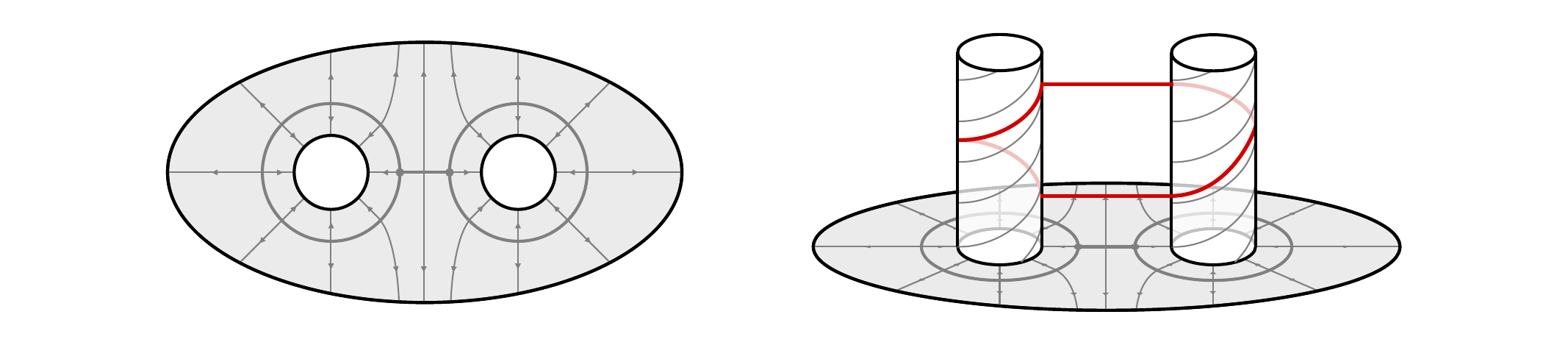}
        \put(294,48){\footnotesize \textcolor{darkred}{$C_{z_0}$}}
        \put(294,76){\footnotesize \textcolor{darkred}{$C_{z_1}$}}
        \put(246,53){\footnotesize \textcolor{darkred}{$C_A$}}
        \put(344,55){\footnotesize \textcolor{darkred}{$C_B$}}
	\end{overpic}
	\caption{Piecewise-linear Legendrian realization of $C$.}
	\label{fig:columns}
\end{figure} 

Let $A,B$ denote the two inner components of $\partial P$. Because $A$ and $B$ are convex boundary components of the Liouville domain $P$, a collar neighborhood of $A$ can be identified with $(-\delta,0]_s \times A$ with Liouville form $e^s\, d\theta$, where $\theta$ is a coordinate on $A$ oriented as the boundary of $P$. The oriented characteristic foliation on the vertical cylinder $A \times \R$ induced by the contact form $dz + e^s\, d\theta$ is then directed by $\partial_z + \partial_{\theta}$. The same model holds near $B$. This is drawn on the columns on the right side of Figure \ref{fig:columns}.

We identify a piecewise-smooth Legendrian $C_{\mathrm{PL}} = C_A \cup C_{z_0} \cup C_{z_1} \cup C_B$ as follows. Let $C_{z_0}, C_{z_1}$ be Reeb shifts at $z=z_0,z_1$ with $z_0 < z_1$ of the handle of the barbell graph, extended along the characteristic foliation of (parallel copies of) $P$ so that their endpoints lie on $A \times \{z_i\}$ and $B \times \{z_i\}$, $i=0,1$. Let $C_A$ be a leaf of the characteristic foliation of the $A\times \R$ which makes a full rotation around $A$ and connects $A \times \{z_0\}$ and $A \times \{z_1\}$; note that we choose $z_0 < z_1$ accordingly with respect to the action of the characteristic foliation. Define $C_B$ analogously. Then, we define $C$ to be the natural Legendrian smoothing of $C_{\mathrm{PL}}$. After a further generic perturbation to produce transverse double points of the projection, the Lagrangian projection to $P$ of $C$ agrees with the left side of Figure \ref{fig:2handle-graph}, giving the "moreover" statement of the lemma. 

\begin{figure}[ht]
	\centering
	\begin{overpic}[width=0.9\linewidth]{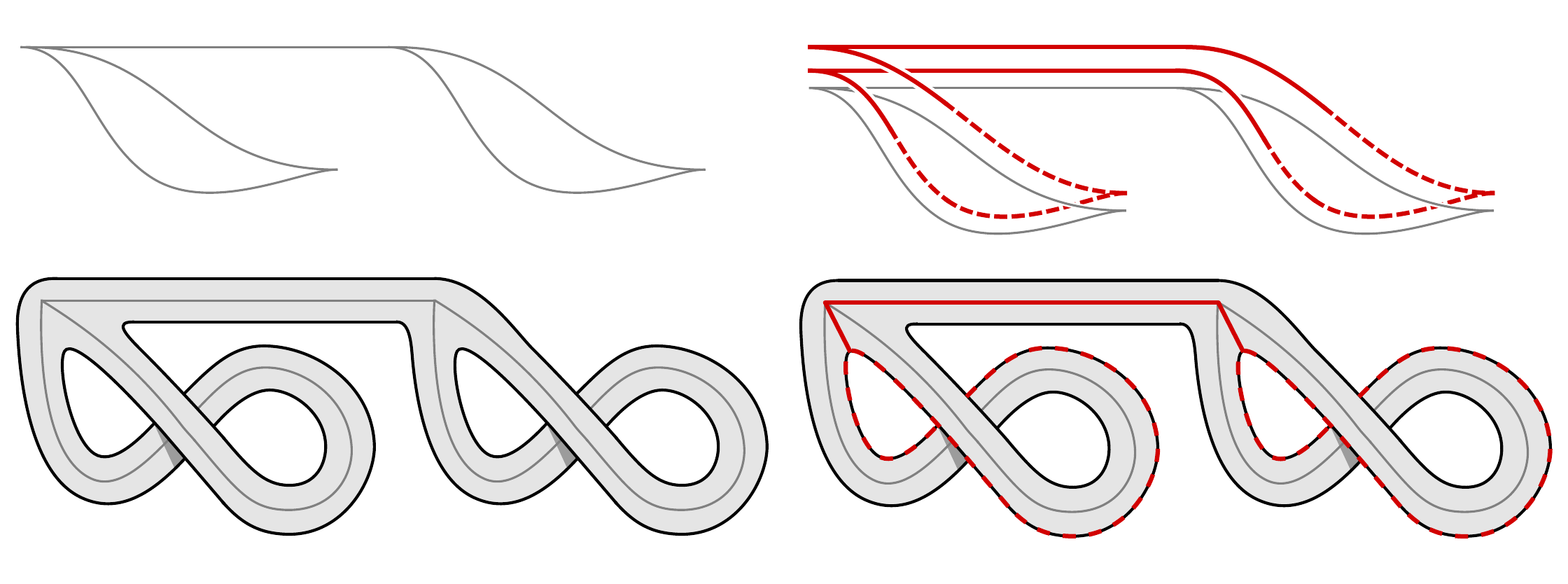}
        \put(5,110){\small \textcolor{gray}{$G'$}}
	\end{overpic}
	\caption{Front projections (top) and Lagrangian projections (bottom) of $G'$ and its symplectic ribbon neighborhood.}
	\label{fig:PLmodel}
\end{figure} 

Next we describe the embedding into $(S^3, \xi_{\std})$. It is a standard fact that a $1$-jet neighborhood of the trivalent Legendrian graph $G'$ on the left of Figure \ref{fig:PLmodel}, i.e.\ that a symplectic ribbon neighborhood of $G'$ is a pair of pants. Here we use $G'$ as an intermediate graph, as it is more directly adapted from the characteristic foliation of $P$; we then pass to $G$ primarily for ease of drawing diagrams. 
Front projections of $C_{z_0}, C_{z_1}$ then correspond to the dark red segments on the right side of the figure, and front projections of $C_A$ and $C_B$ are dashed. The Legendrian graph $G'$ is Legendrian graph isotopic to $G$, and the Legendrian smoothing of $C_{\mathrm{PL}}$ is Legendrian isotopic to $\iota(C)$. This completes the proof. 
\end{proof}

Lemma \ref{lemma:algo1} can be interpreted as identifying the Legendrian attaching sphere of a $2$-handle corresponding to a standard Artin braid generator. This is already sufficient for some interesting examples; see \S\ref{subsubsec:hopf-ex} and \S\ref{subsubsec:trefoil-ex} below. We obtain a generalization to monotonic quasipositive band as follows. 

\begin{figure}[ht]
	\centering
	\begin{overpic}[width=0.8\linewidth]{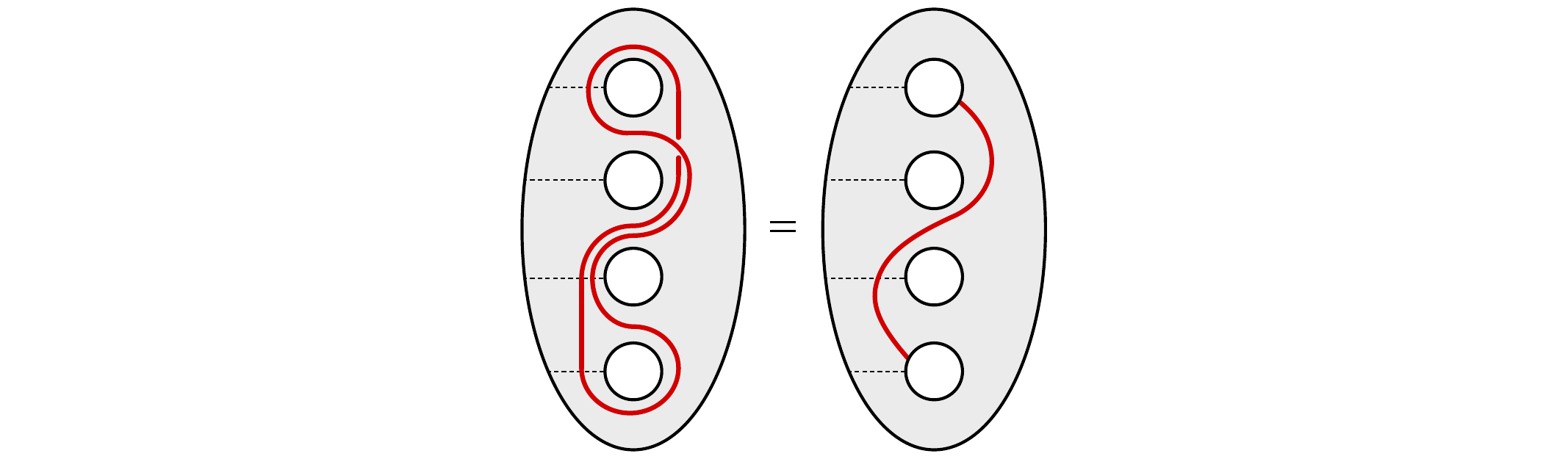}
        
	\end{overpic}
	\caption{The $2$-handle attaching sphere (left) associated to a monotonic quasipositive band core (right).}
	\label{fig:qp-band}
\end{figure} 

First, given a punctured disk representing a disk transverse to the given quasipositive braid, we record dashed horizontal lines to track the intersection with the braid axis disks. Associated to a monotonic quasipositive band is the $2$-handle locally given by Lemma \ref{lemma:NLF_2handle}; for example, see the left side of Figure \ref{fig:qp-band}. In standard fashion we switch to recording the properly embedded arc in the punctured disk which is the core of the band, so that the corresponding mapping class element is the half Dehn twist along the arc; this is depicted in the second frame of the same figure. 

\begin{lemma}\label{lemma:qp_band_algo}
Let $S$ be a compact planar surface with $n$ inner boundary components and marked braid axis disk intersections, and let $\lambda$ be a Liouville form so that $(S, \lambda)$ is a Liouville domain. Let $b$ be a properly embedded arc representing the core of a monotonic quasipositive band generator, and let $C$ denote the smoothly embedded curve in $P\times \R$ which is the associated $2$-handle attaching sphere furnished by Lemma \ref{lemma:NLF_2handle}. Then there is a Legendrian realization of $C$ in $(P\times \R, \ker(dz + \lambda))$ and a contact embedding $\iota: P\times \R \hookrightarrow (S^3, \xi_{\mathrm{st}})$ such that 
\begin{enumerate}
    \item $\iota(P\times \{0\})$ is a symplectic ribbon neighborhood of the Legendrian graph depicted at the top of Figure \ref{fig:qp-recipe}, and 
    \item $\iota(C)$ is obtained by assembling the local pieces, according to $b$, depicted in red in the rest of the figure.
\end{enumerate}
Moreover, the Lagrangian projection of $C$ in $P$ agrees with the projection of $C$ in $S$ as given locally by Lemma \ref{lemma:NLF_2handle}.
\end{lemma}

\begin{figure}[ht]
	\centering
	\begin{overpic}[width=\linewidth]{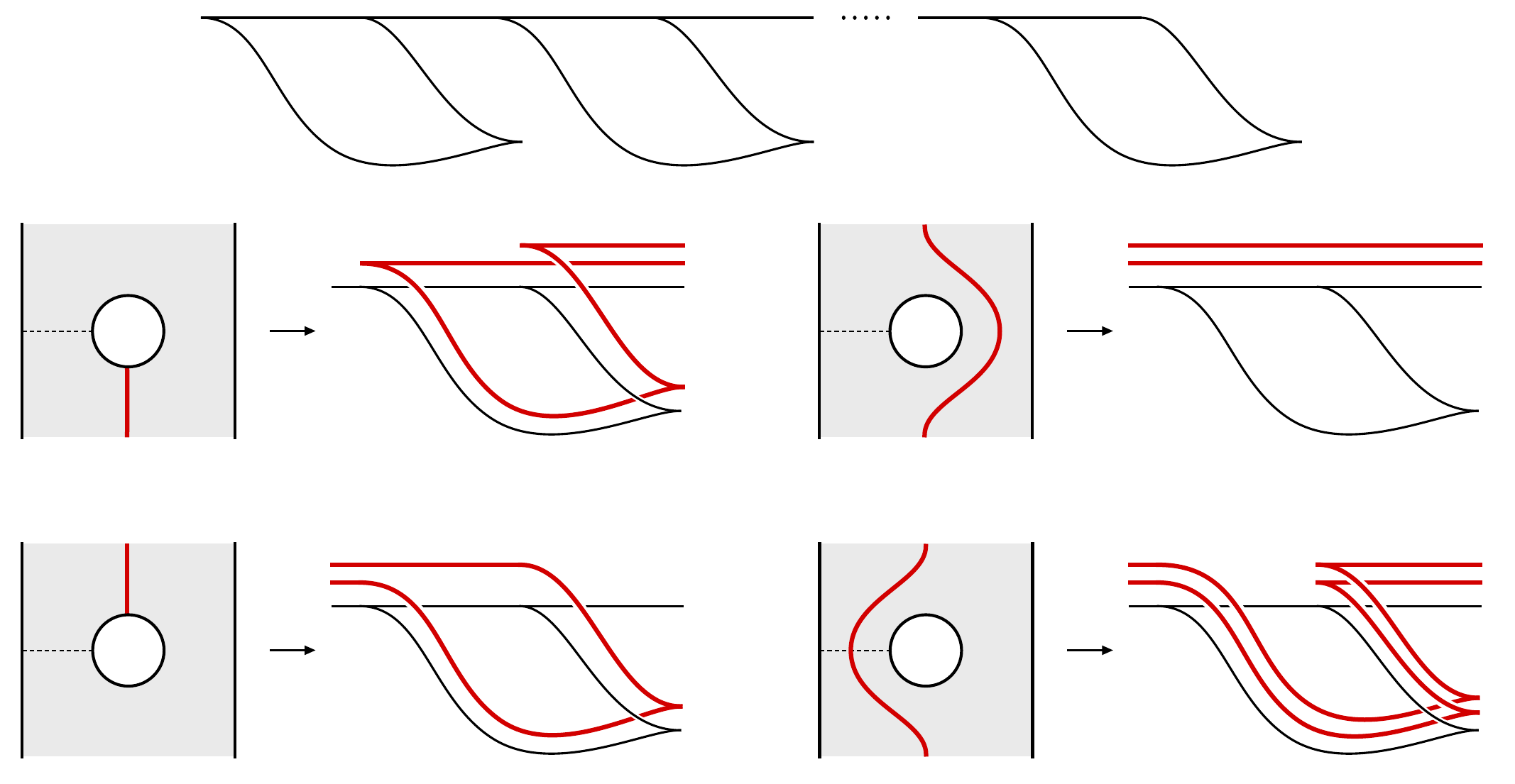}
        
	\end{overpic}
	\caption{The recipe for witnessing the $2$-handle attaching sphere associated to a monotonic quasipositive band core $b$.}
	\label{fig:qp-recipe}
\end{figure} 

\begin{proof}
This is a straightforward adaptation of the proof of Lemma \ref{lemma:algo1}. Specifically, we realize a characteristic foliation on $P$ so that in a neighborhood of each inner boundary component (that is, in the regions depicted in Figure \ref{fig:qp-recipe}) the singular set of the foliation consists of an $S^1$ surrounding the boundary component and two arcs connecting this $S^1$ to the other regions. Piecewise smooth lifts of the branches of the singular set, as in Figure \ref{fig:columns}, smooth to the front projections depicted in Figure \ref{fig:qp-recipe}. Monotonicity then allows us to decompose a given quasipositive band generator into the elementary pieces in Figure \ref{fig:qp-recipe}. (Without assuming monotonicity, the recipe would be more complicated, necessitating an adaptation of Stenhede's work \cite{stenhede2026algorithmlegendrianrealizecurve}.)
\end{proof}

\subsection{The algorithm}\label{subsec:algo-NLF}

Let $\beta = b_1b_2\cdots b_k$ be a quasipositive $n$-braid with each $b_i$ a quasipositive band generator. Assume that the associated quasipositive surface $\Sigma(\beta) \subset B^4$ filling $\hat{\beta} = \partial \Sigma(\beta)$ is connected. To draw a Weinstein handlebody diagram for the exterior $B^4 \setminus \nu(\Sigma(\beta))$: 

\begin{enumerate}
    \item Let $G_n = G_n^1$ denote a Legendrian $n$-barbell graph as depicted in Figure \ref{fig:algo-start}. Draw $k-1$ small Reeb pushoffs $G_n^2, \dots, G_n^k$ of $G_n^1$, together with $n$ dotted circle $1$-handles linking the graphs as indicated in the figure. 

    \begin{figure}[ht]
	\centering
	\begin{overpic}[width=\linewidth]{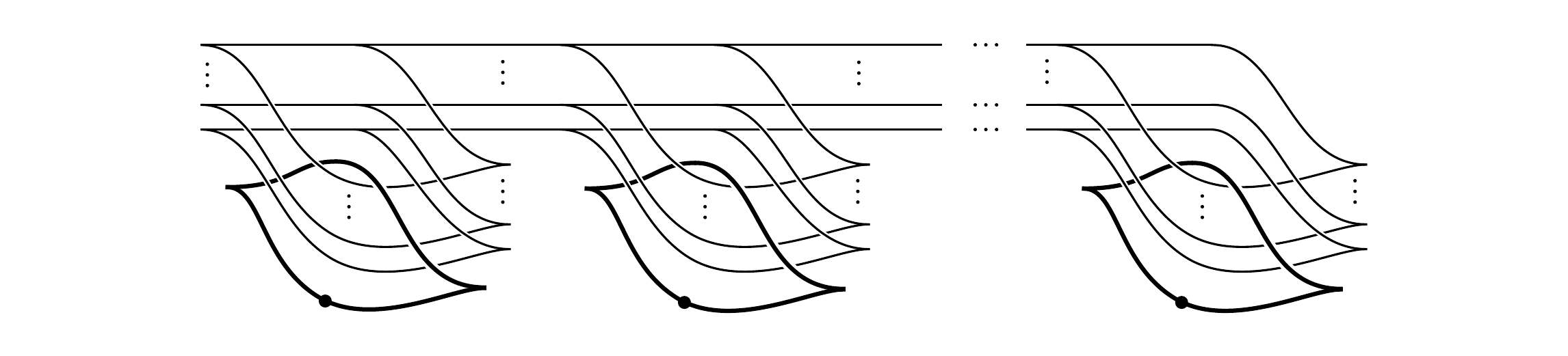}
        \put(43,57.5){\tiny $G_n^1$}
        \put(43,67.5){\tiny $G_n^2$}
        \put(43,82.5){\tiny $G_n^k$}
	\end{overpic}
	\caption{Step (1) of the algorithm.}
	\label{fig:algo-start}
\end{figure} 

    \item In a small neighborhood of $G_n^i$, apply Lemma \ref{lemma:qp_band_algo}, specifically the recipe on elementary pieces in Figure \ref{fig:qp-recipe}, to Legendrian realize the attaching sphere associated to the quasipositive band $b_i$. 

    \item Erase the graphs $G_n^1, \dots, G_n^k$ and attaching Weinstein handles along all of the Legendrian realizations from the previous step. 

    \item Optionally, perform standard Legendrian Kirby calculus to simplify the resulting presentation or convert the dotted circle $1$-handles into standard Gompf form. 
\end{enumerate}

\begin{proof}[Proof of algorithm veracity.]
Let $S$ be a compact planar surface with $n$ inner boundary components and marked braid axis disk intersections. By \S \ref{subsec:panlf_weinstein}, the Weinstein structure on $B^4 \setminus \nu(\Sigma(\beta))$ is obtained by first starting with the standard subcritical Weinstein structure on $S \times D^2$ and then attaching $2$-handles according to Lemma \ref{lemma:NLF_2handle}. 

Viewing $S\times D^2 \to D^2$ as a trivial Lefschetz fibration, the induced open book decomposition on the boundary $\partial (S \times D^2) \cong \#^n (S^1 \times S^2)$ has page diffeomorphic to $S$ and identity monodromy. Note that the subcritical Weinstein structure on $S \times D^2 \cong \natural^n (S^1 \times B^3)$ is obtained by digging out the $n$-many braid axis disks, given by dotted circle $1$-handles, which we may Legendrian realize. A fiber surface of the induced open book on the boundary may be identified with the symplectic ribbon neighborhood of $G_n^1$.  A sequence of successive pages of the open book is obtained by short Reeb pushoffs of $G_n^1$. The Weinstein handles described by Lemma \ref{lemma:NLF_2handle} are then concretely witnessed by applying Lemma \ref{lemma:qp_band_algo} to each band generator along these successive pages. 
\end{proof}

\subsection{Examples}\label{subsec:algo-exs} We present three examples of Weinstein structures on complements of quasipositive surfaces: the annulus filling of a positive Hopf link, the genus-$1$ filling of a trefoil, and a slice disk filling $m(8_{20})$.

\subsubsection{The Hopf link}\label{subsubsec:hopf-ex} The simplest non-trivial example is the positive Hopf link, which is the closure of the positive $2$-braid $\beta = \sigma_1^2$ and thus bounds a quasipositive annulus in $B^4$; see the top left of Figure \ref{fig:hopf-ex}. The standard method \cite[\S 6]{gompf1999manifolds} for carving out ribbon surfaces produces a smooth handlebody decomposition of the exterior as drawn below the braid closure. It is straightforward to find a Stein realization of the handlebody and has been studied before; see, for example, \cite[\S 4.4]{Casals_Murphy_affine}. We verify that our perspective gives a new route to the same diagram.

The corresponding nearly Lefschetz fibration supporting the exterior of the Hopf annulus has as smooth fiber a pair of pants, and has two exotic fibers corresponding to the two positive bands. The associated Weinstein handlebody structure is then $(S^1\times B^3) \natural (S^1 \times B^3)$ with two $2$-handle attachments, and the algorithm described in \S \ref{subsec:algo-NLF} yields the Weinstein handlebody diagram in the top right of Figure \ref{fig:hopf-ex}. We have kept the two Legendrian graphs, the skeleta of the two relevant pages of the open book $(S^1\times S^2) \# (S^1 \times S^2)$, for clarity. 

\begin{figure}[ht]
	\centering
	\begin{overpic}[width=\linewidth]{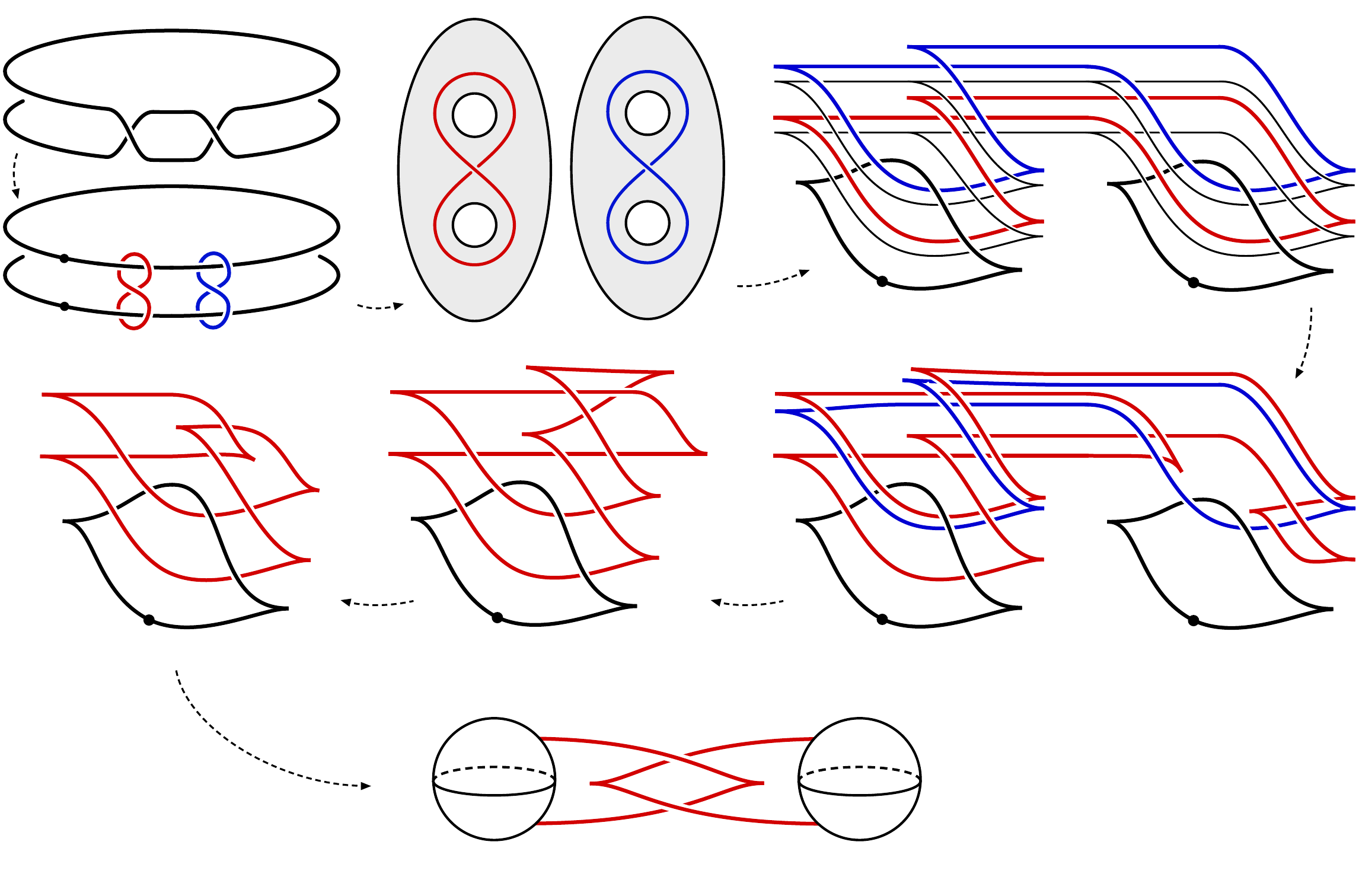}
        \put(40,205){\footnotesize \textcolor{darkred}{$0$}}
        \put(65,205){\footnotesize \textcolor{darkblue}{$0$}}

        \put(145,188){\footnotesize \textcolor{darkred}{$\Lambda_1$}}
        \put(199,188){\footnotesize \textcolor{darkblue}{$\Lambda_2$}}

        \put(410,245){\footnotesize \textcolor{darkred}{$(-1)$}}
        \put(410,255){\footnotesize \textcolor{darkblue}{$(-1)$}}

        \put(410,145){\footnotesize \textcolor{darkred}{$(-1)$}}
        \put(410,155){\footnotesize \textcolor{darkblue}{$(-1)$}}

        \put(210,120){\footnotesize \textcolor{darkred}{$(-1)$}}

        \put(80,150){\footnotesize \textcolor{darkred}{$(-1)$}}

        \put(205,50){\footnotesize \textcolor{darkred}{$(-1)$}}
	\end{overpic}
    \vskip-0.4cm
	\caption{The Weinstein structure on the complement of the Hopf annulus.}
	\label{fig:hopf-ex}
\end{figure}

In the frame below, we have erased the graphs and slid $\Lambda_1$ up across the contact-$(-1)$ surgery along $\Lambda_2$. After this, the $2$-handle attached along $\Lambda_2$ may be canceled with the right most dotted $1$-handle. After a further Legendrian isotopy and then a conversion to a standard $1$-handle picture, we recover the diagram in \cite[Figure 70]{Casals_Murphy_affine}.

\subsubsection{The trefoil}\label{subsubsec:trefoil-ex} The right-handed trefoil, the closure of the positive braid $\beta = \sigma_1^3$, is similar to the Hopf link. Figure \ref{fig:trefoil-ex} depicts the analogous implementation of the algorithm as in Figure \ref{fig:hopf-ex}, with one additional braid generator. From the top right diagram, we slide both the red and blue Legendrians across the green one near the right-most $1$-handle. Afterwards, in the middle, the green $2$-handle cancels said $1$-handle. After some further Legendrian isotopies and then conversion to a standard $1$-handle framework, we produce the handlebody diagram depicted at the botom.

\begin{figure}[ht]
	\centering
	\begin{overpic}[width=\linewidth]{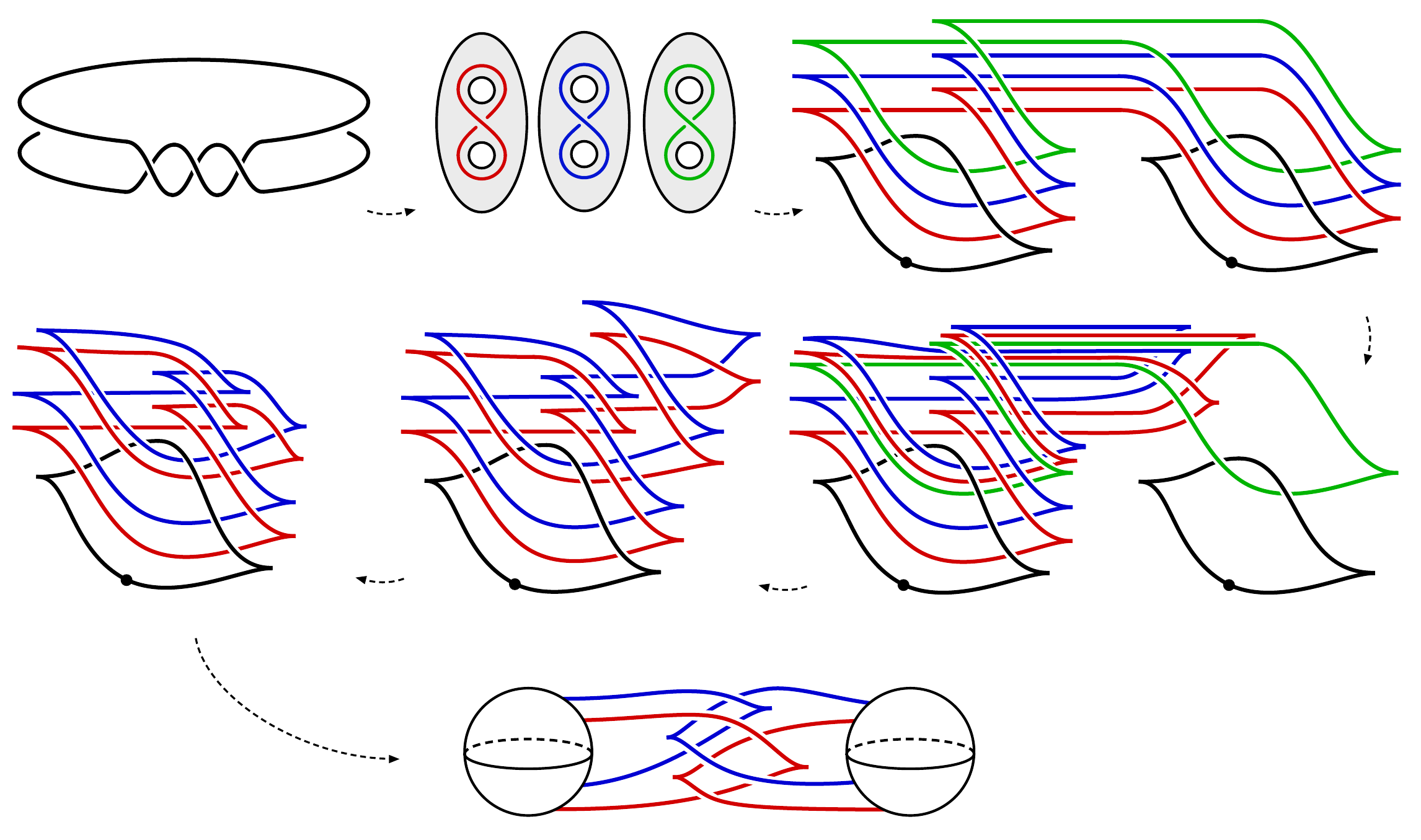}

        \put(410,225){\footnotesize \textcolor{darkred}{$(-1)$}}
        \put(410,235){\footnotesize \textcolor{darkblue}{$(-1)$}}
        \put(410,245){\footnotesize \textcolor{darkgreen}{$(-1)$}}

        \put(410,130){\footnotesize \textcolor{darkred}{$(-1)$}}
        \put(410,140){\footnotesize \textcolor{darkblue}{$(-1)$}}
        \put(410,150){\footnotesize \textcolor{darkgreen}{$(-1)$}}

        \put(210,90){\footnotesize \textcolor{darkred}{$(-1)$}}
        \put(210,100){\footnotesize \textcolor{darkblue}{$(-1)$}}

        \put(85,135){\footnotesize \textcolor{darkred}{$(-1)$}}
        \put(85,145){\footnotesize \textcolor{darkblue}{$(-1)$}}

        \put(195,50){\footnotesize \textcolor{darkred}{$(-1)$}}
        \put(220,50){\footnotesize \textcolor{darkblue}{$(-1)$}}
	\end{overpic}
	\caption{The Weinstein structure on the complement of the trefoil filling.}
	\label{fig:trefoil-ex}
\end{figure}

\subsubsection{$m(8_{20})$}
The previous two examples (the Hopf link and the trefoil) are strongly quasipositive, thus their corresponding braided surfaces are embedded in $S^3$. The simplest knot which is quasipositive but not strongly quasipositive is $m(8_{20})$, so its corresponding braided surface is an immersed ribbon surface in $S^3$. A second difference between $m(8_{20})$ and the prior examples is that the Hopf link and trefoil admit Legendrian representatives which are exact Lagrangian fillable, whereas $m(8_{20})$ does not. Indeed, $m(8_{20})$ is slice, but $\overline{\mathrm{tb}}(m(8_{20})) = -2$; any Legendrian bounding a Lagrangian disk satisfies $\mathrm{tb} = -1$. 

A handlebody diagram for the exterior of a quasipositive slice disk filling $m(8_{20})$ is summarized in Figure \ref{fig:m820}. We begin with the monotonic quasipositive braid whose closure is $m(8_{20})$, as described in Figure \ref{fig:monotonic}, and use the algorithm in \S\ref{subsec:algo-NLF} to produce the first surgery diagram. The rest of the figure is a sequence of modifications, for the sake of simplicity. First, the blue $2$-handle is slid up across the green $2$-handle near the right-most dotted $1$-handle, and said $1$-handle is canceled by the green $2$-handle. Then we slide the two blue strands passing through the remaining right-most $1$-handle across the red $2$-handle, and cancel said $1$-handle with the red $2$-handle. The rest of the figure depicts a sequence of standard Legendrian isotopies and conversion from dotted circles to traditional $1$-handles. 

\begin{figure}[ht]
	\centering
	\begin{overpic}[width=\linewidth]{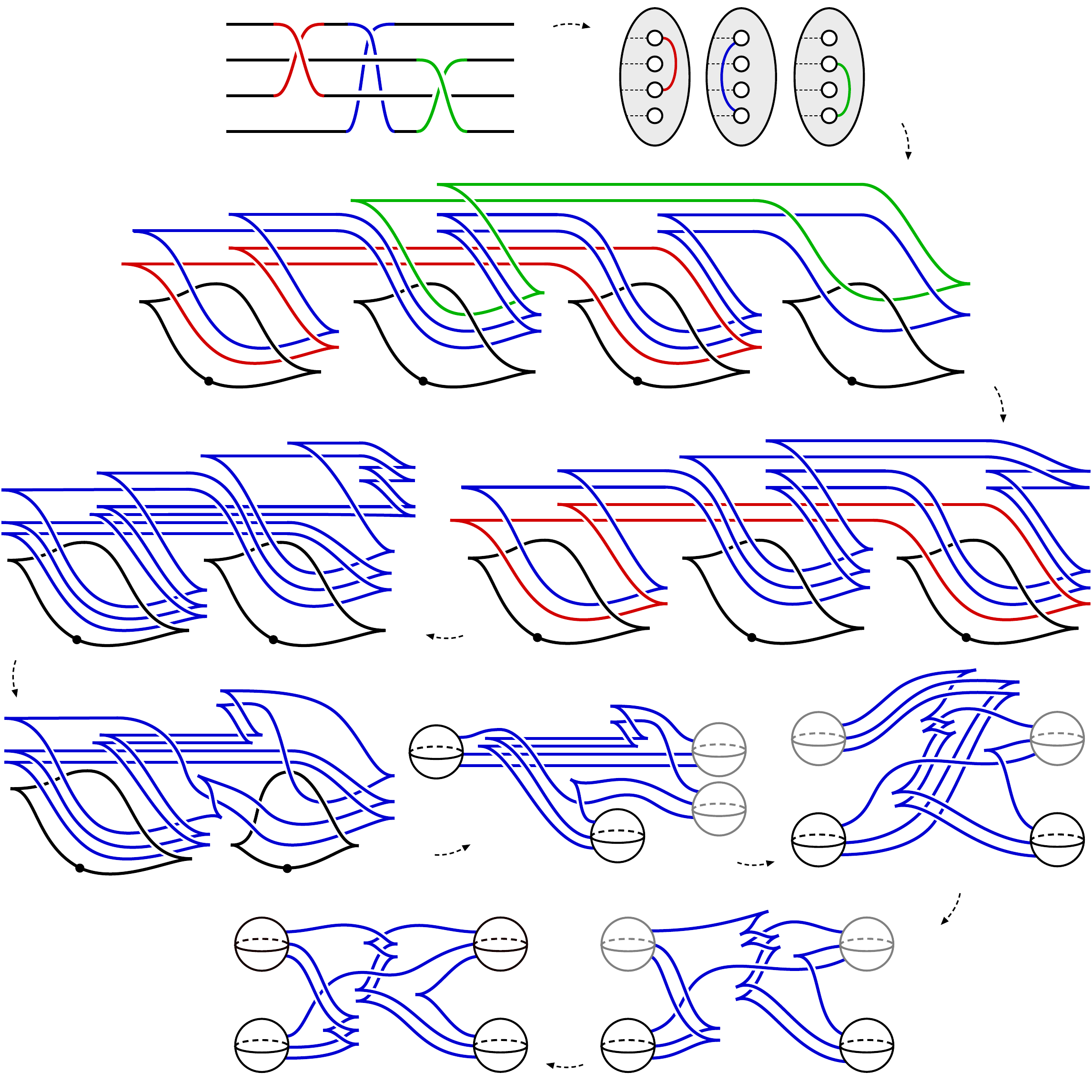}

        \put(385,300){\footnotesize \textcolor{darkred}{$(-1)$}}
        \put(385,310){\footnotesize \textcolor{darkblue}{$(-1)$}}
        \put(385,320){\footnotesize \textcolor{darkgreen}{$(-1)$}}

        \put(400,115){\footnotesize \textcolor{darkblue}{$(-1)$}}

        \put(195,100){\footnotesize \textcolor{darkblue}{$(-1)$}}

        \put(50,250){\footnotesize \textcolor{darkblue}{$(-1)$}}
        \put(60,150){\footnotesize \textcolor{darkblue}{$(-1)$}}

        \put(145,15){\footnotesize \textcolor{darkblue}{$(-1)$}}
        \put(287,15){\footnotesize \textcolor{darkblue}{$(-1)$}}

        \put(212,250){\footnotesize \textcolor{darkblue}{$(-1)$}}
        \put(230,250){\footnotesize \textcolor{darkred}{$(-1)$}}
	\end{overpic}
	\caption{The Weinstein structure on the complement of a $m(8_{20})$ quasipositive slice disk.}
	\label{fig:m820}
\end{figure}

\subsection{From PANLF to PALF, by example}\label{subsec:PANLFtoPALF}

Our algorithm gives a systematic way to produce a PALF on a multisection complement. First, produce a Weinstein surgery diagram from the quasipositive presentation as above, convert it into standard Gompf form, then use an adaptation of known techniques to identify an open book decomposition of $\natural^n (S^1 \times S^2)$ on which the $2$-handle attaching spheres are supported on a page, after which additional right-handed Dehn twists along the attaching spheres gives a PALF presentation of the Weinstein domain. For concreteness, we implement this in the case of the Hopf annulus $A$.  

Let $W:= B^4 \setminus \nu(A)$. Following Figure \ref{fig:hopf-ex}, a Weinstein handlebody diagram for $W$ is obtained by attaching a $2$-handle along the red Legendrian attaching sphere, $\Lambda$, to $S^1 \times B^3$. Using Avdek \cite{avdek2013surgery}, we next build an open book supporting $(S^1 \times S^2, \xi_{\mathrm{st}})$ in which $\Lambda$ embeds on a page, obtaining an open book for $\partial W$ by pre-composing the monodromy with an additional Dehn twist along $\Lambda$. Strictly speaking, Avdek describes how to construct an open book decomposition of $(S^3, \xi_{\mathrm{st}})$ which supports a prescribed Legendrian; we will modify the algorithm to instead work for $(S^1 \times S^2, \xi_{\mathrm{st}})$.\footnote{Alternatively, one may witness the $1$-handle as contact-$(+1)$ surgery along a dotted max-tb unknot $U$, then apply Avdek's algorithm in $S^3$ to the link $U \cup \Lambda$. This approach has the disadvantage that contact-$(+1)$ surgery along $U$ will add a negative Dehn twist to the resulting monodromy factorization.}

Avdek's algorithm for producing supporting open book decompositions from surgery diagrams is rooted in the following theorem of Giroux, stated as \cite[Theorem 2.4]{avdek2013surgery} and proven in \cite{etnyre_obd}. 

\begin{theorem}[{\cite{Giroux2002GeometrieDC}\label{thm:giroux}}]
Suppose that a closed co-oriented contact manifold $(Y,\xi)$ has a cell decomposition with the following properties. 
\begin{enumerate}
    \item The $1$-skeleton $G$ is a Legendrian graph.
    \item Each $2$-cell $D$ is a convex disk with piecewise-Legendrian boundary such that, after performing the unique Legendrian smoothing of $\partial D$, we have $\mathrm{tb}(\partial D) = -1$. 
    \item If $\Sigma$ denotes the Legendrian ribbon of $G$, then for every $2$-cell $D$, we have $|D\cap \partial \Sigma| = 2$.
    \item Every $3$-cell is tight. 
\end{enumerate}
Then $\Sigma$ is the page of a supporting open book decomposition of $(Y, \xi)$. 
\end{theorem}

A Legendrian $1$-skeleton $G$ for $(S^1\times S^2, \xi_{\mathrm{st}})$ may be obtained by attaching to $\Lambda$ three Legendrian arcs: two which bridge the crossings in the final front projection in Figure \ref{fig:hopf-ex}, and one which connects the cusps by passing through the $1$-handle. The top left of Figure \ref{fig:hopf-comp-obd} depicts these arcs in the front projection in black dots, and the diagram immediately below is a perturbed Lagrangian projection (note the twisting of the $1$-handle, c.f. \cite[Figure 72]{Casals_Murphy_affine}).

\begin{figure}[ht]
	\centering
	\begin{overpic}[width=0.9\linewidth]{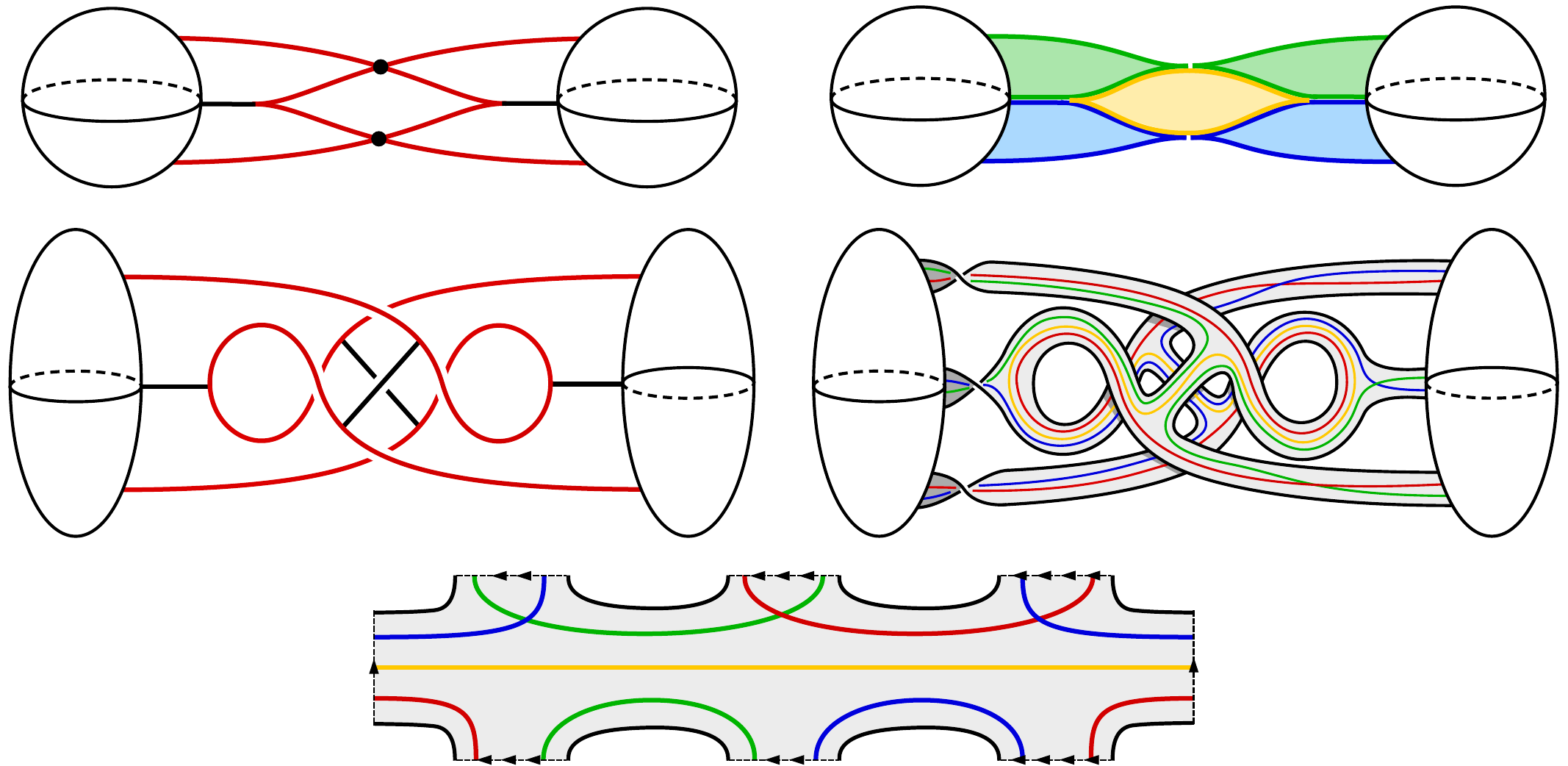}
	\end{overpic}
	\caption{A PALF for $W = B^4 \setminus \nu(A)$ with monodromy $\tau_g\circ \tau_y\circ \tau_b\circ \tau_r$.}
	\label{fig:hopf-comp-obd}
\end{figure}

The necessary set of three $2$-cells (in blue, green, and yellow) is depicted in the front projection in the top right of Figure \ref{fig:hopf-comp-obd}. Strictly speaking, this $2$-complex has a canonical neighborhood which is a standard contact handlebody $(S^1 \times D^2, \xi_{\mathrm{st}}) \cong ([-1,1]_t \times V, \ker(dt + \lambda_V))$ with two parallel dividing curves on its convex boundary, where $V$ is an annulus and $\lambda_V$ is a Liouville form on $V$. We obtain $(S^1 \times S^2, \xi_{\mathrm{st}})$ by simply identifying the annular positive region with the annular negative region by the identity map. 

The boundaries of the $2$-cells (denoted $L_b, L_g, L_y$ respectively) are drawn on the Weinstein ribbon neighborhood of the graph $G$, along with $\Lambda =: L_r$ itself in red, in the middle right of Figure \ref{fig:hopf-comp-obd}. By the proof of \cite[Theorem 4.8]{avdek2013surgery}, the monodromy of the open book of $(S^1 \times S^2, \xi_{\mathrm{st}})$ obtained from the cell decomposition in the top right of Figure \ref{fig:hopf-comp-obd} has monodromy given by right-handed Dehn twists around $L_b, L_y,L_g$, in that order, which we denote $\tau_g\circ \tau_y\circ \tau_b$. Then by \cite[Theorem 2.7]{avdek2013surgery} (see also \cite[Theorem 5.7]{etnyre_obd}), the additional contact-$(-1)$ surgery on $\Lambda$ which produces $\partial W$ results in an open book decomposition with an additional right-handed Dehn twist around $\Lambda = L_r$. We have thus proven:     

\begin{proposition}
The annulus multisection complement $W = B^4 \setminus \nu(A)$ admits the PALF given in Figure \ref{fig:hopf-comp-obd}.    
\end{proposition}

\subsection{Questions}\label{subsec:questions}

The process described in the previous section for passing from a PANLF to a PALF (beginning with a PANLF, producing a Weinstein handlebody diagram, then using a modified Avdek algorithm to produce a PALF) is by no means efficient. We thus pose the question: 

\begin{question}
Can one convert more directly the structure of a PANLF to a PALF by a local model described near exotic fibers?  
\end{question}

\noindent Likewise, it is not clear if the output PALF from the prior subsection is necessarily itself efficient. For instance, note that the resulting open book in Figure \ref{fig:hopf-comp-obd} has genus $2$. It is a famous open question --- see \cite[Remark 1.6]{avdek2013surgery} and \cite[Problem 3.45]{K3} --- as to whether there exists a contact manifold with support genus greater than $1$. We propose a modification of the question more suited to Weinstein fillings that may be more tractable. 

\begin{definition}[Weinstein support genus] The \emph{Weinstein support genus} of a Stein fillable contact $3$-manifold $(Y^3,\xi)$ is the minimal genus of a supporting open book decomposition with a positive allowable factorization.
\end{definition}

\begin{question}
    Is there a contact $3$-manifold with Weinstein support genus two?
\end{question}
The above example of the Hopf annulus gives a natural test case for this question.

\bibliographystyle{amsalpha}
\bibliography{references}

@misc{breen2026derivativelinkscontacttopology,
      title={Derivative links in contact topology}, 
      author={Joseph Breen and Alexander Zupan},
      year={2026},
      eprint={2609.30182},
      archivePrefix={arXiv},
      primaryClass={math.SG},
      url={https://arxiv.org/abs/2609.30182}, 
      note={\href{https://arxiv.org/abs/2609.30182}{\tt arXiv:2609.30182}}
}

@article {donaldson1999lefschetz,
    AUTHOR = {Donaldson, S. K.},
     TITLE = {Lefschetz pencils on symplectic manifolds},
   JOURNAL = {J. Differential Geom.},
  FJOURNAL = {Journal of Differential Geometry},
    VOLUME = {53},
      YEAR = {1999},
    NUMBER = {2},
     PAGES = {205--236},
      ISSN = {0022-040X,1945-743X},
   MRCLASS = {53D35 (57R17)},
  MRNUMBER = {1802722},
MRREVIEWER = {Burak\ Ozbagci},
       URL = {http://projecteuclid.org/euclid.jdg/1214425535},
}

@article {apostolakis2013lefschetz,
    AUTHOR = {Apostolakis, Nikos and Piergallini, Riccardo and Zuddas,
              Daniele},
     TITLE = {Lefschetz fibrations over the disc},
   JOURNAL = {Proc. Lond. Math. Soc. (3)},
  FJOURNAL = {Proceedings of the London Mathematical Society. Third Series},
    VOLUME = {107},
      YEAR = {2013},
    NUMBER = {2},
     PAGES = {340--390},
      ISSN = {0024-6115,1460-244X},
   MRCLASS = {55R55 (57M12 57N13 57R65)},
  MRNUMBER = {3092341},
MRREVIEWER = {I.\ Pop},
       DOI = {10.1112/plms/pds078},
       URL = {https://doi.org/10.1112/plms/pds078},
}

@incollection {eliashberg2018revisited,
    AUTHOR = {Eliashberg, Yakov},
     TITLE = {Weinstein manifolds revisited},
 BOOKTITLE = {Modern geometry: a celebration of the work of {S}imon
              {D}onaldson},
    SERIES = {Proc. Sympos. Pure Math.},
    VOLUME = {99},
     PAGES = {59--82},
 PUBLISHER = {Amer. Math. Soc., Providence, RI},
      YEAR = {2018},
      ISBN = {978-1-4704-4094-7},
   MRCLASS = {53D35 (57R17)},
  MRNUMBER = {3838879},
MRREVIEWER = {Paolo\ Ghiggini},
       DOI = {10.1090/pspum/099/01737},
       URL = {https://doi.org/10.1090/pspum/099/01737},
}

@misc{stenhede2026algorithmlegendrianrealizecurve,
  author = {Stenhede, Eric},
  title = {An algorithm to Legendrian realize a curve on a ribbon surface},
  journal = {arXiv preprint},
  note = {arXiv:2604.08010},
  year = {2026}
}

@inproceedings {chen2018splitting,
    AUTHOR = {Chen, Weimin},
     TITLE = {Contact splitting of symplectic {$\mathbb{Q}$}-homology
              {$\mathbb{CP}^2$}},
 BOOKTITLE = {Proceedings of the {G}\"okova {G}eometry-{T}opology
              {C}onference 2017},
     PAGES = {53--72},
 PUBLISHER = {Int. Press, Somerville, MA},
      YEAR = {2018},
      ISBN = {978-1-57146-360-9},
   MRCLASS = {57R17},
  MRNUMBER = {3838085},
MRREVIEWER = {Burak\ Ozbagci},
}

@article {gompf2013smoothstein,
    AUTHOR = {Gompf, Robert E.},
     TITLE = {Smooth embeddings with {S}tein surface images},
   JOURNAL = {J. Topol.},
  FJOURNAL = {Journal of Topology},
    VOLUME = {6},
      YEAR = {2013},
    NUMBER = {4},
     PAGES = {915--944},
      ISSN = {1753-8416,1753-8424},
   MRCLASS = {32Q55 (32Q28 32Q40 32T15 57R17)},
  MRNUMBER = {3145144},
MRREVIEWER = {Stefan\ Nemirovski},
       DOI = {10.1112/jtopol/jtt017},
       URL = {https://doi.org/10.1112/jtopol/jtt017},
}

@article {mark2022contacttype,
    AUTHOR = {Mark, Thomas E. and Tosun, B\"ulent},
     TITLE = {On contact type hypersurfaces in 4-space},
   JOURNAL = {Invent. Math.},
  FJOURNAL = {Inventiones Mathematicae},
    VOLUME = {228},
      YEAR = {2022},
    NUMBER = {1},
     PAGES = {493--534},
      ISSN = {0020-9910,1432-1297},
   MRCLASS = {57K33 (32E20 57K43)},
  MRNUMBER = {4392461},
MRREVIEWER = {Youlin\ Li},
       DOI = {10.1007/s00222-021-01083-9},
       URL = {https://doi.org/10.1007/s00222-021-01083-9},
}

@misc{mark2025weinsteindomainssymplecticmanifolds,
  author = {Mark, Thomas and Tosun, Bulent},
  title = {On Weinstein domains in symplectic manifolds},
  journal = {arXiv preprint},
  note = {arXiv:2512.04278},
  year = {2025}
}

@article {Lisca1997Tight,
    AUTHOR = {Lisca, P. and Mati\'c, G.},
     TITLE = {Tight contact structures and {S}eiberg-{W}itten invariants},
   JOURNAL = {Invent. Math.},
  FJOURNAL = {Inventiones Mathematicae},
    VOLUME = {129},
      YEAR = {1997},
    NUMBER = {3},
     PAGES = {509--525},
      ISSN = {0020-9910,1432-1297},
   MRCLASS = {57R57 (53C15 57R15)},
  MRNUMBER = {1465333},
MRREVIEWER = {John\ B.\ Etnyre},
       DOI = {10.1007/s002220050171},
       URL = {https://doi.org/10.1007/s002220050171},
}

@article{Lin2016Caps,
  author = {Lin, Francesco},
  title = {Exact {Lagrangian} caps of {Legendrian} knots},
  journal = {J. Symplectic Geom.},
  volume = {14},
  number = {1},
  pages = {269--295},
  year = {2016}
}

@misc{Murphy2013Exact,
  author = {Murphy, Emmy},
  title = {Closed exact {Lagrangians} in the symplectization of contact manifolds},
  journal = {arXiv preprint},
  note = {arXiv:1304.6620},
  year = {2013}
}

@misc{DimitroglouRizellGolovko2024NonRegular,
  author = {Dimitroglou Rizell, Georgios and Golovko, Roman},
  title = {Instability of {Legendrian} knottedness, and non-regular {Lagrangian} concordances of knots},
  journal = {arXiv preprint},
  note = {arXiv:2409.00290},
  year = {2024}
}

@misc{DimitroglouRizellGolovko2025NonRegular,
      title={{Non-regular Lagrangian concordances between Lagrangian fillable Legendrian knots}}, 
      author={Dimitroglou Rizell, Georgios and Golovko, Roman},
      journal = {arXiv preprint},
      note = {arXiv:2509.13594},
      year = {2025}
}

@book {grauert1979steinspaces,
    AUTHOR = {Grauert, Hans and Remmert, Reinhold},
     TITLE = {Theory of {S}tein spaces},
    SERIES = {Classics in Mathematics},
      NOTE = {Translated from the German by Alan Huckleberry,
              Reprint of the 1979 translation},
 PUBLISHER = {Springer-Verlag, Berlin},
      YEAR = {1979},
     PAGES = {xxii+255},
      ISBN = {3-540-00373-8},
   MRCLASS = {32E10 (32-01 32Q28)},
  MRNUMBER = {2029201},
       DOI = {10.1007/978-3-642-18921-0},
       URL = {https://doi.org/10.1007/978-3-642-18921-0},
}

@misc{breen2024regularlysliceimpliesoncestably,
      title={Regularly slice implies once-stably decomposably slice}, 
      author={Joseph Breen},
      year={2024},
      eprint={2410.21031},
      archivePrefix={arXiv},
      primaryClass={math.SG},
      url={https://arxiv.org/abs/2410.21031},
    note={\href{https://arxiv.org/abs/2410.21031}{\tt arXiv:2410.21031}},
}

@article{EliashbergGanatraLazarev2018Regular,
  AUTHOR = {Eliashberg, Yakov and Ganatra, Sheel and Lazarev, Oleg},
     TITLE = {Flexible {L}agrangians},
   JOURNAL = {Int. Math. Res. Not. IMRN},
  FJOURNAL = {International Mathematics Research Notices. IMRN},
      YEAR = {2020},
    NUMBER = {8},
     PAGES = {2408--2435},
      ISSN = {1073-7928,1687-0247},
   MRCLASS = {53D12},
  MRNUMBER = {4090744},
MRREVIEWER = {Sheila\ Sandon},
       DOI = {10.1093/imrn/rny078},
       URL = {https://doi.org/10.1093/imrn/rny078},
}

@misc{brw1,
title = {Regular {L}agrangians in {L}efschetz Fibrations},
author = {Joseph Breen and Agniva Roy and Luya Wang},
year = 2025,
eprint = {2510.12170},
archiveprefix = {arxiv},
primaryclass = {math.SG},
note = {\href{https://arxiv.org/abs/2510.12170}{\tt arXiv:2510.12170}}
}

@misc{spinal_I,
	title        = {On symplectic fillings of spinal open book decompositions {I}: Geometric constructions},
	author       = {Samuel Lisi and Jeremy Van Horn-Morris and Chris Wendl},
	year         = 2018,
	eprint       = {1810.12017},
	archiveprefix = {arXiv},
	primaryclass = {math.SG},
note = {\href{https://arxiv.org/abs/1810.12017}{\tt arXiv:1810.12017}}
}

@article {gompf_topological_symplectic,
    AUTHOR = {Gompf, Robert E.},
     TITLE = {Toward a topological characterization of symplectic manifolds},
   JOURNAL = {J. Symplectic Geom.},
  FJOURNAL = {The Journal of Symplectic Geometry},
    VOLUME = {2},
      YEAR = {2004},
    NUMBER = {2},
     PAGES = {177--206},
      ISSN = {1527-5256,1540-2347},
   MRCLASS = {53D35 (57R17)},
  MRNUMBER = {2108373},
MRREVIEWER = {Ivan\ Smith},
       DOI = {10.4310/jsg.2004.v2.n2.a1},
       URL = {https://doi.org/10.4310/jsg.2004.v2.n2.a1},
}

@misc{spinal_II,
	title        = {On symplectic fillings of spinal open book decompositions {I}{I}: Holomorphic curves and classification},
	author       = {Samuel Lisi and Jeremy Van Horn-Morris and Chris Wendl},
	year         = 2020,
	eprint       = {2010.16330},
	archiveprefix = {arXiv},
	primaryclass = {math.SG},
note = {\href{https://arxiv.org/abs/2010.16330}{\tt arXiv:2010.16330}}
}

@article{baykurvhm,
  title={Families of contact 3-manifolds with arbitrarily large Stein fillings},
  author={Baykur, R. {\.I}nan\c{c} and Van Horn-Morris, Jeremy and Lisi, Samuel and Wendl, Chris},
  journal={Journal of Differential Geometry},
  volume={101},
  number={3},
  pages={423--465},
  year={2015},
  publisher={Lehigh University}
}

@article{baykur2016multisections,
  title={Multisections of Lefschetz fibrations and topology of symplectic 4--manifolds},
  author={Baykur, R. {\.I}nan\c{c} and Hayano, Kenta},
  journal={Geometry \& Topology},
  volume={20},
  number={4},
  pages={2335--2395},
  year={2016},
  publisher={Mathematical Sciences Publishers}
}

@incollection {eliashberg_2-knots,
    AUTHOR = {Eliashberg, Y.},
     TITLE = {Topology of {$2$}-knots in {${\bf R}^4$} and symplectic
              geometry},
 BOOKTITLE = {The {F}loer memorial volume},
    SERIES = {Progr. Math.},
    VOLUME = {133},
     PAGES = {335--353},
 PUBLISHER = {Birkh\"auser, Basel},
      YEAR = {1995},
      ISBN = {3-7643-5044-X},
   MRCLASS = {57Q45 (32D10 53C23 58F05)},
  MRNUMBER = {1362834},
MRREVIEWER = {Hansj\"org\ Geiges},
}

@article{hayden2021quasipositive,
  title={Quasipositive links and Stein surfaces},
  author={Hayden, Kyle},
  journal={Geometry \& Topology},
  volume={25},
  number={3},
  pages={1441--1477},
  year={2021},
  publisher={Mathematical Sciences Publishers}
}

@article{boileau-orevkov,
  title={Quasi-positivit{\'e} d'une courbe analytique dans une boule pseudo-convexe},
  author={Boileau, Michel and Orevkov, Stepan},
  journal={Comptes Rendus de l'Acad{\'e}mie des Sciences-Series I-Mathematics},
  volume={332},
  number={9},
  pages={825--830},
  year={2001},
  publisher={Elsevier}
}

@article {Wendl_strongly,
    AUTHOR = {Wendl, Chris},
     TITLE = {Strongly fillable contact manifolds and {$J$}-holomorphic
              foliations},
   JOURNAL = {Duke Math. J.},
  FJOURNAL = {Duke Mathematical Journal},
    VOLUME = {151},
      YEAR = {2010},
    NUMBER = {3},
     PAGES = {337--384}
}

@article {Siefring_asymptotics,
    AUTHOR = {Siefring, Richard},
     TITLE = {Relative asymptotic behavior of pseudoholomorphic
              half-cylinders},
   JOURNAL = {Comm. Pure Appl. Math.},
  FJOURNAL = {Communications on Pure and Applied Mathematics},
    VOLUME = {61},
      YEAR = {2008},
    NUMBER = {12},
     PAGES = {1631--1684}
}

@article {rudolph2004algebraic,
    AUTHOR = {Rudolph, Lee},
     TITLE = {Algebraic functions and closed braids},
   JOURNAL = {Topology},
  FJOURNAL = {Topology. An International Journal of Mathematics},
    VOLUME = {22},
      YEAR = {1983},
    NUMBER = {2},
     PAGES = {191--202}
}

@misc{etnyre_obd,
  title={Lectures on open book decompositions and contact structures},
  author={Etnyre, John B},
  year={2004},
    note = {\href{https://arxiv.org/abs/math/0409402}{\tt arXiv:0409402}}
}

@misc{wendl2016lectures,
      title={Lectures on Symplectic Field Theory}, 
      author={Chris Wendl},
      year={2016},
      eprint={1612.01009},
      archivePrefix={arXiv},
      primaryClass={math.SG},
note = {\href{https://arxiv.org/abs/1612.01009}{\tt arXiv:1612.01009}}
}

@article{p-vhm,
  title={Planar open books, monodromy factorizations and symplectic fillings},
  author={Plamenevskaya, Olga and Horn-Morris, Van and others},
  journal={Geometry \& Topology},
  volume={14},
  number={4},
  pages={2077--2101},
  year={2010},
  publisher={Mathematical Sciences Publishers}
}

@article {Akbulut_Ozbagci,
    AUTHOR = {Akbulut, Selman and Ozbagci, Burak},
     TITLE = {On the topology of compact {S}tein surfaces},
   JOURNAL = {Int. Math. Res. Not.},
  FJOURNAL = {International Mathematics Research Notices},
      YEAR = {2002},
    NUMBER = {15},
     PAGES = {769--782}
}

@article {simha,
    AUTHOR = {Simha, R. R.},
     TITLE = {On the complement of a curve on a {S}tein space of dimension
              two},
   JOURNAL = {Math. Z.},
  FJOURNAL = {Mathematische Zeitschrift},
    VOLUME = {82},
      YEAR = {1963},
     PAGES = {63--66}
}

@article {Baker_Etnyre_VHM_rational,
    AUTHOR = {Baker, Kenneth L. and Etnyre, John B. and Van Horn-Morris,
              Jeremy},
     TITLE = {Cabling, contact structures and mapping class monoids},
   JOURNAL = {J. Differential Geom.},
  FJOURNAL = {Journal of Differential Geometry},
    VOLUME = {90},
      YEAR = {2012},
    NUMBER = {1},
     PAGES = {1--80},
      ISSN = {0022-040X,1945-743X},
   MRCLASS = {53D35 (57R15)},
  MRNUMBER = {2891477},
MRREVIEWER = {Daniel\ V.\ Mathews},
       URL = {http://projecteuclid.org/euclid.jdg/1335209489},
}

@misc{spinal_exotic,
      title={Spinal open books and symplectic fillings with exotic fibers}, 
      author={Hyunki Min and Agniva Roy and Luya Wang},
      year={2024},
      eprint={2410.10697},
      archivePrefix={arXiv},
      primaryClass={math.GT},
      url={https://arxiv.org/abs/2410.10697}, 
}

@article {Casals_Murphy_affine,
    AUTHOR = {Casals, Roger and Murphy, Emmy},
     TITLE = {Legendrian fronts for affine varieties},
   JOURNAL = {Duke Math. J.},
  FJOURNAL = {Duke Mathematical Journal},
    VOLUME = {168},
      YEAR = {2019},
    NUMBER = {2},
     PAGES = {225--323},
      ISSN = {0012-7094,1547-7398},
   MRCLASS = {57R17 (53D10 53D12)},
  MRNUMBER = {3909897},
MRREVIEWER = {V.\ D.\ Sedykh},
       DOI = {10.1215/00127094-2018-0055},
       URL = {https://doi.org/10.1215/00127094-2018-0055},
}

@article{avdek2013surgery,
  author = {Avdek, Russell},
  title = {Contact surgery and supporting open books},
  journal = {Algebr. Geom. Topol.},
  volume = {13},
  number = {3},
  pages = {1613--1660},
  year = {2013},
  doi = {10.2140/agt.2013.13.1613}
}

@book{gompf1999manifolds,
  author    = {Gompf, Robert E. and Stipsicz, Andr{\'a}s I.},
  title     = {{4-Manifolds and Kirby Calculus}},
  series    = {Graduate Studies in Mathematics},
  number    = {20},
  publisher = {American Mathematical Society},
  address   = {Providence, RI},
  year      = {1999},
  pages     = {558},
  isbn      = {978-1-4704-7455-3},
  url       = {https://bookstore.ams.org/gsm-20}
}

@article {Conway_Etnyre_Tosun,
    AUTHOR = {Conway, James and Etnyre, John B. and Tosun, B\"ulent},
     TITLE = {Symplectic fillings, contact surgeries, and {L}agrangian
              disks},
   JOURNAL = {Int. Math. Res. Not. IMRN},
  FJOURNAL = {International Mathematics Research Notices. IMRN},
      YEAR = {2021},
    NUMBER = {8},
     PAGES = {6020--6050},
      ISSN = {1073-7928,1687-0247},
   MRCLASS = {57K33 (57K10)},
  MRNUMBER = {4251270},
MRREVIEWER = {Burak\ Ozbagci},
       DOI = {10.1093/imrn/rny291},
       URL = {https://doi.org/10.1093/imrn/rny291},
}

@inproceedings {Giroux2002GeometrieDC,
    AUTHOR = {Giroux, Emmanuel},
     TITLE = {G\'{e}om\'{e}trie de contact: de la dimension trois vers les
              dimensions sup\'{e}rieures},
 BOOKTITLE = {Proceedings of the {I}nternational {C}ongress of
              {M}athematicians, {V}ol. {II} ({B}eijing, 2002)},
     PAGES = {405--414},
 PUBLISHER = {Higher Ed. Press, Beijing},
      YEAR = {2002},
   MRCLASS = {53D35 (57M50 57R17)},
  MRNUMBER = {1957051},
MRREVIEWER = {Hansj\"{o}rg Geiges},
}

@book {Cieliebak_Eliashberg,
    AUTHOR = {Cieliebak, Kai and Eliashberg, Yakov},
     TITLE = {From {S}tein to {W}einstein and back},
    SERIES = {American Mathematical Society Colloquium Publications},
    VOLUME = {59},
      NOTE = {Symplectic geometry of affine complex manifolds},
 PUBLISHER = {American Mathematical Society, Providence, RI},
      YEAR = {2012},
     PAGES = {xii+364},
      ISBN = {978-0-8218-8533-8},
   MRCLASS = {53-02 (32E10 32Q65 53C15 53D05 58E05)},
  MRNUMBER = {3012475},
MRREVIEWER = {Chris\ M.\ Wendl},
       DOI = {10.1090/coll/059},
       URL = {https://doi.org/10.1090/coll/059},
}

@article {Wendl_SHS,
    AUTHOR = {Wendl, Chris},
     TITLE = {Open book decompositions and stable {H}amiltonian structures},
   JOURNAL = {Expo. Math.},
  FJOURNAL = {Expositiones Mathematicae},
    VOLUME = {28},
      YEAR = {2010},
    NUMBER = {2},
     PAGES = {187--199},
      ISSN = {0723-0869,1878-0792},
   MRCLASS = {53D35 (32Q65 57R17)},
  MRNUMBER = {2671115},
MRREVIEWER = {Michael\ L.\ Hutchings},
       DOI = {10.1016/j.exmath.2009.09.001},
       URL = {https://doi.org/10.1016/j.exmath.2009.09.001},
}

@article {wendl_hierarchy,
    AUTHOR = {Wendl, Chris},
     TITLE = {A hierarchy of local symplectic filling obstructions for
              contact 3-manifolds},
   JOURNAL = {Duke Math. J.},
  FJOURNAL = {Duke Mathematical Journal},
    VOLUME = {162},
      YEAR = {2013},
    NUMBER = {12},
     PAGES = {2197--2283},
      ISSN = {0012-7094,1547-7398},
   MRCLASS = {57R17 (32Q65 53D10 53D42)},
  MRNUMBER = {3102479},
MRREVIEWER = {Youlin\ Li},
       DOI = {10.1215/00127094-2348333},
       URL = {https://doi.org/10.1215/00127094-2348333},
}

@article {Banyaga,
    AUTHOR = {Banyaga, Augustin},
     TITLE = {Sur la structure du groupe des diff\'eomorphismes qui
              pr\'eservent une forme symplectique},
   JOURNAL = {Comment. Math. Helv.},
  FJOURNAL = {Commentarii Mathematici Helvetici},
    VOLUME = {53},
      YEAR = {1978},
    NUMBER = {2},
     PAGES = {174--227},
      ISSN = {0010-2571,1420-8946},
   MRCLASS = {58D05 (57R50 58F06 81E99)},
  MRNUMBER = {490874},
MRREVIEWER = {D.\ B. A. Epstein},
       DOI = {10.1007/BF02566074},
       URL = {https://doi.org/10.1007/BF02566074},
}

@book {mcduff_salamon_intro,
    AUTHOR = {McDuff, Dusa and Salamon, Dietmar},
     TITLE = {Introduction to symplectic topology},
    SERIES = {Oxford Graduate Texts in Mathematics},
   EDITION = {Third},
 PUBLISHER = {Oxford University Press, Oxford},
      YEAR = {2017},
     PAGES = {xi+623},
      ISBN = {978-0-19-879490-5; 978-0-19-879489-9},
   MRCLASS = {53D35 (53D40 57R17 57R57 57R58)},
  MRNUMBER = {3674984},
MRREVIEWER = {Hansj\"org\ Geiges},
       DOI = {10.1093/oso/9780198794899.001.0001},
       URL = {https://doi.org/10.1093/oso/9780198794899.001.0001},
}

@article {Cao_Gallup_Hayden_Sabloff,
    AUTHOR = {Cao, Chang and Gallup, Nathaniel and Hayden, Kyle and Sabloff,
              Joshua M.},
     TITLE = {Topologically distinct {L}agrangian and symplectic fillings},
   JOURNAL = {Math. Res. Lett.},
  FJOURNAL = {Mathematical Research Letters},
    VOLUME = {21},
      YEAR = {2014},
    NUMBER = {1},
     PAGES = {85--99},
      ISSN = {1073-2780,1945-001X},
   MRCLASS = {57R17 (53D05 53D12)},
  MRNUMBER = {3247041},
MRREVIEWER = {Georgios\ Dimitroglou Rizell},
       DOI = {10.4310/MRL.2014.v21.n1.a7},
       URL = {https://doi.org/10.4310/MRL.2014.v21.n1.a7},
}

@article{GirouxPardon2017Lefschetz,
  author = {Giroux, Emmanuel and Pardon, John},
  title = {Existence of {Lefschetz} fibrations on {Stein} and {Weinstein} domains},
  journal = {Geom. Topol.},
  volume = {21},
  pages = {963--997},
  year = {2017}
}

@article{Akbulut2001Lefschetz,
  author  = {Akbulut, Selman and Ozbagci, Burak},
  title   = {Lefschetz fibrations on compact {S}tein surfaces},
  journal = {Geom. Topol.},
  volume  = {5},
  pages   = {319--334},
  year    = {2001}
}

@article{LoiPiergallini2001Stein,
  author = {Loi, Andrea and Piergallini, Riccardo},
  title = {Compact {Stein} surfaces with boundary as branched covers of {$B^4$}},
  journal = {Invent. Math.},
  volume = {143},
  number = {2},
  pages = {325--348},
  year = {2001}
}

@article {Thurston,
    AUTHOR = {Thurston, W. P.},
     TITLE = {Some simple examples of symplectic manifolds},
   JOURNAL = {Proc. Amer. Math. Soc.},
  FJOURNAL = {Proceedings of the American Mathematical Society},
    VOLUME = {55},
      YEAR = {1976},
    NUMBER = {2},
     PAGES = {467--468},
      ISSN = {0002-9939,1088-6826},
   MRCLASS = {57D15},
  MRNUMBER = {402764},
MRREVIEWER = {Edgar\ H.\ Brown, Jr.},
       DOI = {10.2307/2041749},
       URL = {https://doi.org/10.2307/2041749},
}

@misc{hedden_quasipositive,
      title={Quasipositive braids in open books as boundaries of complex curves}, 
      author={{R. {\.I}nan\c{c} Baykur} and {John Etnyre} and {Kyle Hayden} and {Matthew Hedden} and {Keiko Kawamuro} and {Jeremy Van Horn-Morris}},
      year={},
      eprint={},
      archivePrefix={},
      primaryClass={},
      note={in preparation}, 
}

@book{K3,
  title = {K3: A New Problem List in Low-Dimensional Topology},
  editor = {Baykur, R. {\.I}nan\c{c} and Kirby, Robion C. and Ruberman, Daniel},
  series = {Mathematical Surveys and Monographs},
  number = {295},
  publisher = {American Mathematical Society},
  address = {Providence, RI},
  year = {2026},
  note = {©2026},
  pages = {xi+430},
  isbn = {978-1-4704-8433-0},
  mrnumber = {5077826}
}

@incollection{Eliashberg90b,
  address = {Cambridge},
  author = {Eliashberg, Yakov},
  booktitle = {Geometry of low-dimensional manifolds, 2 (Durham, 1989)},
  mrclass = {53C23 (32D10 53C15 58E99)},
  mrnumber = {MR1171908 (93g:53060)},
  mrreviewer = {Nikolai K. Smolentsev},
  pages = {45--67},
  publisher = {Cambridge Univ. Press},
  series = {London Math. Soc. Lecture Note Ser.},
  title = {Filling by holomorphic discs and its applications},
  volume = {151},
  year = {1990},
}

@incollection{EliashbergFraser98,
  address = {Providence, RI},
  author = {Eliashberg, Yakov and Fraser, Maia},
  booktitle = {Geometry, topology, and dynamics (Montreal, PQ, 1995)},
  mrclass = {57M25 (57N37)},
  mrnumber = {MR1619122 (99f:57007)},
  mrreviewer = {Serge L. Tabachnikov},
  pages = {17--51},
  publisher = {Amer. Math. Soc.},
  series = {CRM Proc. Lecture Notes},
  title = {Classification of topologically trivial {L}egendrian knots},
  volume = {15},
  year = {1998},
}

@article {Eliashberg_Polterovich,
    AUTHOR = {Eliashberg, Y. and Polterovich, L.},
     TITLE = {Local {L}agrangian {$2$}-knots are trivial},
   JOURNAL = {Ann. of Math. (2)},
  FJOURNAL = {Annals of Mathematics. Second Series},
    VOLUME = {144},
      YEAR = {1996},
    NUMBER = {1},
     PAGES = {61--76},
      ISSN = {0003-486X,1939-8980},
   MRCLASS = {58F05 (57Q45 57R40)},
  MRNUMBER = {1405943},
MRREVIEWER = {Serge\ L.\ Tabachnikov},
       DOI = {10.2307/2118583},
       URL = {https://doi.org/10.2307/2118583},
}

@article {Cote_DR,
    AUTHOR = {C\^ot\'e, Laurent and Dimitroglou Rizell, Georgios},
     TITLE = {Symplectic rigidity of fibers in cotangent bundles of open
              {R}iemann surfaces},
   JOURNAL = {Math. Ann.},
  FJOURNAL = {Mathematische Annalen},
    VOLUME = {386},
      YEAR = {2023},
    NUMBER = {3-4},
     PAGES = {1717--1751},
      ISSN = {0025-5831,1432-1807},
   MRCLASS = {53D12},
  MRNUMBER = {4612405},
MRREVIEWER = {Jiabin\ Yin},
       DOI = {10.1007/s00208-022-02447-z},
       URL = {https://doi.org/10.1007/s00208-022-02447-z},
}

@article {Ekholm_Smith,
    AUTHOR = {Ekholm, Tobias and Smith, Ivan},
     TITLE = {Nearby {L}agrangian fibers and {W}hitney sphere links},
   JOURNAL = {Compos. Math.},
  FJOURNAL = {Compositio Mathematica},
    VOLUME = {154},
      YEAR = {2018},
    NUMBER = {4},
     PAGES = {685--718},
      ISSN = {0010-437X,1570-5846},
   MRCLASS = {53D35 (53D42 57R17)},
  MRNUMBER = {3778191},
MRREVIEWER = {Cecilia\ Karlsson},
       DOI = {10.1112/S0010437X17007692},
       URL = {https://doi.org/10.1112/S0010437X17007692},
}

@misc{asplund2025homotopyrigiditynearbylagrangian,
      title={Homotopy rigidity of nearby Lagrangian cocores}, 
      author={Johan Asplund and Yash Deshmukh and Alex Pieloch},
      year={2025},
      eprint={2511.09548},
      archivePrefix={arXiv},
      primaryClass={math.SG},
      url={https://arxiv.org/abs/2511.09548},
      note = {\href{https://arxiv.org/abs/2511.09548}{\tt arXiv:2511.09548}},
}

@article{Abouzaid2012Nearby,
  author  = {Abouzaid, Mohammed},
  title   = {Nearby Lagrangians with Vanishing Maslov Class Are Homotopy Equivalent},
  journal = {Inventiones Mathematicae},
  volume  = {189},
  pages   = {251--313},
  year    = {2012},
  doi     = {10.1007/s00222-011-0365-0}
}

@article{Hind2012Stein,
  author  = {Hind, Richard},
  title   = {Lagrangian Unknottedness in {Stein} Surfaces},
  journal = {Asian Journal of Mathematics},
  volume  = {16},
  number  = {1},
  pages   = {1--36},
  year    = {2012}
}

@article{Kragh2013Nearby,
  author  = {Kragh, Thomas},
  title   = {Parametrized Ring-Spectra and the Nearby Lagrangian Conjecture},
  journal = {Geometry \& Topology},
  volume  = {17},
  pages   = {639--731},
  year    = {2013},
  doi     = {10.2140/gt.2013.17.639},
  note    = {With an appendix by Mohammed Abouzaid}
}

@article{DimitroglouRizellGoodmanIvrii2016,
  author  = {Dimitroglou Rizell, Georgios and Goodman, Elizabeth and Ivrii, Alexander},
  title   = {Lagrangian Isotopy of Tori in {$S^2 \times S^2$} and {${\mathbb C}P^2$}},
  journal = {Geometric and Functional Analysis},
  volume  = {26},
  pages   = {1297--1358},
  year    = {2016},
  doi     = {10.1007/s00039-016-0388-1}
}

@article{AbouzaidKragh2018Nearby,
  author  = {Abouzaid, Mohammed and Kragh, Thomas},
  title   = {Simple Homotopy Equivalence of Nearby Lagrangians},
  journal = {Acta Mathematica},
  volume  = {220},
  number  = {2},
  pages   = {207--237},
  year    = {2018},
  doi     = {10.4310/ACTA.2018.v220.n2.a1}
}
\end{document}